\documentclass{amsart}
\usepackage{bbm}
\usepackage{color}
\usepackage{amscd}
\usepackage{nicefrac}
\usepackage{graphicx}
\usepackage{hyperref}
\usepackage{amssymb}
\usepackage{euscript}
\usepackage{enumitem}
\usepackage[cmtip,all]{xy}
\usepackage[export]{adjustbox}

\numberwithin{equation}{subsection}

\allowdisplaybreaks

\theoremstyle{plain}
\newtheorem{thm}{Theorem}[subsection]

\newtheorem{theorem}[thm]{Theorem}

\newtheorem{cor}[thm]{Corollary}
\newtheorem{corollary}[thm]{Corollary}

\newtheorem{lemma}[thm]{Lemma}

\newtheorem{prop}[thm]{Proposition}
\newtheorem{definition}[thm]{Definition}

\newtheorem{remark}[thm]{Remark}

\newtheorem{proposition}[thm]{Proposition}
\newtheorem{example}[thm]{Example}
\newtheorem{non-example}[thm]{Non-example}

\newtheorem{lem}[thm]{Lemma}

\newtheorem*{claim*}{Claim} 
\newtheorem*{lemma*}{Lemma}
\newtheorem*{theorem*}{Theorem}
\newtheorem*{conjecture*}{Conjecture}

\newcommand{\bC}{{\mathbb C}}

\newcommand{\bN}{{\mathbb N}}

\newcommand{\bQ}{{\mathbb Q}}
\newcommand{\bR}{{\mathbb R}}

\newcommand{\bZ}{{\mathbb Z}}

\newcommand{\scrE}{\EuScript E}
\newcommand{\scrF}{\EuScript F}

\newcommand{\frakA}{\mathfrak{A}}
\newcommand{\frakB}{\mathfrak{B}}
\newcommand{\frakC}{\mathfrak{C}}
\newcommand{\frakD}{\mathfrak{D}}
\newcommand{\frakE}{\mathfrak{E}}
\newcommand{\frakF}{\mathfrak{F}}
\newcommand{\frakG}{\mathfrak{G}}
\newcommand{\frakH}{\mathfrak{H}}

\newcommand{\frakM}{\mathfrak{M}}
\newcommand{\frakO}{\mathfrak{O}}

\newcommand{\frakS}{\mathfrak{S}}
\newcommand{\frakT}{\mathfrak{T}}
\newcommand{\frakY}{\mathfrak{Y}}

\newcommand{\half}{{\textstyle\frac{1}{2}}}

\newcommand{\iso}{\cong}
\newcommand{\htp}{\simeq}
\newcommand{\smooth}{C^\infty}

\newcommand{\AT}{\frakA\frakT}
\newcommand{\AS}{\frakA\frakS}
\newcommand{\MC}{\frakM\frakC}
\newcommand{\EX}{\frakE\frakY}
\newcommand{\frakAX}{\frakA\frakY}
\newcommand{\FM}{\frakF\frakM}

\newcommand{\LP}{P}

\begin{document}
\title[]{Symplectic cohomology relative \\ to a smooth anticanonical divisor}
\author{Daniel Pomerleano, Paul Seidel}

\maketitle

\begin{abstract}
For a monotone symplectic manifold and a smooth anticanonical divisor, there is a formal deformation of the symplectic cohomology of the divisor complement, defined by allowing Floer cylinders to intersect the divisor. We compute this deformed symplectic cohomology, in terms of the ordinary cohomology of the manifold and divisor; and also describe some additional structures that it carries.
\end{abstract}

\section{Introduction}

\subsection{Background}
Let $M$ be a closed monotone (Fano) symplectic manifold, and $D \subset M$ a symplectic divisor which is anticanonical (Poincar{\'e} dual to the first Chern class of $M$). Borman, Sheridan and Varolg\"une\c{s} \cite{borman-sheridan-varolgunes21} have studied the relation between the symplectic cohomology of the complement, $\mathit{SH}^*(M \setminus D)$, and the ordinary cohomology $H^*(M)$. In their setup, $D$ has normal crossings (they also allow the components of $D$ to have multiplicities in $(0,1]$; to simplify the discussion, we consider only the case where all the multiplicities are $1$). They introduced a filtered complex whose cohomology is $H^*(M)[q^{\pm 1}]$, with $q$ a formal variable of degree $2$; and such that each graded piece associated to the filtration are quasi-isomorphic to the standard complex underlying symplectic cohomology \cite[Theorem B]{borman-sheridan-varolgunes21}. As a consequence, they obtained a spectral sequence \cite[Theorem C]{borman-sheridan-varolgunes21}
\begin{equation} \label{eq:bsv}
\mathit{SH}^*(M \setminus D)[q^{\pm 1}] \Longrightarrow H^*(M)[q^{\pm 1}].
\end{equation}
This is meaningful in terms of mirror symmetry: there, $\mathit{SH}^*(M \setminus D)$ describes the cohomology of polyvector fields on the mirror $X$ of $M \setminus D$ (see e.g.\ \cite{pascaleff17,pomerleano21}); and the deformation to $H^*(M)$ corresponds to turning on the superpotential $W$ which produces the Landau-Ginzburg mirror $W: X \rightarrow \bC$ of $M$ relative to $D$; see e.g.\ \cite{auroux07} for an exposition. The preprints \cite{el-alami-sheridan24a, el-alami-sheridan24b}, posted simultaneously with this one, continue the approach from \cite{borman-sheridan-varolgunes21}; we will comment on the relation at various points later on (Remarks \ref{th:el-alami-sheridan}, \ref{th:completed}, \ref{th:linfty-remark}, \ref{th:other-mc}).

\subsection{Results}
We only look at the much less general situation where $D$ is a smooth anticanonical divisor. We define directly a $q$-deformation of the (telescope or homotopy direct limit) chain complex underlying symplectic cohomology. The deformation involves Floer-type cylinders that intersect $D$, while their limits are still one-periodic Hamiltonian orbits in $M \setminus D$. The cohomology of the deformed complex will be denoted by $\mathit{SH}^*_q(M,D)$. By construction, it comes with a $q$-filtration spectral sequence
\begin{equation} \label{eq:q-spectral-sequence}
\mathit{SH}^*(M \setminus D)[q] \Longrightarrow \mathit{SH}^*_q(M,D).
\end{equation}
Note that in this context, symplectic cohomology is bounded below (in the grading); hence, in any given degree, only finitely many powers of $q$ contribute to \eqref{eq:q-spectral-sequence}.
%

\begin{theorem} \label{th:main}
There is a canonical isomorphism
\begin{equation} \label{eq:main}
H^*(M)[q] \oplus \bigoplus_{w \geq 1} H^*(D)z^w\stackrel{\iso}{\longrightarrow} \mathit{SH}^*_q(M,D),
\end{equation}
where $z^w$ are formal symbols of degree $0$ (in spite of the notation, there is no $z$-linearity in this statement). The restriction of that map to $H^*(M)[q]$ is $q$-linear.
\end{theorem}

The construction of the map \eqref{eq:main} uses pseudo-holomorphic thimbles, following a strategy already deployed in a series of papers: going back to \cite{piunikhin-salamon-schwarz95} for the $H^*(M)$-component, and \cite{ganatra-pomerleano21, tonkonog19, ganatra-pomerleano20, pomerleano21} for the $H^*(D)$-components. 

Let's extend the $\bZ[q]$-module structure from $H^*(M)[q]$ to the entire domain of \eqref{eq:main}, in the unique way which is compatible with that isomorphism. We will not describe this extension explicitly (see Remark \ref{th:further} below); but it's clear for degree reasons that any element of $H^*(D)z^w$ must be mapped to $H^*(M)[q]$ by a sufficiently high power of $q$. As an immediate consequence, one has:

\begin{corollary} \label{th:invert-q}
The first part of the map from \eqref{eq:main} induces a $q$-linear isomorphism
\begin{equation}
H^*(M)[q^{\pm 1}] \stackrel{\iso}{\longrightarrow} \bZ[q^{\pm 1}] \otimes_{\bZ[q]} \mathit{SH}^*_q(M,D).
\end{equation}
\end{corollary}

Together with \eqref{eq:q-spectral-sequence}, this recovers \eqref{eq:bsv} (even though it's by no means clear that our spectral sequence is the same as that from \cite{borman-sheridan-varolgunes21}).

\begin{remark} \label{th:el-alami-sheridan}
A theorem essentially equivalent to Corollary \ref{th:invert-q} is part of the results in \cite{el-alami-sheridan24b}. In fact, their geometric setup is substantially more general, since it allows $D$ to have normal crossings. The approach in \cite{el-alami-sheridan24b} is significantly different from ours: they construct a Maurer-Cartan element in the chain complex underlying symplectic cohomology, by a combination of geometric and indirect algebraic arguments. Then, they appeal to \cite{borman-sheridan-varolgunes21} to show that after inverting $q$, the differential deformed by this Maurer-Cartan element computes the cohomology of $M$. In particular, there is no analogue of Theorem \ref{th:main} in \cite{el-alami-sheridan24b}.
\end{remark}

Deformed symplectic cohomology has an $S^1$-equivariant analogue, involving another formal variable $u$ of degree $2$, and which we therefore denote by $\mathit{SH}^*_{u,q}(M,D)$. It is unproblematic to lift the thimble map \eqref{eq:main} to an equivariant one,
\begin{equation} \label{eq:equivariant-main}
H^*(M)[u,q] \oplus \bigoplus_{w \geq 1} H^*(D)[u]z^w \longrightarrow \mathit{SH}^*_{u,q}(M,D).
\end{equation}
The equivariant version of Theorem \ref{th:main} (which actually follows straightforwardly from the original statement) says that:

\begin{corollary} \label{th:equivariant-main}
The map \eqref{eq:equivariant-main} is an isomorphism.
\end{corollary}

The map \eqref{eq:equivariant-main} is $u$-linear by definition. On the first summand of its domain, it is also $q$-linear. For slightly more complicated reasons than before (involving the action filtration rather than just degrees), the following still holds:

\begin{lemma} \label{th:equivariant-q-torsion}
Equip the entire left hand side of \eqref{eq:equivariant-main} with the $\bZ[q]$-module structure that corresponds to the existing one on the right hand side. Then, any element will be mapped to the subspace $H^*(M)[u,q]$ by a sufficiently high power of $q$.
\end{lemma}

As a consequence, we get an equivariant version of Corollary \ref{th:invert-q}:

\begin{corollary} \label{th:invert-q-2}
The first part of \eqref{eq:equivariant-main} induces a canonical $(u,q)$-linear isomorphism
\begin{equation} \label{eq:inverted-equivariant-main}
H^*(M)[u,q^{\pm 1}] \stackrel{\iso}{\longrightarrow} \bZ[q^{\pm 1}] \otimes_{\bZ[q]} \mathit{SH}^*_{u,q}(M,D).
\end{equation}
\end{corollary}

\begin{remark} \label{th:completed}
One might hope to prove Corollary \ref{th:invert-q-2} using just Corollary \ref{th:invert-q} and a $u$-filtration argument (or using \cite{el-alami-sheridan24b} as a starting point). However, it seems that such an approach only yields a weaker version. Namely, suppose one starts with the chain complex of $\bZ[q^{\pm 1}]$-modules underlying $\bZ[q^{\pm 1}] \otimes_{\bZ[q]} \mathit{SH}^*_q(M,D)$, and constructs an equivariant version. In order for filtration arguments to work, that version has to be complete with respect to $u$, which means that it allows power series in (the degree $0$ expression) $u/q$. In terms of ordinary cohomology, this corresponds to using $H^*(M)[q^{\pm 1}][[u/q^{-1}]]$, which is a completion of the domain of \eqref{eq:inverted-equivariant-main}. However, while the isomorphism statement \eqref{eq:inverted-equivariant-main} implies a completed version, the converse implication is not necessarily true.
\end{remark}

The canonical connection on $\mathit{SH}^*_{u,q}(M,D)$ is a map
\begin{equation} \label{eq:connection-on-sh}
\begin{aligned}
& \nabla_{u\partial_q}: \mathit{SH}^*_{u,q}(M,D) \longrightarrow
\mathit{SH}^*_{u,q}(M,D), \\
& \nabla_{u\partial_q}(ux) = u\nabla_{u\partial_q}(x), \\
& \nabla_{u\partial_q}(qx) = q\nabla_{u\partial_q}(x) + ux.
\end{aligned}
\end{equation}
Here is a partial statement of compatibility of this operation with \eqref{eq:equivariant-main}:

\begin{proposition} \label{th:connection}
The connection and the first part of \eqref{eq:equivariant-main} fit into a commutative diagram
\begin{equation}
\xymatrix{
H^*(M)[u,q] \ar[rr] 
\ar[d]_-{uq\partial_q + [D] \ast_q }
&&
\mathit{SH}^*_{u,q}(M,D) \ar[d]^{q\nabla_{u\partial_q}}
\\
H^{*+2}(M)[u,q]
\ar[rr]
&&
\mathit{SH}^{*+2}_{u,q}(M,D)
}
\end{equation}
where $\ast_q$ is the small quantum product (and therefore, $uq\partial_q + [D] \ast_q$ is the quantum connection).
\end{proposition}

While no applications are given here, the main motivation for these results is their use in \cite{pomerleano-seidel23}. The non-equivariant version, in the form of Corollary \ref{th:invert-q}, plays a minor role there; it is used only to derive certain finite generation statements. In contrast, Corollary \ref{th:invert-q-2} and Proposition \ref{th:connection} are central to the purpose of \cite{pomerleano-seidel23}, which is to study the quantum connection using the wrapped Fukaya category of $M \setminus D$ and its $q$-deformation. (For that application, one needs Corollary \ref{th:invert-q-2} as stated here; the weaker $u$-completed version mentioned in Remark \ref{th:completed} would not be sufficient.) Note that the definition of deformed symplectic cohomology in \cite{pomerleano-seidel23} is technically different from, even though philosophically closely related to, the one here; so we also have to explain how to bridge that gap.

\begin{remark} \label{th:further}
We leave a number of questions unanswered, which concern the relation with the enumerative geometry of $(M,D)$. For instance, we have not fully determined the $q$-action on the domain of \eqref{eq:main} which makes that map an isomorphism; a first piece of that is addressed by Lemma \ref{th:replace-t-by-s}, but the general answer is expected to be much more complicated, presumably involving punctured Gromov-Witten invariants (for which see e.g.\cite{abramovichetal20, fan-wu21}). The same applies to the equivariant theory, both for the $q$-action and the connection \eqref{eq:connection-on-sh}. Finally, there are other operations on symplectic cohomology, such as the pair-of-pants product, whose $q$-deformed versions we have not considered at all.
%
\end{remark}

The structure of this paper is as follows. Section \ref{sec:sh} introduces the (quite elementary) geometric arguments which we use to control the behaviour of solutions to inhomogeneous pseudo-holomorphic map equations; it then proceeds to give our definition of deformed symplectic cohomology. Section \ref{sec:thimbles} introduces two kinds of equations living on the thimble, which together give rise to the map in \eqref{eq:main}. The next two sections are preliminaries, explaining a version of the action filtration (Section \ref{sec:action}) and certain Morse-theoretic constructions, concerning the real blowup of $M$ along $D$ and its boundary (Section \ref{sec:morse}). After that, Section \ref{sec:proof} contains our main argument, showing that \eqref{eq:main} is a quasi-isomorphism; this substantially uses results from \cite{ganatra-pomerleano20}. Section \ref{sec:equivariant} adds the equivariant versions of these arguments. Section \ref{sec:operations} is again preliminary work, preparing for the argument in Section \ref{sec:connection} which both defines the connection $\nabla_{u\partial_q}$ and proves its compatibility with the quantum connection (Proposition \ref{th:connection}). Finally, as mentioned above, Section \ref{sec:pullout} mediates between the framework here (surfaces with added marked points, at which the map intersects the divisor) and that in \cite{pomerleano-seidel23} (surfaces with added punctures, at which a Maurer-Cartan element associated to the divisor is inserted).

{\em Acknowledgments.} The authors would like to thank Nick Sheridan for explaining the ideas developed by him and Borman (which ultimately led to \cite{borman-sheridan-varolgunes21} as well as \cite{el-alami-sheridan24b}). The first author received partial funding from NSF grant DMS-2306204. 

\section{Symplectic cohomology and its deformation\label{sec:sh}}

We start by explaining the basic features of Floer theory in our setup: in particular, the behaviour of trajectories that intersect the divisor, and their Gromov limits. We then use those properties to define deformed symplectic cohomology. As always, the situation is that $M^{2n}$ is a closed symplectic manifold with $[\omega_M] = c_1(M)$, and $D \subset M$ a smooth symplectic hypersurface Poincar{\'e} dual to $c_1(M)$.

\subsection{Geometric basics\label{subsec:basics}}
In a tubular neighbourhood of $D$, there is a Hamiltonian $S^1$-action $(\rho_t)$ which fixes $D$ pointwise, and rotates the normal bundle. We fix such an $S^1$-action once and for all, with the convention that $(\rho_t)_{0 \leq t \leq 1}$ is one full anticlockwise rotation; and take $h$ to be its moment map, normalized so that $h|D = 0$ ($D$ is a local Morse-Bott minimum for $h$). 
\begin{itemize} \itemsep.5em
\item 
A function $H$ on $M$ {\em respects $D$} if its derivative vanishes at each point of $D$ (hence its Hamiltonian vector field, $\omega_M(\cdot,X) = dH$, is zero on $D$). It {\em has slope} $\sigma > 0$ if, in some neighbourhood of $D$, $H + \sigma h$ is constant (note the sign: increasing the slope means rotating the normal bundle clockwise).

\item A compatible almost complex structure $J$ on $M$ {\em respects} $D$ if that is an almost complex submanifold. It is {\em locally $S^1$-invariant} if $(\rho_t)$ preserves $J$, in some neighbourhood of $D$.
\end{itemize}

Take a time-dependent Hamiltonian $\bar{H} = (\bar{H}_t)_{t \in S^1}$, which is of slope $\sigma \in (\bQ \setminus \bZ)^{>0} = (\bQ \setminus \bZ) \cap \bR^{>0}$ for all $t$, and such that all one-periodic orbits $x: S^1 = \bR/\bZ \rightarrow M$ lying outside $D$ are nondegenerate (the constant orbits in $D$ are of Morse-Bott type, because $\sigma \notin \bZ$). We will only consider orbits that are nullhomologous in $M$ (the restriction to rational slopes, and that to nullhomologous orbits, are for technical simplicity; both could be lifted with more effort). An orbit $x$ lying outside $D$ has a well-defined action $A(x) \in \bR$. Namely, take a connected oriented compact surface $S$ with an oriented identification $\partial S \iso S^1$, and a map $y: S \rightarrow M$ with $y|S^1 = x$. One then sets
\begin{equation} \label{eq:action}
A(x) \stackrel{\mathrm{def}}{=} \int_{S^1} \bar{H}_t(x(t))\, \mathit{dt} - \int_S y^*\omega_M\; + y \cdot D.
\end{equation}
There is also a well-defined degree $\mathrm{deg}(x) \in \bZ$. One takes the same map $y$, and chooses a trivialization of the symplectic vector bundle $y^*TM$. The restriction of that trivialization to $\partial S$ gives rise to a Conley-Zehnder index for $x$ (our convention is that for small time-independent Hamiltonians and constant $y$, this agrees with the Morse index). One adds $2(y \cdot D)$ to the Conley-Zehnder index to define $\mathrm{deg}(x)$. 

Choose almost complex structures $\bar{J} = (\bar{J}_t)$ which are $t$-independent in a neighbourhood of $D$, and locally $S^1$-invariant. The standard Floer equation is
\begin{equation} \label{eq:floer}
\left\{ 
\begin{aligned}
& u = u(s,t): \bR \times S^1 \longrightarrow M, \\
& \textstyle \lim_{s \rightarrow \pm\infty} u = x_{\pm}, \\
& \partial_s u + \bar{J}_t(\partial_t u - \bar{X}_t) = 0.
\end{aligned}
\right.
\end{equation}

Rather than discussing this, we pass immediately to its $s$-dependent generalization, the continuation map equation. Fix slopes $\sigma_{\pm} \in (\bQ \setminus \bZ)^{>0}$; and correspondingly Hamiltonians $\bar{H}_{\pm}$ and almost complex structures $\bar{J}_{\pm}$, with the same properties as before. Suppose we have Hamiltonians $H = (H_{s,t})$ and almost complex structures $J = (J_{s,t})$, both respecting $D$, such that $H_{s,t} = \bar{H}_{\pm, t}$ and $J_{s,t} = \bar{J}_{\pm, t}$ for $\pm s \gg 0$. The continuation map equation replaces the last line in \eqref{eq:floer} with
\begin{equation} \label{eq:continuation}
\partial_s u + J_{s,t}(\partial_t u - X_{s,t}) = 0.
\end{equation}
Since we are assuming $X_{s,t}|D = 0$, pseudo-holomorphic maps inside $D$ appear as special solutions to that equation. We recall a couple of basic facts, for solutions with limits $x_{\pm}$ outside $D$: the energy identity 
\begin{equation} \label{eq:energy}
E(u) = \int_{\bR \times S^1} \|\partial_s u\|^2 = A(x_-) - A(x_+) + (u \cdot D) + \int_{\bR \times S^1} u^*(\partial_s H_{s,t}),
\end{equation}
with its consequence
\begin{equation} \label{eq:energy-2}
A(x_-) \geq A(x_+) - (u \cdot D) - \int_{-\infty}^{\infty} \max \{\partial_s H_{s,t} \, : \, (x,t) \in M \times S^1 \} \; \mathit{ds};
\end{equation}
and the index formula for the linearized operator,
\begin{equation} \label{eq:index}
\mathrm{index}(D_u) = \mathrm{deg}(x_-) - \mathrm{deg}(x_+) + 2(u \cdot D).
\end{equation}

%

\begin{lemma} \label{th:winding}
Let $u$ be a solution of \eqref{eq:continuation} not contained in $D$.

(i) Each point $z = (s,t) \in u^{-1}(D)$ is isolated, and the local intersection number $\mu_z(u)$ is positive.

(ii) Suppose that $x_- \in D$. Then there is some integer $\mu_{-\infty}(u) \geq -\lfloor\sigma_-\rfloor$ such that for $s \ll 0$, the loop $t \mapsto u(s,t)$ is disjoint from $D$, and has winding number $\mu_{-\infty}(u)$ around $D$. (Here, the winding number is computed in a small ball around $x_-$, hence is well-defined.)

(iii) Suppose that $x_+ \in D$. Then there is some integer $\mu_{+\infty}(u) \geq \lceil\sigma_+\rceil$ such that for $s \gg 0$, the loop $t \mapsto u(s,-t)$ is disjoint from $D$, and has winding number $\mu_{+\infty}(u)$ around $D$.
\end{lemma}

As a consequence, the intersection number (excluding $\pm\infty$)
\begin{equation}
u \cdot D = \sum_{z \in u^{-1}(D)} \mu_z(u)
\end{equation}
is finite; nonnegative; and zero if and only if $u$ is disjoint from $D$. 

\begin{proof}
(i) We need to quickly recall the Gromov trick. On $\bR \times S^1 \times M$ consider the almost complex structure 
\begin{equation} \label{eq:Gromov}
\tilde{J} = \begin{pmatrix}
i & 0 \\
(X_{s,t}\otimes \mathit{dt}) \circ i - J_{s,t}\circ (X_{s,t}\otimes \mathit{dt}) & J_{s,t}
\end{pmatrix}.
\end{equation}
Then $u$ solves \eqref{eq:continuation} iff its graph $\tilde{u}(s,t) = (s,t,u(s,t))$ is $\tilde{J}$-holomorphic. Our assumptions on $(J,H)$ ensure that $\tilde{D} = \bR \times S^1 \times D$ is a $\tilde{J}$-complex submanifold. We now apply standard pseudo-holomorphic curve theory: since $\tilde{u}$ is not contained in $\tilde{D}$, the subset $\tilde{u}^{-1}(\tilde{D})$ is discrete, and each point in it comes with positive multiplicity. But that subset equals $u^{-1}(D)$, and the multiplicities also remain the same.

We'll only do (ii), as the proof of (iii) is parallel. Define
\begin{equation} \label{eq:tilde-u}
u^\circ(s,t) = \rho_{\sigma_- t}(u(s,t)).
\end{equation}
For $s \ll 0$, this is a pseudo-holomorphic map (near $D$, where we are working, $J_{-,t}$ is independent of $t$ and $\rho$-invariant) and satisfies 
\begin{equation} \label{eq:twisted-periodicity}
u^\circ(s,t+1) = \rho_{\sigma_-}(u^\circ(s,t)).
\end{equation}
Choose some $N \in \bN$ such that $N\sigma_- \in \bZ$, and consider $u^\circ$ as defined for $s \ll 0$, $t \in \bR/N\bZ$. As a periodic pseudo-holomorphic map of finite energy, it necessarily extends smoothly to $-\infty$, and has an isolated intersection point with $D$ there. Let $\mu_{-\infty}(u^\circ)>0$ be the intersection multiplicity; for $s \ll 0$, the loop $\bR/N\bZ \rightarrow M$, $t \mapsto u^\circ(s,t)$, has winding number $\mu_{-\infty}(u^\circ)$ around $D$. Because of the condition \eqref{eq:twisted-periodicity}, that winding number must lie in $N(\bZ + \sigma_-)$. From \eqref{eq:tilde-u}, it follows that the loop $\bR/N\bZ \rightarrow M$, $t \mapsto u(s,t)$, has winding number $\mu_{-\infty}(u^\circ) - N\sigma_-$. We therefore get the desired result, with $\mu_{-\infty}(u) = (\mu_{-\infty}(u^\circ)-N\sigma_-)/N  > -\sigma_-$.
\end{proof}

The situation for our main technical Lemmas is as follows. Fix $(\bar{H}_{\pm}, \bar{J}_{\pm})$ with slopes $\sigma_{\pm}$, as before. Consider a sequence $(H_k,J_k)$ of data defining continuation map equations, all of which agree for $\pm s \gg 0$ with $(\bar{H}_{\pm}, \bar{J}_{\pm})$ (more precisely, the bounds where this holds should be independent of $k$). As $k \rightarrow \infty$, the $(H_k,J_k)$ should converge to some $(H,J)$.

\begin{lemma} \label{th:gromov-limit-0}
Let $(u_k)$ be a sequence of solutions of the continuation map equation for $(H_k,J_k)$, which have limits $x_{\pm}$ lying outside $D$. Suppose that the sequence Gromov-converges to a broken solution with cylindrical components $(u^i)$, $i = 1,\dots,I$,
together with pseudo-holomorphic sphere bubbles $v^{ij}$ attached to $u^i$. Then, for $k \gg 0$ we have
\begin{equation} \label{eq:intersection}
u_k \cdot D = \sum_{\!\!\!\! u^i \not\subset D\!\!\!\!} \big(u^i \cdot D + \mu_{-\infty}(u^i) + \mu_{+\infty}(u^i)\big)
+ \sum_{\!\!\!\!u^i \subset D\!\!\!\!} \bar{u}^i \cdot D 
+ \sum_{ij} (v^{ij} \cdot D).
\end{equation}
Here, we have used $\mu_{\pm \infty}$ also for cases when the limit is outside $D$, but in that case it is defined to be $0$; and for those $u_i \subset D$, the $\bar{u}^i$ are the obvious compactifications to pseudo-holomorphic spheres.
\end{lemma}

\begin{proof}
This is an elementary topological fact. Compactify all the cylindrical components by adding either a point (if the limit lies in $D$) or a circle (otherwise) to the ends; then join each compactified component to the next; and attach the spheres to them. The outcome is a compact connected nodal surface with two boundary circles (corresponding to the original limits $x_\pm$) and a continuous map from that surface to $M$. By definition of $\mu_{\pm\infty}$, the right hand side of \eqref{eq:intersection} is simply the intersection of that map with $D$. Because of the nature of the convergence process, this number is the same as that for $u_k$, $k \gg 0$.
\end{proof}

\begin{lemma} \label{th:gromov-limit-1}
Take a sequence $(u_k)$, with limits $x_{\pm}$ outside $D$, and such that
\begin{equation} \label{eq:m-bound}
u_k \cdot D = m, \text{ where } m < \lceil \sigma_- \rceil - \lfloor \sigma_+ \rfloor.
\end{equation}
Suppose that $u_k^{-1}(D)$ lies in a compact subset of $\bR \times S^1$, independent of $k$; and that our sequence Gromov-converges. Then, all one-periodic orbits which occur in the limiting broken solution are outside $D$; the principal component, together with the sphere bubbles attached to it, has intersection number $m$ with $D$; while all other cylindrical components are disjoint from $D$, and have no sphere bubbles attached to them.
\end{lemma}

\begin{proof}
Let $u^{i_*}$ be the principal component of the limit. This is a solution of the continuation map equation for $(H,J)$. The other components $u^i$ are Floer trajectories with slope $\sigma_-$ ($i<i_*$) or $\sigma_+$ ($i>i_*$). Denote the one-periodic orbits involved by
\begin{equation} \label{eq:the-orbits}
\textstyle
x^i = \begin{cases} 
\textstyle x_- = \lim_{s \rightarrow -\infty} u^1 & i=0, \\
\textstyle \lim_{s \rightarrow +\infty} u^i = \lim_{s \rightarrow-\infty} u^{i+1} & 0<i<I, \\
\textstyle x_+ = \lim_{s \rightarrow +\infty} u^I & i = I.
\end{cases}
\end{equation}
Suppose that for some $i$ and $c \geq 0$, we have the following:
\begin{equation} \label{eq:goes-into}
\parbox{34em}{
$u^i$ does not lie in $D$, but has $+\infty$ limit $x^i$ in $D$; $u^{i+1+c}$ does not lie in $D$, but has $-\infty$ limit $x^{i+c}$ in $D$; and the intermediate components $u^{i+1},\dots,u^{i+c}$ lie in $D$. 
}
\end{equation}
By Lemma \ref{th:winding} and \eqref{eq:m-bound},
\begin{equation} \label{eq:combine-m}
\mu_{+\infty}(u^i) + \mu_{-\infty}(u^{i+c+1}) 
\geq \begin{cases} 
\lceil \sigma_- \rceil - \lfloor \sigma_- \rfloor = 1 & i+c < i_*, \\
\lceil \sigma_+ \rceil - \lfloor \sigma_+ \rfloor = 1 & i > i_*, \\
\lceil \sigma_- \rceil - \lfloor \sigma_+ \rfloor > m
& \text{in the remaining case.}
\end{cases}
\end{equation}
After this preliminary consideration, the argument is as follows.

{\em If $u^{i_*}$ is not contained in $D$,} the $\mu_{\pm\infty}$ summands from \eqref{eq:intersection} can be arranged in pairs which belong to the first two cases of \eqref{eq:combine-m}, each time contributing positively. Moreover, because of the assumption on $u_k^{-1}(D)$, we have 
\begin{equation} \label{eq:principal-bubbles}
u^{i_*} \cdot D + \sum_j v^{i_*,j} \cdot D \geq m. 
\end{equation}
We have now written \eqref{eq:intersection} as a sum of nonnegative terms, one of which is $\geq m$. As a consequence, this term must be equal to $m$, and all other terms must be zero. This means that after all, \eqref{eq:goes-into} can't happen. Note that the condition \eqref{eq:m-bound} has not entered into this part of the argument.

{\em If $u^{i_*}$ is contained in $D$,} the $\mu_{\pm\infty}$ summands from \eqref{eq:intersection} can be arranged in pairs as in \eqref{eq:combine-m}, and where the third case occurs exactly once. We have now written \eqref{eq:intersection} as a sum of contributions which are all nonnegative, and one of which is $> m$, yielding a contradiction.
\end{proof}

\begin{example} \label{th:gromov-limit-2}
The simplest instance of Lemma \ref{th:gromov-limit-1} is when $m = 0$, or equivalently the $u_k$ are disjoint from $D$. In that case, the outcome for the Gromov limit is: there are no sphere bubbles; and the cylindrical components, together with the one-periodic orbits involved, remain outside $D$. When doing Floer theory in $M \setminus D$, this argument replaces the conventional use of the maximum principle.
\end{example}

In applications, we will often encounter a generalization where the continuation map equations themselves split into $R \geq 2$ pieces. The relevant setup requires a bit of patience.
\begin{itemize} \itemsep.5em
\item Fix slopes $\sigma^r$, $r = 0,\dots,R$, and corresponding Floer data $(\bar{H}^r,\bar{J}^r)$. 
\item For $r = 1,\dots,R$ and $k \in \bN$, choose data $(H^r_k,J^r_k)$ for continuation maps, with similar properties as before: they agree with $(\bar{H}^{r-1},\bar{J}^{r-1})$ for $s \ll 0$, with $(\bar{H}^r, \bar{J}^r)$ for $s \gg 0$ (with $s$-bounds that are independent of $k$); and as $k \rightarrow \infty$, $(H^r_k,J^r_k)$ converges to some $(H^r,J^r)$. It is convenient to consider these continuation map equations as living on separate copies $C^r = \bR \times S^1$ of the cylinder.
\item Fix gluing lengths $l^r_k$, for $r = 1,\dots,R-1$, each of which is large and goes to $\infty$ as $k \rightarrow \infty$. Glue the $s \rightarrow \infty$ end of $C^r$ to the $s \rightarrow -\infty$ end of $C^{r+1}$, by identifying $(s,t) \in C^r$ with $(s-l^r_k) \in C^{r+1}$. Altogether, this yields a single cylinder $C_k \iso \bR \times S^1$, which will carry data $(H_k,J_k)$. \end{itemize}

\begin{lemma} \label{th:complicated}
Take a sequence $u_k: C_k \rightarrow M$ of continuation map solutions for $(H_k,J_k)$, with fixed limits $x_{\pm}$ lying outside $D$. Suppose that $u_k^{-1}(D)$ can be decomposed into sets-with-multiplicity $\Sigma^1_k,\dots,\Sigma^R_k$, where
\begin{equation} \label{eq:mr}
m^r = |\Sigma^r_k| < \lceil \sigma^{r-1} \rceil - \lfloor \sigma^r \rfloor,
\end{equation}
and such that $\Sigma^r_k$ lies in a bounded (independently of $k$) part of $C^r$, glued into $C_k$. If the sequence Gromov-converges, all one-periodic orbits which occur in the limiting broken solution lie outside $D$; the principal component living on each $C^r$, together with the sphere bubbles attached to it, has intersection number $m^r$ with $D$; all other cylindrical components are disjoint from $D$, and carry no sphere bubbles.
\end{lemma}

\begin{proof}
The limit has cylindrical pieces $u^i$, $i = 1,\dots,I$. The principal ones, labeled by 
\begin{equation} \label{eq:principal-ones}
i = i_*^r \text{ for some }i_*^1 < \cdots < i_*^R, 
\end{equation}
solve the continuation map on $C^r$ associated to $(H^r,J^r)$. The other ones are Floer trajectories, and of course there are sphere bubbles $v^{ij}$ as well. The formula \eqref{eq:intersection} still holds.

{\em For each principal component $u^i$, $i = i_*^r$, which is not contained in $D$}, we have the counterpart of \eqref{eq:principal-bubbles}, meaning that the component together with bubbles attached to it has intersection number $\geq m^r$ with $D$; this is because of the assumption on the position of $\Sigma^r_k$.

{\em Consider a chain of successive components contained in $D$, as in \eqref{eq:goes-into}}. Let's say this chain includes $b \geq 0$ principal components, namely those with $i = i_*^r$, \dots, $i = i_*^{r+b-1}$. As in \eqref{eq:combine-m} one has
\begin{equation}
\mu_{+\infty}(u^i) + \mu_{-\infty}(u^{i+c+1}) \geq \lceil \sigma^r \rceil  - \lfloor \sigma^{r+b} \rfloor > m^r + \cdots + m^{r+b-1}.
\end{equation}

Summing over all terms of those two kinds already yields a contribution to \eqref{eq:intersection} which is at least $\sum m^r$, with equality only possible if the second situation never happens. Moreover, it follows that the only intersections with $D$ can happen on principal components and their bubbles.
\end{proof}

\subsection{Symplectic cohomology\label{subsec:sh}}
Fix an increasing sequence $\sigma_w \in (\bQ \setminus \bZ)^{>0}$, $w = 0,1,\dots$, with the following property:
\begin{equation} \label{eq:one-integer}
\text{there is at least one integer between $\sigma_w$ and $\sigma_{w+1}$.}
\end{equation}
(And therefore, $\sigma_w > w$.) For each $w$, we choose a time-dependent Hamiltonian $\bar{H}_w$ of slope $\sigma_w$, and corresponding almost complex structures $\bar{J}_w$, with suitable genericity properties. Let $\mathit{CF}(w) = \mathit{CF}(\bar{H}_w)$ be the associated Floer complex, with Floer differential $d_0$, using only one-periodic orbits and Floer trajectories in $M \setminus D$. The telescope construction is the chain complex
\begin{equation} \label{eq:telescope}
C = \big(\mathit{CF}(0) \oplus \mathit{CF}(1) \oplus \cdots\big) \oplus \eta\big(\mathit{CF}(0) \oplus \mathit{CF}(1) \oplus \cdots\big) 
\end{equation}
where $\eta$ is a symbol of degree $-1$. The differential $d_C$ is built from the following (Figure \ref{fig:telescope}):
\begin{itemize} \itemsep.5em
\item The Floer differential $d_0$ on each $\mathit{CF}(w)$; and $-d_0$ on each piece $\eta \mathit{CF}(w)$. We schematically represent this by drawing the cylinder on which the Floer equation lives:
\begin{equation}
\xymatrix{
\mathit{CF}^{*+1}(w) 
&& \ar[ll]_-{\includegraphics[valign=c]{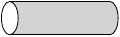}}^{\displaystyle d_0} \mathit{CF}^*(w)
}
\end{equation}
\item Continuation maps $d_0^\dag$ from each slope to the next, again using only solutions in $M \setminus D$. We draw the underlying Riemann surface as a cylinder with an additional marked circle, which reminds us that the continuation map equation breaks translation-invariance:
\begin{equation} \label{eq:c0-map}
\xymatrix{
\mathit{CF}^{*}(w+1) 
&& \ar[ll]_-{\includegraphics[valign=c]{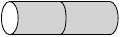}}^{\displaystyle d_0^\dag} \mathit{CF}^*(w)
}
\end{equation}
In the telescope construction, we use $d_0^\dag$ as a degree $1$ map $\eta \mathit{CF}(w) \rightarrow \mathit{CF}(w+1)$.
\item
$(-1)$ times the identity map $\eta\mathit{CF}(w) \rightarrow \mathit{CF}(w)$.
\end{itemize}
The restriction to solutions in $M \setminus D$, both for the differential and continuation map equation, works thanks to Example \ref{th:gromov-limit-2}. The cohomology of $C$ will be denoted by $\mathit{SH}^*(M \setminus D)$, as usual.

Let $q$ be a formal variable of degree $2$. Consider
\begin{equation} \label{eq:deformed-telescope}
C_q = C[[q]],
\end{equation}
with a $q$-deformed version $d_{C_q}$ of the previous differential $d_C$. Technical details will be explained later; for now, we give an informal description.
\begin{itemize} \itemsep.5em
\item For every $w \geq 0$ and $m>0$, we construct a map 
\begin{equation} \label{eq:dm-map}
\xymatrix{
\mathit{CF}^{*+1-2m}(w+m) &&
\ar[ll]^-{\displaystyle d_m}_-{\overbrace{\includegraphics[valign=c]{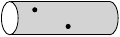}}^{\text{$m$ points}}} \mathit{CF}^*(w)
}
\end{equation}
which counts solutions of an appropriate continuation map equation, having intersection number $m$ with $D$. The precise equation will depend on the location of the intersection points (hence the picture), chosen to be compatible with degenerations in which the cylinder splits into pieces. More precisely, we use a parameter space of degree $m$ divisors on the cylinder, up to translation, and a corresponding parametrized continuation map equation. Taken together with the Floer differential $d_0$, these maps satisfy
\begin{equation} \label{eq:dm-equation}
\sum_{i+j = m} d_id_j = 0.
\end{equation}
They contribute to $d_{C_q}$ through terms $q^m d_m: \mathit{CF}(w)[[q]] \rightarrow \mathit{CF}(w+m)[[q]]$ as well as $-q^m d_m: \eta\mathit{CF}(w)[[q]] \rightarrow \eta\mathit{CF}(w+m)[[q]]$.

\item For every $w \geq 0$ and $m \geq 0$, we have a map
\begin{equation} \label{eq:cm-map}
\xymatrix{
\mathit{CF}^{*-2m}(w+m+1) &&
\ar[ll]^-{\displaystyle d_m^\dag}_-{\overbrace{\includegraphics[valign=c]{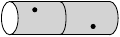}}^{\text{$m$ points}}} \mathit{CF}^*(w)
}
\end{equation}
generalizing \eqref{eq:c0-map}. These maps satisfy
\begin{equation} \label{eq:cm-equation}
\sum_{i+j=m} d_i^\dag d_j - d_i d_j^\dag = 0.
\end{equation}
The definition of $d_m^\dag$ is similar to that of $d_m$, but the parameter space is not divided by translation (in \eqref{eq:cm-map} this is represented by drawing the extra circle) and the slopes are different. In $d_{C_q}$, we use a term $q^m d_m^\dag: \eta\mathit{CF}(w)[[q]] \rightarrow \mathit{CF}(w+m+1)[[q]]$.
\end{itemize}
The relations \eqref{eq:dm-equation} and \eqref{eq:cm-equation} imply that $d_{C_q}^2 = 0$  (see Figure \ref{fig:deformed-telescope} for a summary). We denote the cohomology of $C_q$ by $\mathit{SH}^*_q(M,D)$.
\begin{figure}
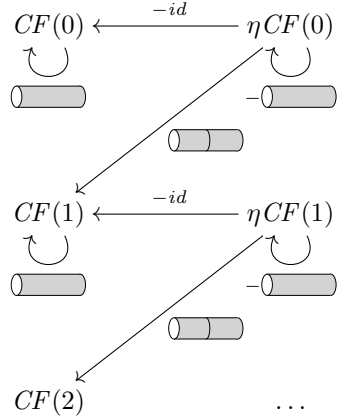

\begin{centering}
\[
\xymatrix{ \mathit{CF}(0) 
\ar@(dr,dl)[]^-{\includegraphics[scale=.5]{cylinder.pdf}}
&&
\eta \mathit{CF}(0)
\ar[ll]_{-\mathit{id}} 
\ar@(dr,dl)[]^-{-\includegraphics[scale=.5,valign=c]{cylinder.pdf}}
\ar[ddll]^-{\includegraphics[scale=.5]{cylinder-with-circle.pdf}}
\\ \\
\mathit{CF}(1)
\ar@(dr,dl)[]^-{\includegraphics[scale=.5]{cylinder.pdf}}
&&
\eta\mathit{CF}(1) 
\ar[ll]_{-\mathit{id}} 
\ar[ddll]^-{\includegraphics[scale=.5]{cylinder-with-circle.pdf}}
\ar@(dr,dl)[]^-{-\includegraphics[scale=.5,valign=c]{cylinder.pdf}}
\\ \\
\mathit{CF}(2) && \dots
}
\]
\caption{\label{fig:telescope}A schematic picture of the telescope construction \eqref{eq:telescope}.}
\end{centering}
\end{figure}%
\begin{figure}
\begin{centering}
\[
\xymatrix{ \mathit{CF}(0)[[q]] 
\ar[dd]_-{q\,\includegraphics[scale=.5,valign=c]{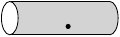}}
\ar@/_6pc/[dddd]^-{q^2\,\includegraphics[scale=.5,valign=c]{cylinder-with-dots.pdf}}
&&
\eta \mathit{CF}(0)[[q]]
\ar[ll]_{-\mathit{id}} 
\ar[dd]^-{-q\,\includegraphics[scale=.5,valign=c]{cylinder-with-dot.pdf}}
\ar[ddll]^-{\includegraphics[scale=.5,valign=c]{cylinder-with-circle.pdf}}
\ar@/^6pc/[dddd]_-{-q^2\,\includegraphics[scale=.5,valign=c]{cylinder-with-dots.pdf}}
\ar`[rr]`[ddddd]^-{q\,\includegraphics[scale=.5,valign=c]{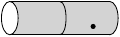}}`[dddddll][ddddll]
&& \\ \\
\mathit{CF}(1)[[q]]
\ar[dd]_-{q\,\includegraphics[scale=.5,valign=c]{cylinder-with-dot.pdf}}
&&
\eta\mathit{CF}(1)[[q]] 
\ar[ll]_{-\mathit{id}} 
\ar[ddll]^-{\includegraphics[scale=.5,valign=c]{cylinder-with-circle.pdf}}
\ar[dd]^-{-q\,\includegraphics[scale=.5,valign=c]{cylinder-with-dot.pdf}}
\\ \\
\mathit{CF}(2)[[q]] && \dots
\\
&& &&
}
\]
\caption{\label{fig:deformed-telescope}A schematic picture of the deformed telescope construction \eqref{eq:deformed-telescope}. The Floer differentials on each summand have been omitted. The total complex is $q$-complete (meaning, it is not just the direct sum of the $\mathit{CF}(m)[[q]]$ and $\eta\mathit{CF}(m)[[q]]$ pieces).}
\end{centering}
\end{figure}%

\subsection{The construction of $d_m$\label{subsec:dm}}
For $m > 0$, consider a collection of $m$ unordered points-with-multiplicity on the cylinder. 
One writes this as
\begin{equation} \label{eq:point-configuration}
\Sigma = (\Sigma_z)_{z \in \bR \times S^1} \quad \text{ with } \Sigma_z \geq 0, \textstyle \; \sum_z \Sigma_z = m.
\end{equation}
This is of course a divisor of degree $m$ on the Riemann surface $\bR \times S^1$ (not to be confused with our use of the divisor $D$ in the target space $M$). Such collections, up to translation in $\bR$-direction, are parametrized by $\frakD_m = \mathit{Sym}_m(\bR \times S^1)/\bR$. This space has a very simple compactification $\overline{\frakD}_m$ to a manifold with corners, where the cylinder can split into several ones, each carrying at least one point of our divisor. As a set,
\begin{equation} \label{eq:d-strata}
\bar\frakD_m = \coprod_{\substack{R \geq 1 \\ m^1 + \cdots + m^R = m}}
\frakD_{m^1} \times \cdots \times \frakD_{m^R}.
\end{equation}
Over the space $\frakD_m$ there is an obvious universal curve (a fibration with fibres isomorphic to $\bR \times S^1$, where the identification is unique up to translation, and containing a canonical collection of marked points). This universal curve extends to $\overline{\frakD}_m$, where its fibre over each stratum \eqref{eq:d-strata} consists of $r$ disjoint cylinders. 

To define the operation \eqref{eq:dm-map} one chooses data for continuation map equations on the universal curve over $\frakD_m$, which equal $(\bar{H}_w, \bar{J}_w)$ for $s \gg 0$, respectively $(\bar{H}_{w+m}, \bar{J}_{w+m})$ for $s \ll 0$. Here, it is understood that the region where this Floer-like behaviour holds is always disjoint from the marked points.
We ask that the data should extend smoothly to $\bar\frakD_m$, and on each stratum \eqref{eq:d-strata} should be obtained by pullback from those on the $\frakD_{m_k}$. There is an additional condition in that context, which imposes a stricter consistency with gluing. Take a sequence of points in $\frakD_m$, converging to a limit in some boundary stratum. The associated fibres $C_k$ of the universal curve can be thought of as being glued together from $R$ pieces $C^1,\dots,C^R$. Then, there should be an $S>0$ (independent of $k$) such that the continuation map datum on $C_k$ agrees with the relevant Floer data outside pieces $[-S,S] \times S^1 \subset C^1$, \dots, $[-S,S] \times S^1 \subset C^R$. Informally speaking, the datum on $C_k$ is Floer type on necks. 

Given one-periodic orbits $x_{\pm}$ outside $D$, take the space $\frakD_m(x_-,x_+)$ of pairs consisting of a point of $\frakD_m$ and a solution $u: C \rightarrow M$ of the associated continuation map equation, living on the fibre $C$ of the universal curve. This map should have limits $x_{\pm}$ and intersections
\begin{equation} \label{eq:intersection-m}
u \cdot D = m, \;\; u^{-1}(D) = \Sigma,
\end{equation}
where the second part is understood with multiplicity. Given that and the consistency conditions, one gets a uniform bound on the energy of $u$ from \eqref{eq:energy}. Let $\bar\frakD_m(x_-,x_+)$ be the standard Gromov compactification (as a stable map, where we do not treat the points of $\Sigma$ as marked points). Concerning the structure of Gromov limits, the heavy lifting has already been done by Lemma \ref{th:complicated}. We add one observation, which follows immediately by looking at the contributions to intersection numbers:

\begin{lemma} \label{th:bubble}
Take any point $z$ of a cylindrical component $u^i$ of the Gromov limit; write $(v_z^{ij})$ for the spheres in the bubble tree (if any) attached at this point. Then 
\begin{equation}
\mu_z(u^i) + \sum_j v_z^{ij} \cdot D = \Sigma_z. 
\end{equation}
In particular, bubbling can only happen at marked points. \qed
\end{lemma}

As far as transversality is concerned, we never look at the entire space $\frakD_m$, but only at the strata where the marked points coincide according to a fixed partition of $m$. If $\Pi$ is the partition, consisting of $1 \leq |\Pi| \leq m$ summands, then that stratum is a finite cover of the unordered configuration space $\mathit{Conf}_{|\Pi|}(\bR \times S^1)/\bR$ (finite cover, because the points come with locally constant multiplicities). Take the fibre of the universal curve $(C,\Sigma)$ at a point of such a stratum, and a map $u$ as in \eqref{eq:intersection-m}. If we look at the data for the continuation map at points $(z,u(z))$, they are constrained when $z$ lies on a neck (by consistency), and when $u(z) \in D$ (because of the condition that the divisor be preserved), which means $z \in \Sigma$; that always leaves an open subset of points on our curve where they can be chosen freely, which is sufficient for all transversality arguments (see \cite{cieliebak-mohnke07} for background on transversality with tangency constraints). With that in mind, a generic choice achieves the following regularity properties:
\begin{enumerate}[label=(D\arabic*)] \itemsep.5em
\item \label{item:d-main}
{\em (Main stratum)} Take the subset of $\frakD_m(x_-,x_+)$ where the marked points are pairwise distinct; we can assume that this is regular. By \eqref{eq:index}, its dimension is $\mathrm{deg}(x_-) - \mathrm{deg}(x_+) + 2m-1$.

\item \label{item:d-collision}
{\em (Collision, ignore bubbles)} Consider the stratum of $\frakD_m$ where the coincidence of marked points is described by a partition $\Pi$ with $|\Pi|<m$ (in other words, it's not the partition $m = 1+\cdots+1$). We require that $u$ may not intersect $D$ outside $\Sigma$; but at the points of $\Sigma$, it does not have to intersect $D$, and if does, the intersection multiplicity can be arbitrary. The resulting moduli spaces will have different connected components, depending on the multiplicities. Regularity implies that this space is of dimension $\mathrm{deg}(x_-) - \mathrm{deg}(x_+) + 2|\Pi| - 1$. 

\item \label{item:d-simple-bubble}
{\em (Bubbling without collision)} Let's return to the situation where the marked points are pairwise distinct. We assume as in \ref{item:d-collision} that $u$ may not intersect $D$ outside $\Sigma$. Additionally, there should be one point $z$ in $\Sigma$ and a sphere bubble $v$ (pseudo-holomorphic for the almost complex structure governing the continuation map equation of $u$ at $z$), such that $v \cdot D = 1$ and $u(z)$ lies on the image of $v$.
The coincidence condition with a Chern number $1$ sphere is of codimension $2$ if the sphere is not contained in $D$, and of codimension $4$ (inside $D$) if it is. Therefore, the resulting space has dimension $\leq \mathrm{deg}(x_-) - \mathrm{deg}(x_+) + 2m-3$.
\end{enumerate}

We can now complete the standard description of low-dimensional moduli spaces.

\begin{proposition} \label{th:floer}
Under the regularity assumptions imposed above, consider spaces $\frakD_m(x_-,x_+)$ of dimension $\mathrm{deg}(x_-) - \mathrm{deg}(x_+) + 2m - 1 \leq 1$.

(i) All points in such a space consist of curves where the marked points are pairwise distinct.

(ii) If the dimension is $0$, we have $\bar\frakD_m(x_-,x_+) \setminus \frakD_m(x_-,x_+) = \emptyset$. Hence, the space $\frakD_m(x_-,x_+)$ is itself compact.

(iii) If the dimension is $1$, the points in $\bar\frakD_m(x_-,x_+) \setminus \frakD_m(x_-,x_+)$ consist of the cylinder splitting into exactly two pieces, where the marked points are still pairwise distinct, and there are no other components or bubbles. Hence, $\bar\frakD_m(x_-,x_+)$ is a compact one-manifold, and the points of $\bar\frakD_m(x_-,x_+) \setminus \frakD_m(x_-,x_+)$ are its boundary.
\end{proposition}

\begin{proof}
(i) is clear: the other strata are as in \ref{item:d-collision} (in the special case where $\mu_z(u) = \Sigma_z$), hence of codimension $\geq 2$. For (ii) and (iii), look at a point in the Gromov compactification, using Lemma \ref{th:complicated}. We label the cylindrical components as usual by $u^i$; their limits as in \eqref{eq:the-orbits}; the principal components as in \eqref{eq:principal-ones}; and the number of marked points on each principal component is as in \eqref{eq:mr}. Every non-principal component $u^i$ is a Floer trajectory in $M \setminus D$, sitting in a moduli space of dimension $\mathrm{deg}(x^{i-1})-\mathrm{deg}(x^i) - 1$. Next, consider a principal component $u^i$, $i = i_*^r$. 
\begin{itemize} \itemsep.5em
\item
{\em If on this component, all marked points are distinct and there is no bubbling}, we are in the situation of \ref{item:d-main} above, and the dimension is $\mathrm{deg}(x^{i-1}) - \mathrm{deg}(x^i) + 2m^r-1$. 
\item
{\em If at least two marked points have collided}, we ignore any bubbles attached. By Lemma \ref{th:bubble}, the map may not intersect $D$ outside the marked points. Therefore, find that $u^i$ belongs to a space of dimension $\leq \mathrm{deg}(x^{i-1})-\mathrm{deg}(x^i) + 2m^r- 3$, following \ref{item:d-collision}. 
\item
{\em Suppose that there is no collision of marked points, but that bubbling has occurred.} By Lemma \ref{th:bubble}, bubbling can only happen at marked points, and all bubbles must have intersection number $1$ with $D$. We forget all but one bubble, and see from \ref{item:d-simple-bubble} that $u^i$ again lies in a space of dimension $\leq \mathrm{deg}(x^{i-1}) - \mathrm{deg}(x^i) + 2m^r - 3$.
\end{itemize}
If we take the moduli spaces to which the cylindrical components belong, then their dimensions add up to $\leq \mathrm{deg}(x_+) - \mathrm{deg}(x_-) + 2m - I$ (where $I$ is the number of components); and the sum is $\leq \mathrm{deg}(x_-) - \mathrm{deg}(x_-) + 2m - 2 - I$ if at least one component is subject to collision of marked points or bubbling. This implies the desired result.
\end{proof}

\begin{remark}
Our discussion of transversality assumptions included more situations than actually occur in the compactifications of low-dimensional moduli spaces. For instance, when \ref{item:d-simple-bubble} appears in the proof of Proposition \ref{th:floer}, we additionally know that $\mu_z(u) \leq 1$ at points of $\Sigma$, and $\mu_z(u) = 0$ at the point where the bubble is attached. It has been an expository choice not to carry those restrictions over into \ref{item:d-simple-bubble}, since the more general context leads to the same dimension bounds.
\end{remark}

Given that, one defines $d_m$ by counting elements of the zero-dimensional spaces $\frakD_m(x_-,x_+)$. Proposition \ref{th:floer}(ii) then implies that \eqref{eq:dm-equation} will hold.

\subsection{The construction of $d_m^\dag$}
This construction is very similar to the previous one, so we'll be very brief. The starting point is the space $\frakD_m^\dag = \mathit{Sym}_m(\bR \times S^1)$, where we do not divide by translation. Quotients by translation appears in the limit as points go to $\pm\infty$, which means that the compactification has the form
\begin{equation} \label{eq:c-strata}
\bar\frakD_m^\dag = \coprod_{\substack{R \geq 1,\; c \in \{1,\dots,R\} \\ m^1 + \cdots + m^R = m}}
\frakD_{m^1} \times \cdots \times \frakD_{m^c}^\dag \times \cdots \times \frakD_{m^R}.
\end{equation}
We choose a family of continuation map data on the universal curve, which equals $(\bar{H}_m,\bar{J}_m)$ for $s \gg 0$, and $(\bar{H}_{m+w+1},\bar{J}_{m+w+1})$ for $s \ll 0$. In the limit as a curve splits into pieces, one wants to get the corresponding data on the cylinder corresponding to the $\frakD_{m^c}^\dag$ factor, and the previously chosen data underlying the maps $d_m$ on the other cylinders. We then again define $d_m^\dag$ by counting points.

\section{Thimbles\label{sec:thimbles}}
Pseudo-holomorphic maps on a genus zero surface with one end (the thimble) are a standard tool used to relate ordinary cohomology and Floer theory. To apply that strategy to our problem, we need to revisit the previous discussion of Floer trajectories, updating it to the thimble situation. After that, we introduce two constructions using such thimbles, where the second one has the divisor taking on a more central role. 

\subsection{Morse theory notation\label{subsec:morse}}
Suppose we have a Morse function $f_M: M \rightarrow \bR$ and gradient or pseudo-gradient vector field $X_M$. (A pseudo-gradient is a vector field such $X_M.f_M > 0$ outside the critical point set, and which is the actual gradient, for some metric, near each critical point.) We usually think in terms of the negative vector field $-X_M$. Hence, if $c_+$ is a critical point, the stable manifold $W^s(M,c_+)$ consists of those point that flow downwards into $c_+$. Explicitly, a point lies in $W^s(M,c_+)$ if it is the starting point $b(0)$ of a half-trajectory
\begin{equation} \label{eq:half-trajectory}
\left\{
\begin{aligned}
& b: [0,\infty) \longrightarrow M, \\
& db/ds = -X_M, \\
& \textstyle \lim_{s \rightarrow +\infty} b(s) = c_+;
\end{aligned}
\right.
\end{equation}
correspondingly, the unstable manifold $W^u(M,c_-)$ is the space of endpoints $b(0)$ for negative half-trajectories
\begin{equation} \label{eq:negative-half-trajectory}
\left\{
\begin{aligned}
& b: (-\infty,0] \longrightarrow M, \\
& db/ds = -X_M, \\
& \textstyle \lim_{s \rightarrow -\infty} b(s) = c_-;
\end{aligned}
\right.
\end{equation}
If $(f_M,X_M)$ is Morse-Smale (the stable and unstable manifolds intersect transversally), we write $\mathit{CM}^*(M) = \mathit{CM}^*(f_M)$ for the resulting Morse complex, with cohomological conventions. This means that a flow line $b: \bR \rightarrow M$ with asymptotics $c_{\pm}$ as before, contributes to the coefficient of the differential which takes $c_+$ to $c_-$.

\subsection{Basics}
Take the Riemann surface $T = (\bR \times S^1) \cup \{+\infty\} \iso \bC$. Suppose we have Floer data $(\bar{H}_-,\bar{J}_-)$. Choose $(H_{s,t}, J_{s,t})$ which agree with $(\bar{H}_-,\bar{J}_-)$ for $s \ll 0$, but such that on the opposite end, both the family $J_{s,t}$ and the $\smooth(M)$-valued one-form $H_{s,t}\, \mathit{dt}$ extend smoothly over $+\infty \in T$. One then looks at maps
\begin{equation} \label{eq:thimble-map}
\left\{
\begin{aligned}
& u: T \longrightarrow M, \\
& \textstyle \lim_{s \rightarrow -\infty} u = x_-, 
\end{aligned}
\right.
\end{equation}
satisfying the same equation as in \eqref{eq:continuation} (that also extends over $+\infty$). The counterpart of \eqref{eq:energy} for such solutions just omits the $x_+$ term,
\begin{equation} \label{eq:thimble-energy}
E(u) = \int_{\bR \times S^1} \|\partial_s u\|^2 = A(x_-) + (u \cdot D) + \int_{\bR \times S^1} u^*(\partial_s H_{s,t}).
\end{equation}
Lemma \ref{th:gromov-limit-0} carries over, with a slight change of interpretation. We have components $u^i$, $i = 1,\dots,I$, of which all but the last one are Floer cylinders. The last one (the principal component) is defined on the thimble; there, the $\mu_{+\infty}(u^I)$ term is an ordinary intersection multiplicity with $D$, hence nonnegative. The counterpart of Lemma \ref{th:gromov-limit-1} is:

\begin{lemma} \label{th:thimble-1}
Take a sequence of continuation map data $(H_k,J_k)$, all of which agree with $(\bar{H}_-, \bar{J}_-)$ for $s \ll 0$ (on a subset that's independent of $k$), and which as $k \rightarrow \infty$ converge to some $(H,J)$. Take a sequence $(u_k: T \rightarrow M)$ of solutions, which have limit $x_-$ lying outside $D$, and such that
\begin{equation} \label{eq:m-left}
u_k \cdot D = m, \text{ where } m < \lceil \sigma_- \rceil.
\end{equation}
Suppose that $u_k^{-1}(D)$ lies in a compact subset of $T$, independent of $k$; and that our sequence Gromov-converges. Then, all one-periodic orbits which occur in the broken solution are outside $D$; the principal component, together with the sphere bubbles attached to it, has intersection number $m$ with $D$; and all other components are disjoint from $D$.
\end{lemma}

\begin{proof}
The argument is essentially the same as before, so we'll only cover one step, which is where \eqref{eq:m-left} enters. Suppose that for some $i<I$, we have:
\begin{equation} \label{eq:thimble-in-d}
\parbox{34em}{$u^i$ does not lie in $D$, but has $+\infty$ limit $x^i$ in $D$; and all subsequent components $u^{i+1},\dots$, including $u^I$, are contained in $D$.}
\end{equation}
In that case, we know from Lemma \ref{th:winding} that $\mu_{+\infty}(u^i) \geq \lceil \sigma_- \rceil > m$. All the other contributions to the total intersection number \eqref{eq:intersection} are nonnegative, so \eqref{eq:thimble-in-d} is after all impossible.
\end{proof}

There is also an analogue of Lemma \ref{th:gromov-limit-2}, where as $k \rightarrow \infty$, the equation itself degenerates to ones defined on a finite collections of cylinders, together with the thimble. Because there are no new ingredients, we will not write down the statement in detail.

\subsection{Classical thimbles\label{subsec:classical-thimble}}
The first part of our construction is closest to its classical origin in \cite{piunikhin-salamon-schwarz95}. It consists of a map
\begin{equation} \label{eq:pss-map}
H^*(M) \longrightarrow \mathit{SH}^*_q(M,D)
\end{equation}
which, extended $q$-linearly, forms the simpler part of \eqref{eq:main}. Here's a summary of the construction.
\begin{itemize}
\item We can define a chain map 
\begin{equation} \label{eq:s0}
\xymatrix{
\mathit{CF}^*(0)
&&& \ar[lll]_-{\includegraphics[valign=c]{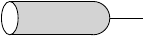}}^{\displaystyle s_0} \mathit{CM}^*(M)
}
\end{equation}
as follows. Fix a one-periodic orbit $x_-$ of our Hamiltonian, and a critical point $c_+$ of the Morse function. The corresponding coefficient of \eqref{eq:s0} counts pairs $(u,b)$ consisting of a map \eqref{eq:thimble-map} in $M \setminus D$ and a half-trajectory \eqref{eq:half-trajectory}, joined together by the assumption that
\begin{equation} \label{eq:uv}
u(+\infty) = b(0). 
\end{equation}
\end{itemize}
From now on, let's impose the following assumption:
\begin{equation}
\label{eq:stable-transverse}
\parbox{34em}{All stable manifolds $W^s(c_+)$ are transverse to $D$.}
\end{equation}

\begin{itemize} \itemsep.5em
\item 
For all $m > 0$ one can define maps
\begin{equation} \label{eq:d-t}
\xymatrix{
\mathit{CF}^{*-2m}(m) 
&&& \ar[lll]_-{\overbrace{\includegraphics[valign=c]{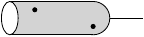}\hspace{-2em}}^{\text{$m$ points}}\hspace{2em}}^{\displaystyle s_m} \mathit{CM}^*(M)
}
\end{equation}
by looking at pairs $(u,b)$ as before, where this time $u$ has intersection number $m$ with $D$. The basic property of these maps (including $s_0$) is that
\begin{equation} \label{eq:d-s-morse}
\sum_{i+j = m} d_i s_j = s_m d_M.
\end{equation}
\end{itemize}
The chain map underlying \eqref{eq:pss-map} is 
\begin{equation} \label{eq:total-s-map}
s_{C_q} = \sum_m q^m s_m: \mathit{CM}^*(M) \longrightarrow C_q.
\end{equation}

\begin{remark}
In the definition of $s_0$, one can choose the Morse function so that it restricts to a Morse function on $D$, with negative Hessian in normal direction at critical points (meaning, the Morse indices in $M$ are $2$ larger than those in $D$). In that case, the critical points lying in $D$ form a subcomplex of $\mathit{CM}^*(M)$, and $s_0$ factors through the quotient. The induced map on that quotient realizes the ordinary Piunikhin-Salamon-Schwarz map 
\begin{equation}
H^*(M \setminus D) \longrightarrow \mathit{HF}^*(0),
\end{equation}
which is an isomorphism if one has chosen $\sigma_0 \in (0,1)$. There is no such factorization for the higher $s_m$.
\end{remark}

To carry out the actual construction, we use the parameter space $\frakS_m = \mathit{Sym}_m(T)$ and its compactification where cylinders can split off at $-\infty$,
\begin{equation} \label{eq:bar-s}
\bar\frakS_m = \coprod_{\substack{R\geq 1 \\ m^1+\cdots+m^R = m}} 
\frakD_{m^1} \times \cdots \times \frakD_{m^{R-1}} \times \frakS_{m^R}.
\end{equation}
Over $\frakS_m$, we choose data on the universal curve to define an equation \eqref{eq:thimble-map}, equaling $(\bar{H}_m,\bar{J}_m)$ for $s \ll 0$; again, this behaviour on the end should only be enforced away from marked points. The data should satisfy consistency constraints similar to those in Section \ref{subsec:dm}, which involve the data previously chosen on the $\frakD$ spaces. The associated moduli space $\frakS_m(x_-,c_+)$ consists of a point $\Sigma \in \frakS_m$ and a pair $(u,b)$ as before, such that 
\begin{equation} \label{eq:intersect-sigma}
u^{-1}(D) = \Sigma
\end{equation}
(considered as usual with multiplicities). There is a standard compactification $\bar{\frakS}_m(x_-,c_+)$, which combines Gromov convergence (considering $T$ as a Riemann surface with a distinguished point at $+\infty$) and breaking of Morse trajectories. More precisely, a point in the compactification consists of: cylindrical components $u^1,\dots,u^{I-1}$, of which those with $i = i_*^1,\dots,i_*^{R-1}$ are principal components (they carry marked points and satisfy continuation map equations, corresponding in \eqref{eq:bar-s} to the first $R-1$ factors), and the rest Floer components; a thimble component $u^I$; a pseudo-gradient half-line $b^1$; further trajectories $b^2,\dots$; and bubbles. The condition $u^I(+\infty) = b^1(0)$ may not be satisfied in such a limit, but then the two points must be joined by a component of the bubble tree. In that context, we need an elementary estimate for the dimension of certain bubble configurations.

\begin{lemma} \label{th:chain-dimension}
Take a simple chain $(v^1,\dots,v^K)$, $K>0$, of pseudo-holomorphic spheres. This means that each component is a simple (not multiply covered) pseudo-holomorphic sphere, $v^k: S = (\bR \times S^1) \cup \{\pm \infty\} \rightarrow M$; their images are pairwise distinct (no two are reparametrizations of each other); and
\begin{equation} \label{eq:chain-incidence}
v^1(+\infty) = v^2(-\infty), \dots, v^{K-1}(+\infty) = v^K(-\infty).
\end{equation}

(i) If it is regular (which holds for generic almost complex structures), the space of such chains, modulo the action of $\bC^* = \bR \times S^1$ on each component, is of dimension 
\begin{equation} \label{eq:leq}
\leq 2n - 2K + 2\sum_{k=1}^K (v^k \cdot D).
\end{equation}

(ii) If we additionally require that $v^1(-\infty) \in D$ or $v^K(+\infty) \in D$, the dimension bound goes down by $2$ (or by $4$ if we impose both constraints at the same time).
\end{lemma}

\begin{proof}
Start with just one simple sphere $v$. If $v$ is not contained in $D$, it belongs to a space of dimension $2n - 2 + 2(v \cdot D)$, so \eqref{eq:leq} is an equality. Imposing constraints $v(-\infty) \in D$ or $v(+\infty) \in D$ lowers the dimension by $2$ each. Finally, if $v$ is contained in $D$, it belongs to a space of dimension $\mathrm{dim}(D) - 2 \leq 2n + 2(v \cdot D) - 6$.

Take a chain of the kind described in the Lemma, and add an extra component $v$, with $v(-\infty) = v^K(+\infty)$, to its end. We go through the effect on the dimension case-by-case:
\begin{itemize} \itemsep.5em
\item {\em If $v(+\infty) \notin D$,} the dimension of the moduli space increases by $2(v \cdot D) - 2$. Note that if the original chain had its endpoint $v^K(+\infty) \in D$, then we lose that property after adding the extra component.

\item {\em If $v(+\infty) \in D$ but $v$ is not contained in $D$,} the dimension increases only by $2(v \cdot D) - 4$, due to the extra constraint. The new chain also has endpoint in $D$, irrespective of whether that was true for the original one.

\item {\em If $v$ is contained in $D$}, the dimension decreases by $2$, and of course both the original and new chain have endpoints in $D$.
\end{itemize}
This means that the inequalities from (i) and (ii) are inherited by the longer chain.
\end{proof}

Returning to our main moduli space, the relevant transversality requirements are as follows:
\begin{enumerate}[label=(S\arabic*)] \itemsep.5em
\parindent0em \parskip.5em
\item {\em (Main stratum)}
Consider the subspace of $\frakS_m(x_-,c_+)$ where the marked points are pairwise distinct, and none are equal to $+\infty$. We assume that this is regular. Its dimension is then $\mathrm{deg}(x_-) - \mathrm{deg}(c_+) + 2m$, where $\mathrm{deg}(c_+)$ is the Morse index. 

\item \label{item:s-collision}
{\em (Collision, no marked point at $+\infty$)}
This is the analogue of \ref{item:d-collision}, adding the condition that no marked point should lie at $+\infty$; the dimension is $\mathrm{deg}(x_-) - \mathrm{deg}(c_+) + 2|\Pi|$.

\item \label{item:s-simple-bubbling}
{\em (Bubbling without collision, no marked point at $+\infty$)}
This is the analogue of \ref{item:d-simple-bubble}, see Section \ref{subsec:dm}, again with the added requirements that no marked point should lie at $+\infty$; one gets dimension $\leq \mathrm{deg}(x_-) - \mathrm{deg}(c_+) + 2m - 2$.

\item \label{item:s-goes-to-infinity}
{\em (Marked point at $+\infty$)} 
We again assume that the marked points should collide according to a partition $\Pi$ (which can be the trivial one where they are all distinct), and that no intersections with $D$ should occur outside the set of such points. We also require that $\Sigma_{+\infty} > 0$, meaning that some marked points should lie at $+\infty$. This yields dimension $\mathrm{deg}(x_-) - \mathrm{deg}(c_+) + 2|\Pi| - 2$. To be precise, there are actually two sub-cases here. On the stratum where $\mu_{+\infty}(u) = 0$, so that $u(+\infty) \notin D$, the requirement that $u(+\infty) = b(0)$ is an intersection condition with the stable manifold (inside $M$); if on the other hand $\mu_{+\infty}(u) > 0$, meaning that $u(+\infty) \in D$, then $u(+\infty)$ must lie on the intersection of the stable manifold and $D$ (inside $D$). Because of the assumption that the stable manifolds are transverse to $D$, one gets the same dimension in both cases.

\item \label{item:s-messy-dimension}
{\em (Incidence condition switches to bubble chain)} We have a partition as before. In addition we have a simple chain $(v^1,\dots,v^K)$, $K>0$, of pseudo-holomorphic bubbles, for the almost complex structure which appears in our thimble equation at the point $+\infty$, connecting $u(+\infty)$ and $b(0)$. This means that we have \eqref{eq:chain-incidence} as well as
\begin{equation} \label{eq:endpoint-incidence}
u(+\infty) = v^1(-\infty), \;\; v^K(+\infty) = b(0).
\end{equation}
Finally, the bubble chain should satisfy
\begin{equation} \label{eq:degree-of-chain}
\sum_{k=1}^K (v^k \cdot D) \leq \Sigma_{+\infty}.
\end{equation}
This situation is complicated enough to warrant spelling out the dimension count:
\begin{equation} \label{eq:dimension-array}
\begin{array}{l|l}
\text{index of the linearized operator for $u$} & \mathrm{deg}(x_-) \\
\hspace{1em} -\text{codimension of the incidence with $D$} & 
 \\ \hline
\hline
\text{dimension of the parameter space} & \leq 2(m - \Sigma_{+\infty})
\\ \hline
\text{dimension of the simple chain} 
 & \leq -\mathrm{deg}(c_+) - 2K + 2\sum_k (v^k \cdot D)
\\ 
\hspace{1em} \text{$-$codimension of \eqref{eq:endpoint-incidence}}
& 
\\ \hline
\text{total dimension} & 
\leq \mathrm{deg}(x_-) - \mathrm{deg}(c_+) 
\\ & \hspace{1em} + 2\big(m - \Sigma_{+\infty} - K + \sum_k (v_k \cdot D)\big).
\end{array}
\end{equation}
Here, we have used Lemma \ref{th:chain-dimension} (actually both parts of the Lemma, because for $\mu_{+\infty}(u) > 0$ the incidence $u(+\infty) = v^1(-\infty)$ takes place in $D$). From \eqref{eq:degree-of-chain} we conclude that the total dimension from \eqref{eq:dimension-array} is 
\begin{equation} \label{eq:complicated-dimension}
\leq \mathrm{deg}(x_-) - \mathrm{deg}(c_+) + 2(m - K).
\end{equation}
\end{enumerate}


The analogue of Proposition \ref{th:floer} is this:

\begin{proposition} \label{th:pss}
Under the regularity assumptions imposed above, consider spaces $\frakS_m(x_-,c_+)$ of dimension $\mathrm{deg}(x_-) - \mathrm{deg}(c_+) + 2m \leq 1$.

(i) All points in such a space consist of curves where the marked points are pairwise distinct, and no marked point lies at $+\infty$.

(ii) If the dimension is $0$, we have $\bar\frakS_m(x_-,c_+) \setminus \frakS_m(x_-,c_+) = \emptyset$.

(iii) If the dimension is $1$, the points in $\bar\frakS_m(x_-,c_+) \setminus \frakS_m(x_-,c_+)$ are of two kinds. In one kind, the Riemann surface splits into two pieces, one a cylinder and the other a thimble. The marked points are still pairwise distinct; none of them lie at $+\infty$; and there are no further components or bubbles. In the other kind, a pseudo-gradient trajectory splits off from $b$.
\end{proposition}

\begin{proof}
Part (i) is again easy, since the points which violate that condition fall under \ref{item:s-collision} or \ref{item:s-goes-to-infinity}. The proof of (ii), (iii) is based on Lemma \ref{th:thimble-1}, or rather its generalization along the same lines as in Lemma \ref{th:complicated}. Let's look at what can happen to the thimble component $u^I$ of the Gromov limit, and its incidence condition with the half-line $b^1$:
\begin{itemize} \itemsep.5em
\item 
{\em Suppose that $u^I(+\infty) \neq b^1(0)$}. Take the bubble tree connecting $u^I(+\infty)$ and $b^1(0)$, and apply a standard simplifying operation to get a simple chain. Forget all bubbles attached to other points of this component. By \ref{item:s-messy-dimension} above, the result (consisting of $u^I$, the simplified bubble chain, and $b^1$) belongs to a space of codimension $\geq 2$, compared to what happens for a generic thimble-and-half-line with the same limits and intersection number with $D$.

\item
{\em Suppose that there is at least one marked point at $+\infty$, and that $u^I(+\infty) = b^1(0)$.} We then ignore all bubbles attached to $u^I$, and end up in the situation \ref{item:s-goes-to-infinity}, which yields codimension $\geq 2$.

\item
{\em Suppose that there is no marked point at $+\infty$, and that $u^I(+\infty) = b^1(0)$, but some marked points collide.} We again ignore all bubbles, and then \ref{item:s-collision} yields codimension $\geq 2$.

\item
{\em Finally, suppose there is no marked point at $+\infty$, that the marked points remain pairwise distinct, that $u^I(+\infty) = b^1(0)$, but that we have bubbling.} We forget all bubbles except one, and then codimension $2$ follows from \ref{item:s-simple-bubbling}.
\end{itemize}
Given that, the rest of the argument involves doing the same for the cylindrical components, and then adding up the resulting dimensions; both those parts follow the proof of Lemma \ref{th:floer} exactly, so we omit them.
\end{proof}

\subsection{Thimbles with tangency constraints\label{section:thimbleswithtangency}} 
The next construction follows \cite{ganatra-pomerleano21, ganatra-pomerleano20, pomerleano21} (the last one being closest to the argument here) by looking at thimbles with a fixed order of tangency to the divisor. We start with a Morse-Smale pair $(f_D,X_D)$ for the divisor $D$. Write $(\mathit{CM}^*(D),d_D)$ for the Morse complex.
\begin{itemize}
\item Fix $w \geq 1$. We consider maps from the thimble to $M$, which at $+\infty$ have $w$-fold intersection multiplicity with $D$, and elsewhere are disjoint from $D$. These are coupled with half-infinite flow lines $b$ by the incidence condition \eqref{eq:uv}, but taking place in $D$. The outcome is a chain map
\begin{equation}
\xymatrix{
\mathit{CF}^*(w)
&&& \ar[lll]_-{\includegraphics[valign=c]{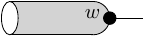}}^{\displaystyle t_{w,0}} \mathit{CM}^*(D).
}
\end{equation}

\item
More generally, for $w>0$ and $m \geq 0$, we can consider thimbles which have additional intersection points with $D$, amounting to a total $m+w$ intersection number, and where the $s \ll 0$ behaviour of the equation on the thimble now follows the Floer equation for $H_{w+m}$. 
This gives maps
\begin{equation} \label{eq:tt}
\xymatrix{
\mathit{CF}^{*-2m}(w+m)
&&& \ar[lll]_-{\!\!\!\overbrace{\includegraphics[valign=c]{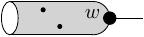}\hspace{-3em}}^{\text{$m$ added points}}\hspace{3em}}^{\displaystyle t_{w,m}} \mathit{CM}^*(D),
}
\end{equation}
satisfying
\begin{equation} \label{eq:t-map-equation}
\sum_{i+j=m} d_i t_{w,j} = t_{w,m} d_D.
\end{equation}
\end{itemize}
One assembles these into chain maps
\begin{equation} \label{eq:t-chain}
t_{C_q,w} = \sum_m q^m t_{w,m}: \mathit{CM}^*(D) \longrightarrow C_q.
\end{equation}
The induced map $H^*(D) \rightarrow \mathit{SH}^*_q(M,D)$ is the $z^w$ component of \eqref{eq:main}.

The parameter space used here is the same as for $s_m$, but since the continuation maps it parametrizes are a priori different, we choose to distinguish it notationally, as $\frakT_{w,m} = \mathit{Sym}_m(T)$. When considering maps $u: T \rightarrow M$, we now require that
\begin{equation} \label{eq:sigma-plus}
u^{-1}(D) = \Sigma + w\{+\infty\} \;\;\Leftrightarrow\;\; \mu_z(u) = \begin{cases}
\Sigma_z & z \neq +\infty, \\
\Sigma_{+\infty} + w & z = +\infty.
\end{cases}
\end{equation}
Denote the resulting space by $\frakT_{w,m}(x_-,c_+)$, and its Gromov compactification by $\bar\frakT_{w,m}(x_-,c_+)$. We need an observation in the spirit of Lemma \ref{th:bubble}:

\begin{lemma} \label{th:bubble-2}
Consider a Gromov-convergent sequence in $\frakT_{w,m}(x_-,c_+)$. Suppose that in the limit of the associated sequence in $\frakT_{w,m}$, there are $G-1$ marked points at $+\infty$. In our Gromov limit, take the domains of all bubbles attached to $+\infty$, and glue them to a single nodal Riemann surface. Then, the preimage of $D$ in that nodal surface has at most $G$ connected components.
\end{lemma}
\begin{figure}
\begin{centering}
\includegraphics{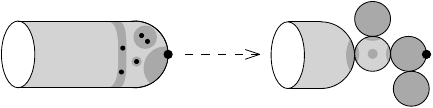}
\caption{\label{fig:bubble-tree}Four of the regions $U$ (shaded more darkly) from the proof of Lemma \ref{th:bubble-2}, with the corresponding $U_k$.}
\end{centering}
\end{figure}

\begin{proof}
In the Gromov limit, take the thimble, together with all the bubbles, and consider them as a nodal Riemann surface $C$, with a map $u: C \rightarrow M$. Let $K \subset C$ be a connected component of $u^{-1}(D)$ lying inside the bubble tree at $+\infty$. Choose a small open neighbourhood $K \subset U \subset C$. By assumption, the image of $\partial U$ avoids $D$, hence the intersection number $(u|U) \cdot D$ is well-defined, and positive. 

Let's look at our original sequence, consisting of divisors $\Sigma_k \subset T$ and maps $u_k: T \rightarrow M$. To the region $U$ corresponds (not uniquely, but sufficiently so for the subsequent argument; see Figure \ref{fig:bubble-tree}) a sequence of regions $U_k \subset T$, for $k \gg 0$, such that $(u_k|U_k) \cdot D = (u|U) > 0$. Hence, $U_k$ must contain either a point of $\Sigma_k$ or the point $+\infty$. Moreover, because of the definition of $U$, the subsets $U_k$ will be contained in a neighbourhood of $+\infty \in T$, which can be made arbitrarily small.

If we have several different connected components $K$, we get pairwise disjoint $U$, hence also $U_k$. Each such region, for $k \gg 0$, must contain either $+\infty$, or one of the marked points which go to $+\infty$ in the limit; which explains the bound.
\end{proof}

Lemma \ref{th:bubble-2} constrains the pattern of intersections between a bubble tree and $D$. To apply this, we need to show that such constraints survive the simplification process which enters into transversality arguments.

\begin{lemma} \label{th:simplify}
Let $(v^1,\dots,v^K)$ be a chain of $J$-holomorphic spheres, meaning that it satisfies \eqref{eq:chain-incidence}, and such that $v^1(-\infty) \neq v^K(+\infty)$. Then there is a simple chain $(\bar{v}^1,\dots,\bar{v}^{\bar{K}})$ with $\bar{v}^1(-\infty) = v^1(-\infty)$, $\bar{v}^{\bar{K}}(+\infty) = v^K(+\infty)$, such that the following holds. Think of $(v^1,\dots,v^K)$ as a single map $v$ defined on the nodal surface obtained by gluing $+\infty$ in each sphere to $-\infty$ in the next one, and similarly $\bar{v}$ for the simple chain. Then, the number of connected components of $\bar{v}^{-1}(D)$ is less or equal than that for $v^{-1}(D)$.
\end{lemma}

\begin{proof}
There's an explicit construction, which proceeds in the following steps.
\begin{itemize} \itemsep.5em
\item If $v^i(-\infty) = v^j(+\infty)$ for some $i \leq j$, we remove $v^i,\dots,v^j$ from our chain. Having repeated that as often as necessary, the outcome is that the nodal points and endpoints in our chain will map to pairwise distinct points in $M$. 
\item With the previous step in mind, assume that  $v^1(-\infty),v^2(-\infty),\dots,v^K(-\infty),v^K(+\infty)$ were already pairwise different (which also implies that no component can be constant). At this point, we replace each multiply-covered sphere $v^i$ by the underlying simple one $\bar{v}^i$, parametrized in such a way that $\bar{v}^i(-\infty) = v^i(-\infty)$, $\bar{v}^i(+\infty) = v^i(+\infty)$ (by assumption, there are distinct points on the domain of the simple curve which map to $v^i(-\infty)$ and $v^i(\infty)$, so the required parametrization always exists).
\item As before, let's suppose that the previous steps are already done. This means that $v^1(-\infty),v^2(-\infty),\dots,v^K(-\infty),v^K(+\infty)$ are pairwise different, and each $v^i$ is a simple pseudo-holomor\-phic map. Suppose that for some $i<j$, the maps $v^i$ and $v^j$ have the same image. We then remove $v^{i+1},\dots,v^j$ from our chain, and replace $v^i$ by a reparametrized version $\tilde{v}^i$, such that $\tilde{v}^i(-\infty) = v^i(-\infty)$, $\tilde{v}^i(+\infty) = v^j(+\infty)$ (this is possible because $v^i(-\infty) \neq v^j(+\infty)$, as before). One repeats that until the simple chain condition is satisfied.
\end{itemize}
One easily sees that the number of connected components of the preimage of $D$ cannot increase at any step of the algorithm.
\end{proof}

\begin{remark}
The simplification process we have described does not necessarily lead to the shortest possible simple chain with given endpoints. As an illustration, suppose we start with simple curves $(v^1,v^2,v^3)$, where $(v^1)^{-1}(D) = +\infty$, $(v^3)^{-1}(D) = -\infty$, and such that $v^1,v^3$ intersect at some point $p$ outside $D$; also, $v^2 \subset D$; otherwise, we assume the situation is as generic as possible (no other intersection or selfintersection points). One could remove $v^2$ from the chain, and take $p$ to be the image of the nodal point of the newly simplified chain. In that case $v^{-1}(D)$ has one connected component (the domain of $v^2$), whereas the corresponding preimage for the simplified chain has $2$ ($+\infty$ in the domain of $v^1$, and $-\infty$ in the domain of $v^3$). However, the procedure from the proof of Lemma \ref{th:simplify} would not remove $v^2$, hence avoids that problem.
\end{remark}

We use the arguments above for dimension calculations, via the following variant of Lemma \ref{th:chain-dimension}:

\begin{lemma} \label{th:chain-dimension-2}
Let $(v^1,\dots,v^K)$ be a simple chain of pseudo-holomorphic spheres. Suppose that when we join together the domains (each to the next) to a nodal Riemann surface, the preimage of $D$ in that nodal surface has $G$ connected components. 

(i) Assuming regularity, that chain belongs to a moduli space of dimension
\begin{equation} \label{eq:g-formula}
\leq 2n - 2K + 2G.
\end{equation}
More precisely, in this moduli space, we keep the nature of the intersections with $D$ fixed (which components lie inside it, and the orders of tangency for the other components), and again divide by $\bC^*$ acting on each component.

(ii) If one additionally requires that $v^1(-\infty) \in D$ or $v^K(+\infty) \in D$, the dimension bound decrease by $2$ (or by $4$ if we impose both constraints at the same time).
\end{lemma}

\begin{proof}
Take a simple sphere $v$ not contained in $D$. Consider its domain as a sphere with $G$ marked points, and fixed intersection multiplicities with $D$ at those points, up to $\bC^*$ reparametrizations. The moduli space correspondingly has dimension $2n - 2 + 2G$. If one requires that $v(-\infty) \in D$, that constrains the position of one of the marked points, lowering the dimension by $2$, and the same for $v(+\infty)$. Finally, for a sphere that is contained in $D$, we have $G = 1$ and dimension $(2n-2)-2 = 2n - 6 + 2G$. The rest of the proof is the same add-a-sphere inductive argument as in Lemma \ref{th:chain-dimension}; we omit the details.
\end{proof}


Transversality provides the following properties, for generic choices:
\begin{enumerate}[label=(T\arabic*)] \itemsep.5em
\item \label{item:t-main}
{\em (Main stratum)} 
The subspace of $\frakT_{w,m}(x_-,c_+)$ where the marked points are pairwise distinct, and none are equal to $+\infty$, is of dimension $\mathrm{deg}(x_-) - \mathrm{deg}(c_+) + 2m$.

\item \label{item:t-collision} 
{\em (Collision, no marked point at $+\infty$)} 
This is as in \ref{item:s-collision}. We get the usual dimension $\mathrm{deg}(x_-) - \mathrm{deg}(c_+) + 2|\Pi|$, where $|\Pi| < m$.

\item \label{item:t-bubbling}
{\em (Bubbling at marked point, without collision, and with no marked point at $+\infty$)} This is as in \ref{item:s-simple-bubbling}, of dimension $\mathrm{deg}(x_-) - \mathrm{deg}(c_+) + 2m-2$.

\item \label{item:t-bubbling-2}
{\em (Bubbling at $+\infty$, without collision, and with no marked point at $+\infty$)} Here, we consider our map satisfying $u(+\infty) = b(0)$; in addition there is a simple sphere bubble $v$ going through $u(+\infty)$, and such that either the bubble is contained in $D$, or $v^{-1}(D)$ consists of a single point. The first case has dimension $\mathrm{deg}(x_-) - \mathrm{deg}(c_+) + 2m-4$, and the second one is of dimension $2$ higher.

\item \label{item:t-goes-to-infinity}
{\em (Marked point goes to $+\infty$, incidence condition is preserved)} This is the analogue of \ref{item:s-goes-to-infinity}, with $\Sigma_{+\infty} > 0$. It yields a space of dimension $\mathrm{deg}(x_-) - \mathrm{deg}(c_+) - 2|\Pi| - 2$.

\item \label{item:t-messy-dimension}
{\em (Incidence condition switches to bubble chain)} This is the analogue of \ref{item:s-messy-dimension}, with a simple bubble chain \eqref{eq:chain-incidence}, \eqref{eq:endpoint-incidence}. We additionally require that the number of components in the sense of Lemma \ref{th:chain-dimension-2} should be
\begin{equation} \label{eq:g-bound}
G \leq \Sigma_{+\infty} + 1. 
\end{equation}
The appropriate version of \eqref{eq:dimension-array}, using Lemma \ref{th:chain-dimension-2}(ii) since the endpoint of the chain always lies in $D$, is:
\begin{equation} \label{eq:dimension-array-2}
\begin{array}{l|l}
\text{index of the linearized operator for $u$} & \mathrm{deg}(x_-) \\
\hspace{1em} -\text{codimension of the incidence with $D$} & 
 \\ \hline
\hline
\text{dimension of the parameter space} & \leq 2(m - \Sigma_{+\infty})
\\ \hline
\text{dimension of the simple chain} 
 & \leq -\mathrm{deg}(c_+) - 2K + 2G - 2
\\ 
\hspace{1em} \text{$-$codimension of \eqref{eq:endpoint-incidence}}
& 
\\ \hline
\text{total dimension} & 
\leq \mathrm{deg}(x_-) - \mathrm{deg}(c_+) 
\\ & \hspace{1em} + 2\big(m - \Sigma_{+\infty} - K + G - 1\big).
\end{array}
\end{equation}
From \eqref{eq:g-bound} we then get the same inequality \eqref{eq:complicated-dimension} as in our previous discussion.
\end{enumerate}

\begin{proposition} \label{th:bs}
Under the regularity assumptions imposed above, consider spaces $\frakT_{w,m}(x_-,c_+)$ of dimension $\mathrm{deg}(x_-) - \mathrm{deg}(c_+) + 2m \leq 1$.

(i) All points in such a space consist of curves where the marked points are pairwise distinct, and no marked point lies at $+\infty$.

(ii) If the dimension is $0$, we have $\bar\frakT_{w,m}(x_-,c_+) \setminus \frakT_{w,m}(x_-,c_+) = \emptyset$.

(iii) If the dimension is $1$, the points in $\bar\frakT_{w.m}(x_-,c_+) \setminus \frakT_{w,m}(x_-,c_+)$ are of the same kind as in Proposition \ref{th:pss}.
\end{proposition}

\begin{proof}
(i) follows from \ref{item:t-collision} and \ref{item:t-goes-to-infinity}. Concerning (ii) and (iii), as in the proof of Proposition \ref{th:pss}, we will consider only the principal component $u^I$ of the limit, together with its bubbles and the pseudo-gradient half-line $b^1$.
\begin{itemize} \itemsep.5em 
\item {\em Suppose that $u^I(+\infty) \neq b^1(0)$.} Take the bubble tree joining $u^I$ and $b^1$. Lemma \ref{th:bubble-2} says that if we consider the bubble tree as a single nodal curve, then the preimage of $D$ in that curve has $\leq \Sigma_{+\infty} + 1$ connected components. By keeping only those components that separate the thimble and half-line, one can reduce the tree to a chain of spheres, without increasing the number of connected components of the preimage of $D$. Finally, apply Lemma \ref{th:simplify} to the chain. The outcome is that one has a simple chain connecting $u^I(+\infty)$ and $b^1(0)$, still with the same bound on the preimage of $D$. We are now in the situation of \ref{item:t-messy-dimension}, which yields codimension $2$.

\item {\em Suppose that at least one of the marked points lies at $+\infty$, and that $u^I(+\infty) = b^1(0)$.} We ignore bubbles and get \ref{item:t-goes-to-infinity}.

\item {\em Suppose that no marked point lies at $+\infty$, that $u^I(+\infty) = b^1(0)$, but that some marked points collide.} This leads to \ref{item:t-collision}.

\item {\em Suppose that no marked point lies at $+\infty$; that $u^I(+\infty) = b^1(0)$; that the marked points remain pairwise distinct; but that bubbling occurs at a marked point.} This is \ref{item:t-bubbling}.

\item {\em Suppose that no marked point lies at $+\infty$; that $u^I(+\infty) = b^1(0)$; that the marked points remain pairwise distinct; and bubbling occurs at $+\infty$.} It follows from Lemma \ref{th:bubble-2} that the preimage in $D$ in the resulting bubble tree must be connected. Because $u^I(+\infty) \in D$, the nonconstant component in the tree closest to the thimble is either contained in $D$, or else intersects $D$ exactly at the point where it's attached as part of the tree. After replacing that component with the underlying simple map, we are in situation \ref{item:t-bubbling-2}.
\end{itemize}
\end{proof}

\begin{remark}
As the reader will have noticed, our compactifications are low-tech (compared to the ones from relative Gromov-Witten theory, used for a similar purpose in \cite{pomerleano21}); on top of that, we use the process of simplifying bubble chains (Lemma \ref{th:bubble-2}), which loses a lot of information. In spite of that, keeping track of the number of connected components of the preimage of $D$ has allowed us to retain just enough control to reach the necessary dimension bounds.
\end{remark}

\section{The action filtration\label{sec:action}}
The Floer differential always increases the action, but a general continuation map may decrease it by a bounded amount. We will discretize the action filtration, by arranging that the action values of one-periodic orbits cluster near integers. The gaps between those clusters afford enough flexibility to construct filtered continuation maps. The maps obtained from thimbles can also be shown to be compatible with the filtration.

\subsection{Constructing the filtrations}
We begin by defining the (entirely elementary) classes of Hamiltonians which give rise to filtered Floer complexes.

\begin{definition}
Fix some $\epsilon \leq 1/2$. A time-dependent Hamiltonian $\bar{H}$ 
is called $\epsilon$-bounded if it has the following additional properties. First, 
\begin{equation} \label{eq:h-bound}
|\bar{H}| < \epsilon/2
\end{equation}
everywhere. Secondly, every one-periodic orbit $x$ not lying in $D$ bounds a surface 
\begin{equation} \label{eq:area-bound}
y: S \rightarrow M, \;\; \Big|\int_S y^*\omega_M\Big| < \epsilon/2. 
\end{equation}
\end{definition}

As a consequence of \eqref{eq:h-bound} and \eqref{eq:area-bound}, the actions \eqref{eq:action} satisfy
\begin{equation} \label{eq:clustered-action}
A(x) \in (k-\epsilon,k+\epsilon), \;\;  \text{ where } k = y \cdot D \in \bZ \;\; \text{for $y$ as in \eqref{eq:area-bound}.}
\end{equation}
We say that $x$ has action {\em approximately $k$}. Suppose that $u$ is a solution of Floer's equation, with limits $x_{\pm}$ outside $D$. From the relevant special case of \eqref{eq:energy-2},
\begin{equation} \label{eq:approximate-action-increase}
\begin{aligned}
& x_+ \text{ has action approximately $k_+$} 
\\ & \qquad \Longrightarrow 
x_- \text{ has action approximately $k_-$, for some $k_- \geq k_+ - (u \cdot D)$.}
\end{aligned}
\end{equation}

\begin{definition} \label{th:epsilon-continuation}
Suppose that $\bar{H}_{\pm}$ 
are $\epsilon_{\pm}$-bounded, for $\epsilon_- + \epsilon_+ \leq 2/3$. Take
\begin{equation} \label{eq:continuation-epsilon}
\epsilon \in [(\epsilon_- + \epsilon_+)/2, 1 - \epsilon_+ - \epsilon_-]. 
\end{equation}
A continuation map Hamiltonian $H = (H_{s,t})$ relating $\bar{H}_{\pm}$ is called $\epsilon$-bounded if 
\begin{equation} \label{eq:almost-monotone}
\int_{-\infty}^{\infty} \max \{\partial_s H_{s,t} \, : \, (x,t) \in M \times S^1 \} \; \mathit{ds} < \epsilon.
\end{equation}
\end{definition}

Suppose that \eqref{eq:almost-monotone} holds, and consider the associated continuation map equation. The counterpart of \eqref{eq:approximate-action-increase}, again using \eqref{eq:energy-2}, says that
\begin{equation} \label{eq:approximate-action-increase-2}
\begin{aligned}
& x_+ \text{ has action approximately $k_+$} \Longrightarrow A(x_+) > k_+ - \epsilon_+ \\
& \qquad \Longrightarrow A(x_-) > (k_+ - \epsilon_+) - (u \cdot D) - \epsilon
\geq 
k_+ - (u \cdot D) - 1 + \epsilon_-
\\ & \qquad 
\Longrightarrow x_- \text{ has action approximately $k_-$, for some $k_- \geq k_+ - (u \cdot D)$.}
\end{aligned}
\end{equation}

\begin{lemma} \label{th:satisfy-epsilon-bound}
(i) For (any slope $\sigma$ and) any $\epsilon>0$, there are $\epsilon$-bounded Hamiltonians, for which the one-periodic orbits lying outside $D$ are nondegenerate.

(ii) In (i) one can additionally achieve that: the discs $y$ from \eqref{eq:area-bound} have $y \cdot D \in [-\lfloor \sigma \rfloor,0]$; and the one-periodic orbits lying outside $D$ have $\mathrm{deg}(x) \in \{0,\dots,2n-1\}$.

(iii) For (any slopes, any $\epsilon_{\pm}$-bounded $\bar{H}_{\pm}$, and any) $\epsilon$ as in \eqref{eq:continuation-epsilon}, there is an $\epsilon$-bounded continuation Hamiltonian.
\end{lemma}

\begin{proof}
(i) It is a standard fact that (with our normalization) an orbit of the $S^1$-action lying in a given level set of $h$ bounds a disc whose symplectic area is that value of $h$. Fix some function $\psi$ such that:
\begin{equation} \label{eq:slope-function}
\left\{
\begin{aligned}
& \psi(0) < \epsilon/2, \\
& \psi'(r) = -\sigma \text{ for sufficiently small $r$}, \\
& \psi''(r) \geq 0 \text{ everywhere, and $>0$ at all points where $-\psi'(r)$ is a positive integer},  \\
& \psi(r) = 0 \text{ for $r \geq \epsilon/2\sigma$.}
\end{aligned}
\right.
\end{equation}
Set $H = \psi(\mu)$, extended by zero away from $\{h \leq \epsilon/2\sigma\}$, which clearly satisfies \eqref{eq:h-bound}. The one-periodic orbits of $H$ lying outside $D$ are: constant orbits in $h^{-1}(0)$; and for $w \in \{1,\dots,\lfloor \sigma \rfloor\}$, $(-w)$-fold multiples of $S^1$-orbits, lying on the level set where $\psi'(h) = -w$. The latter bound discs of symplectic area $-wh \in (-w \epsilon/2\sigma,0)$, hence satisfy \eqref{eq:area-bound}. A generic time-dependent perturbation, within the class of Hamiltonians with slope $\sigma$, makes the one-periodic orbits nondegenerate, without losing any of the desired properties.

(ii) For an $(-w)$-fold multiple of an $S^1$-orbit lying on a given level set of $h$, the bounding disc of symplectic area $-kh$ has intersection number $-w$ with $D$. The rest is standard Morse-Bott perturbation theory.

(iii) At every point $(t,x)$, the values of $\bar{H}_{\pm}$ differ by less than $(\epsilon_+ + \epsilon_-)/2$. So one can just interpolate $s$-dependently between the two. 
%
%
\end{proof}

Let's first consider only Floer trajectories disjoint from $D$. Then, the special case $u \cdot D = 0$ of \eqref{eq:approximate-action-increase} says that for any $\epsilon$-bounded Hamiltonian, the Floer differential ($d_0$ in the notation from Section \ref{subsec:sh}) preserves the decreasing filtration
\begin{equation} \label{eq:k-filtration}
F^{\geq K} \mathit{CF}(\bar{H}) = \{\text{subspace generated by $x$ with action approximately $k \geq K$}\}, \;\; K \in \bZ.
\end{equation}
Similarly, if $H$ is $\epsilon$-bounded, the resulting continuation map $d_0^\dag$ is filtered. 

\begin{lemma} \label{th:quarter}
(i) Consider filtered Floer chain complexes defined using $(1/3)$-bounded Hamiltonians, for some $\epsilon \leq 1/3$. If we have two such complexes, with slopes $\lfloor \sigma_- \rfloor \geq \lfloor \sigma_+ \rfloor$, one can define a filtered continuation map relating them, in a way which is unique up to filtered chain homotopy.

(ii) If we use $(1/4)$-bounded Hamiltonians, the class of continuation maps from (i) is closed under composition, again up to filtered chain homotopy.
\end{lemma}

\begin{proof}
(i) We use $(1/3)$-bounded continuation map equations, which exist by Lemma \ref{th:satisfy-epsilon-bound}; the same idea, applied in a parametrized way, yields a filtered chain homotopy between any two such continuation maps.

(ii) Given two such Hamiltonians, we first use $(1/4)$-bounded continuation map equations to relate them. Now suppose we are given three Hamiltonians $\bar{H}_-$, $\bar{H}_0$, $\bar{H}_+$ with slopes $\lfloor \sigma_- \rfloor \geq \lfloor \sigma_0 \rfloor \geq \lfloor \sigma_+ \rfloor$. Use $(1/4)$-bounded continuation map equations to define maps 
\begin{equation} \label{eq:filtered-composition}
\mathit{CF}(\bar{H}_+) \longrightarrow \mathit{CF}(\bar{H}_0) \longrightarrow \mathit{CF}(\bar{H}_-)
\end{equation}
using $(1/4)$-bounded Hamiltonians. Gluing those together yields a $(1/2)$-bounded continuation Hamiltonian which equals $\bar{H}_{\pm}$ at the ends. This still falls into the interval \eqref{eq:continuation-epsilon}, hence the composition of the two maps \eqref{eq:filtered-composition} is filtered chain homotopic to the continuation map $\mathit{CF}(\bar{H}_+) \rightarrow \mathit{CF}(\bar{H}_-)$ obtained directly from a $(1/4)$-bounded Hamiltonian.
\end{proof}

In particular, up to filtered chain homotopy equivalence, the Floer complex $\mathit{CF}(\bar{H})$ defined using a $(1/4)$-bounded Hamiltonian depends only on $\lfloor \sigma \rfloor$.

\begin{lemma} \label{th:parts-of-the-filtration}
Considered filtered Floer complexes, and continuation maps, as in Lemma \ref{th:quarter}.

(i) $F^{\geq K}\mathit{CF}(\bar{H})$ is contractible, for every $K>0$.

(ii) the inclusion $F^{\geq K}\mathit{CF}(\bar{H}) \hookrightarrow \mathit{CF}(\bar{H})$ is a chain homotopy equivalence, for every $K \leq -\lfloor \sigma \rfloor$.

(ii) The filtered continuation map restricts to a chain homotopy equivalence $F^{\geq K}\mathit{CF}(\bar{H}_+) \longrightarrow F^{\geq K}\mathit{CF}(\bar{H}_-)$, for every $K \geq -\lfloor \sigma_+ \rfloor$.
\end{lemma}

\begin{proof}
(i) By Lemma \ref{th:satisfy-epsilon-bound}(ii), there is a particular choice of Hamiltonian such that this subcomplex is zero. By Lemma \ref{th:quarter}, it is therefore a chain homotopy equivalence for all Hamiltonians.

(ii) The argument is the same as in (i), except that Lemma \ref{th:satisfy-epsilon-bound}(ii) is now used to show that for some choice of Hamiltonian, the inclusion is an isomorphism.

(iii) As before, Lemma \ref{th:quarter} tells us that if this true for some choice of $\bar{H}_{\pm}$, then it is true for all with the same values of $\lfloor \sigma_{\pm} \rfloor$. The statement is trivial if those values are the same, so we may assume $\lfloor \sigma_- \rfloor > \lfloor \sigma_+ \rfloor$.
The rest is a modification of the proof of Lemma \ref{th:satisfy-epsilon-bound}(ii). Start with the slope $\sigma_-$ and fix a function $\psi_-$ as in \eqref{eq:slope-function}, with the added property that there $\psi'_-(r) = -\sigma_+$ in some interval around a small value $r = r_0$. Consider the modified function $\psi_+$ which satisfies
\begin{equation}
\left\{
\begin{aligned}
& \psi_+(r) = \psi_-(r) \text{ for $r \geq r_0$,} \\
& \psi_+'(r) = -\sigma_+ \text{ for $r \leq r_0$.}
\end{aligned}
\right.
\end{equation}
Starting with those, one can define $\epsilon$-bounded Hamiltonians $\bar{H}_{\pm}$ with slopes $\sigma_{\pm}$, such that:
\begin{itemize} \itemsep.5em
\item $\bar{H}_- \geq \bar{H}_+$ everywhere, with equality outside the region $\{\mu < r_0\}$.
\item The one-periodic orbits of $\bar{H}_+$ have $y \cdot D \in [-\lfloor \sigma_+ \rfloor,0]$.
\item The one-periodic orbits of $\bar{H}_-$ lying in $\{\mu < r_0\}$ have $y \cdot D \in [-\lfloor \sigma_- \rfloor, -\lceil \sigma_+ \rceil]$.
\end{itemize}
Because of the first property, one can find a continuation map Hamiltonian $H$ with $\partial_s H \leq 0$. By \eqref{eq:energy} any solution $u$ of the resulting continuation map equation, which remains outside $D$, must satisfy $A(x_-) \geq A(x_+)$, with equality iff it is $s$-independent. Those $s$-independent solutions are regular, and imply that the induced map 
\begin{equation}
F^{\geq -\lfloor \sigma_+ \rfloor} \mathit{CF}(\bar{H}_+) \longrightarrow
F^{\geq -\lfloor \sigma_+ \rfloor} \mathit{CF}(\bar{H}_-)
\end{equation}
is an isomorphism. Since that map is filtered, the result for all $K \geq -\lfloor \sigma_+ \rfloor$ follows.
\end{proof}

%
%

For the version of Floer complexes with an extra variable $q$, we extend the notion of action by
\begin{equation} \label{eq:q-action}
A(x q^j) = A(x) + j.
\end{equation}

\begin{lemma} \label{th:filtered-q-telescope}
In the deformed telescope construction $C_q$, one can make the following choices:
\begin{itemize} \itemsep.5em
\item
For each $w \geq 0$, choose $\bar{H}_w$ to be $2^{-w}/6$-bounded. Additionally, it should satisfy the properties $y \cdot D \leq 0$ and $\mathrm{deg}(x) \geq 0$ from Lemma \ref{th:satisfy-epsilon-bound}(ii).

\item
Choose all continuation map Hamiltonians relating $\bar{H}_{w_-}$ and $\bar{H}_{w_+}$, for $w_- > w_+$, to be $(2^{-w_+} - 2^{-w_-})/3$-bounded.
\end{itemize}
Then the approximate action filtration $F^{\geq K} C_q$, defined according to \eqref{eq:q-action}, is compatible with the differential; it is exhaustive, and bounded below in each degree.
\end{lemma}

Here, exhaustive means any element of $C_q$ belongs to one of the subspaces of the filtration; and (awkwardly, because our filtrations are decreasing) bounded below means that in each degree, $F^{\geq K} C_q = 0$ for some $K$ (where the specific $K$ depends on the degree). From an algebraic point of view, having a bounded below filtration is slightly stronger than necessary (completeness would be sufficient in order standard filtration arguments), but it comes for free in our context.

\begin{proof}
First, we have to check that continuation map Hamiltonians with the required property exists, meaning that the bound from \eqref{eq:continuation-epsilon} is satisfied:
\begin{equation}
(2^{-w_+} + 2^{-w_-})/12 \leq 2^{-w_+}/6 \leq (2^{-w_+}-2^{-w_-})/3
\leq 1/3 \leq 1-(2^{w_+} + 2^{w_-})/6.
\end{equation}
We also have to ensure that the continuation map Hamiltonians can be picked compatibly with composition. This follows from the additivity of our chosen bound, $(2^{-w_+} - 2^{-w_0})/3 + (2^{-w_0} - 2^{w_-})/3 = (2^{w_+} - 2^{w_-})/3$.

Explicitly, an element of $F^{\geq K} C_q$ is of the form
\begin{equation} \label{eq:filtered-q-element}
\sum_{j=0}^\infty \Big(\sum_{\substack{\text{finitely}\\\text{many }w}} x_{w,j} + \eta x_{w,j}^\dag \Big) q^j \;\;\text{ where }  x_{w,j}, x_{w,j}^\dag
\in F^{\geq K-j} \mathit{CF}(w).
\end{equation}
Compatibility of the differential with the filtration follows from \eqref{eq:approximate-action-increase-2}. Now we'll use the additional assumptions from Lemma \ref{th:satisfy-epsilon-bound}. A priori, elements of $C_q$ are infinite sums in $q$; but because the degrees of one-periodic orbits are nonnegative, and $q$ has degree $2$, any such sum (describing an element in some given degree) is actually finite, which means that it can involve only finitely many one-periodic orbits. If one takes $K$ sufficiently negative, then all those orbits will satisfy the action bounds in \eqref{eq:filtered-q-element}, which shows that the filtration is exhaustive. On the other hand, the condition $y \cdot D \geq 0$ ensures that $F^{\geq 1}\mathit{CF}(w) = 0$, which means that the nonzero entries in an element of $F^{\geq 2j+1} C_q$ must involve powers $q^j$ or higher. By the same degree argument as before, it follows that in any given degree, the groups $F^{\geq K}C_q$ become zero for $K \gg 0$.
\end{proof}

\begin{remark}
It is instructive to look at the situation where $c_1(M) = m[D]$ for some $m \geq 2$ (we still take $[\omega_M] = [D]$, so that the discussion of action remains the same as before). To have a $\bZ$-grading, one needs to give $q$ degree $2m$. Suppose that we use $\epsilon$-bounded Hamiltonians produced by the Morsification process from the 
proof of Lemma \ref{th:satisfy-epsilon-bound}(ii). The orbits with winding number $-km$ around $D$ have action approximately $-km$ and degree in $2km(1-m) + [0,2n-1]$. Hence, the expression $q^{k(m-1)}x$ has approximate action $(-k)$ and degree in $[0,2n-1]$. The definition of $C_q$ allows infinite sums of such expressions with increasing $k$, which means that the action filtration is no longer exhaustive.
Indeed, as already pointed out in \cite{borman-sheridan-varolgunes21}, there can be no convergent spectral sequence \eqref{eq:bsv} in that context, as the example $(M,D) = (\bC P^n, \bC P^{n-1})$ shows. If on the other hand one assumes $[D] = m c_1(M)$ for some $m \geq 2$, exhaustivity still works (modulo introducing rational gradings); this case is considered in  \cite{borman-sheridan-varolgunes21} but we have not pursued it here, for lack of immediate applications.
\end{remark}

\begin{remark}
In Lemma \ref{th:filtered-q-telescope}, we have used bounds that decrease exponentially in $w$, in order to make the choices compatible with composing arbitrarily many continuation map equations. Readers who find that cumbersome can instead opt to use both sides $0 \leq \mathrm{deg}(x) \leq 2n-1$ of the degree bound from Lemma \ref{th:satisfy-epsilon-bound}(ii). This a priori implies that $d_m = 0$ for $m>n$, and $d_m^\dag = 0$ for $m \geq n$, which means we do not have to worry about filtration aspects for those operations. As a consequence, it is then sufficient to work with Hamiltonians that are $\epsilon$-bounded for a single sufficiently small $\epsilon$ (which depends on the dimension $n$, but is independent of $w$).
\end{remark}

%
%

\subsection{The associated graded space}
Throughout the following discussion, we assume that the slopes are chosen so that the following more precise version of \eqref{eq:one-integer} holds:
\begin{equation} \label{eq:in-between}
\sigma_w \in (w,w+1).
\end{equation}
Take $C_q$ with its filtration, as in Lemma \ref{th:filtered-q-telescope}. The associated graded spaces, for $K \in \bZ$, are
\begin{equation} \label{eq:gr-cq}
\mathit{Gr}^K C_q = \prod_{k=0}^{\infty} q^k \Big( \bigoplus_{w=0}^{\infty} \mathit{Gr}^{K-k}\mathit{CF}(w) \oplus \eta\mathit{Gr}^{K-k}\mathit{CF}(w) \Big).
\end{equation}
The differential on $\mathit{Gr}^KC_q$ retains only the lowest energy parts of $d_m$ and $d_m^\dag$: it consists of
\begin{equation} \label{eq:graded-differential}
\begin{aligned}
\mathit{Gr}^{-m} d_m:\; & q^k\mathit{Gr}^{K-k}\mathit{CF}(w) \longrightarrow q^{k+m}\mathit{Gr}^{K-k-m}\mathit{CF}(w+m), \\
-\mathit{Gr}^{-m} d_m:\; & \eta q^k\mathit{Gr}^{K-k}\mathit{CF}(w) \longrightarrow \eta q^{k+m}\mathit{Gr}^{K-k-m}\mathit{CF}(w+m), \\
-\mathit{id}:\; & \eta q^k\mathit{Gr}^{K-k}\mathit{CF}(w) \longrightarrow q^k\mathit{Gr}^{K-k}\mathit{CF}(w), \\
\mathit{Gr}^{-m} d_m^\dag: &\; \eta q^k\mathit{Gr}^{K-k} \mathit{CF}(w) \longrightarrow q^{k+m}\mathit{Gr}^{K-k-m}\mathit{CF}(w+m+1).
\end{aligned}
\end{equation}
As a direct consequence, 
\begin{equation} \label{eq:g-complex}
G^K = \!\! \prod_{k = \mathrm{max}(0,-K)}^{\infty} \!\!\!\! q^{k+K}\mathit{Gr}^{-k} \mathit{CF}(k)
\end{equation}
is a subcomplex of $\mathit{Gr}^K C_q$.

\begin{lemma} \label{th:g-inclusion}
The inclusion $G^K \hookrightarrow \mathit{Gr}^K C_q$ is a quasi-isomorphism.
\end{lemma}
\begin{figure}
\begin{centering}
\[
\xymatrix{ \mathit{Gr}^{K-k}\mathit{CF}(0) 
\ar@(dr,dl)[]^-{d_0}
&&
\eta \mathit{Gr}^{K-k} \mathit{CF}(0)
\ar[ll]_{-\mathit{id}} 
\ar@(dr,dl)[]^-{-d_0}
\ar[ddll]^-{d_0^\dag}
\\ \\
\mathit{Gr}^{K-k} \mathit{CF}(1)
\ar@(dr,dl)[]^-{d_0}
&&
\eta \mathit{Gr}^{K-k}\mathit{CF}(1) 
\ar[ll]_{-\mathit{id}} 
\ar[ddll]^-{d_0^\dag}
\ar@(dr,dl)[]^-{-d_0}
\\ \\
\mathit{Gr}^{K-k} \mathit{CF}(2) && \dots
}
\]
\caption{\label{fig:filtered-telescope}The associated graded of the $q$-filtration on $\mathit{Gr}^K C_q$, drawn in analogy with Figure \ref{fig:telescope}. Here, $K \in \bZ$ and $k \geq 0$.}
\end{centering}
\end{figure}%

\begin{proof}
Take the (decreasing, bounded above, complete, bounded above) filtration of $\mathit{Gr}^K C_q$ by powers of $q$. The $k$-th associated graded space is the $q^k$ factor in \eqref{eq:gr-cq}, with the differential formed by the $m = 0$ pieces in \eqref{eq:graded-differential} (including the identity map). That graded space is itself a telescope construction (see Figure \ref{fig:filtered-telescope}), with cohomology the limit under continuation maps,
\begin{equation}
 \underrightarrow{\lim}_w\, H(\mathit{Gr}^{K-k} \mathit{CF}(w)).
\end{equation}
For $k < K$, this cohomology is trivial, by Lemma \ref{th:parts-of-the-filtration}(i) and \eqref{eq:in-between}. Similarly, Lemma \ref{th:parts-of-the-filtration}(iii) shows that the continuation map induces a homotopy equivalence $\mathit{Gr}^{K-k} \mathit{CF}(w) \rightarrow \mathit{Gr}^{K-k} \mathit{CF}(w+1)$ for all $w \geq k-K$. Hence, if $k \geq K$, the inclusion of $\mathit{Gr}^{K-k}\mathit{CF}(k-K)$ into the $k$-th associated graded of the $q$-filtration is a quasi-isomorphism. Those are precisely the maps (on the graded of the $q$-filtration) induced by $G^K \hookrightarrow \mathit{Gr}^KC_q$. Therefore, that inclusion is a quasi-isomorphism.
\end{proof}

\begin{lemma} \label{th:nonnegative-k}
For $K \geq 0$, multiplication with $q$ gives a quasi-isomorphism (of degree $2$) $\mathit{Gr}^K C_q \rightarrow \mathit{Gr}^{K+1} C_q$.
\end{lemma}

\begin{proof} 
This is clearly what happens with $G^K$, so it follows from Lemma \ref{th:g-inclusion}.
\end{proof}

\subsection{Filtered maps from thimbles\label{subsection:filteredthimbles}}
The analogue of Definition \ref{th:epsilon-continuation} for the thimble is:

\begin{definition} \label{th:epsilon-thimble}
Suppose that $\bar{H}_-$ is $(\epsilon_-)$-bounded. A Hamiltonian $H$ on the thimble, which equals $\bar{H}_-$ for $s \ll 0$, is called $\epsilon$-bounded for 
\begin{equation} \label{eq:thimble-epsilon-bound}
\epsilon \in [\epsilon_-/2, 1-\epsilon_-]
\end{equation}
if it satisfies \eqref{eq:almost-monotone}.
\end{definition}

These exist for any choice of $\epsilon$ in \eqref{eq:thimble-epsilon-bound}, and the limits $x_-$ of any solution of the associated continuation map must satisfy
\begin{equation}
x_- \text{ has action approximately $k_-$, for some $k_- \geq -(u \cdot D)$.}
\end{equation}

\begin{lemma} \label{th:filtered-thimble}
Take the filtered version of $C_q$ defined in Lemma \ref{th:filtered-q-telescope}. 
Correspondingly, when defining the thimble maps from Section \ref{sec:thimbles}, one can make the following choices:
\begin{itemize}
\item For any equation on the thimble which has slope $\sigma_{w}$ at $-\infty$, use a $(1-2^{-w}/3)$-bounded Hamiltonian.
\end{itemize}
Then, the resulting maps satisfy
\begin{align}
\label{eq:filtered-s}
& s_{C_q}: \mathit{CM}(M) \longrightarrow F^{\geq 0} C_q, \\
\label{eq:filtered-t}
& t_{C_q,w}: \mathit{CM}(D) \longrightarrow F^{\geq -w} C_q.
\end{align}
\end{lemma}

\begin{proof}
Recall that the Hamiltonians for $\mathit{CF}(w)$ are $(2^{-w}/6)$-bounded. It is ok to choose $(1-2^{-w}/3)$-bounded Hamiltonians on the thimble, since that clearly lies in \eqref{eq:thimble-epsilon-bound}. Moreover, this choice is compatible with the bounds for the continuation map Hamiltonians, since $(1-2^{-w_+}/3) + (2^{-w_+} - 2^{w_-})/3 = (1-2^{-w_-}/3)$. With these choices
\begin{equation} \label{eq:filtered-s-t}
\begin{aligned}
& s_m: \mathit{CM}(M) \longrightarrow F^{\geq -m} \mathit{CF}(m), \\
& t_{w,m}: \mathit{CM}(D) \longrightarrow F^{\geq -m-w}\mathit{CF}(m+w),
\end{aligned} 
\end{equation}
which after inserting the necessary powers of $q$, implies the result as stated.
\end{proof}

Let's compose the maps \eqref{eq:filtered-s}, \eqref{eq:filtered-t} with projection to the associated graded space of the filtration. The outcome takes values in the subcomplexes \eqref{eq:g-complex}, more precisely
\begin{align}
\label{eq:graded-s}
&
\mathit{Gr}^0 s_{C_q}: \mathit{CM}(M) \longrightarrow G^0 \subset \mathit{Gr}^0 C_q, \\
\label{eq:graded-t}
&
\mathit{Gr}^{-w} t_{C_q,w}: \mathit{CM}(D) \longrightarrow G^{-w} \subset \mathit{Gr}^{-w} C_q.
\end{align}

\begin{theorem} \label{th:key}
The maps \eqref{eq:graded-s} and \eqref{eq:graded-t} are quasi-isomorphisms.
\end{theorem}

This is our core result, and its proof will extend over the next two sections (the argument for \eqref{eq:graded-t} is completed Section \ref{subsec:negative-action}, and that for \eqref{eq:graded-s} correspondingly in Section \ref{subsec:zero-action}). For now, we summarize how it leads to the theorem stated at the beginning of the paper.

\begin{proof}[Proof of Theorem \ref{th:main}]
Extend \eqref{eq:filtered-s} $q$-linearly, and combine it with \eqref{eq:filtered-t} to get a chain map
\begin{equation} \label{eq:big-map}
s \oplus \bigoplus_{w=1}^{\infty} t_w:
\mathit{CM}(M)[[q]] \oplus \bigoplus_{w=1}^\infty \mathit{CM}(D) z^w \longrightarrow C_q.
\end{equation}
Here, the symbol $z$ merely serves to label the summand on which we apply $t_w$. This is a filtered map, where the domain carries the filtration (exhaustive, and bounded below in each degree)
\begin{equation} \label{eq:filtered-total-thimble}
F^{\geq K}\Big( \mathit{CM}(M)[[q]] \oplus \bigoplus_{w=1}^\infty \mathit{CM}(D) z^w \Big)
= \begin{cases} q^K \mathit{CM}(M)[[q]] & K \geq 0, \\
\mathit{CM}(M)[[q]] \oplus \bigoplus_{w \leq -K} \mathit{CM}(D) z^w & K<0.
\end{cases}
\end{equation}
Theorem \ref{th:key} for \eqref{eq:graded-s} (combined with Lemma \ref{th:g-inclusion}) shows that the associated graded at $K = 0$  is a quasi-isomorphism, which via Lemma \ref{th:nonnegative-k} generalizes to $K \geq 0$; and the corresponding statement for \eqref{eq:graded-t} yields the same conclusion for $K<0$. By the standard filtration argument, it follows that \eqref{eq:big-map} itself is a quasi-isomorphism.
\end{proof}

\section{Morse theory\label{sec:morse}}
Let $N$ be the real-oriented blow up of $M$ along $D$. This is a compact manifold with boundary, equipped with a canonical map $\pi_N: N \rightarrow M$, which restricts to a diffeomorphism $N \setminus \partial N \cong M \setminus D$ on the interior, and to a circle bundle on the boundary, $\pi_{\partial N}: \partial N \rightarrow D$. 
There are natural operations
\begin{align} 
i_{\partial N}^*: \; & H^*(N) \longrightarrow H^*(\partial N) && \text{restriction,} \\
\pi_N^*: \; & H^*(M) \longrightarrow H^*(N) && \text{pullback,} \\
\pi_{\partial_N}^*:\;& H^*(D) \longrightarrow H^*(\partial N) && \text{pullback,} \\
\pi_{\partial N,*}:\;& H^*(\partial N) \longrightarrow H^{*-1}(D) && \text{integration along the fibres.}
\end{align} 
The goal of this section is to develop chain level models for these operations, using Morse theory; this replaces the use of de Rham theory in \cite[Section 7.3]{pomerleano-seidel23}, which is less technically convenient for interfacing with Floer theory.

\subsection{The circle bundle\label{section:Morsecircle}}
We begin by spelling out the geometric setup. Let $\nu D$ be the normal bundle of $D \subset M$. As a set, the real oriented blowup is
\begin{equation}
N = (M \setminus D) \cup \{\text{oriented real lines in $\nu D$}\},
\end{equation}
with the obvious map $\pi_N$. The local model is the map $[0,\infty) \times S^1 \times \bR^{2n-2} \rightarrow \bR^{2n}$, $(t,v,w) \mapsto (tv,w)$. Diffeomorphisms of $\bR^{2n}$ which preserve $\{0\} \times \bR^{2n-2}$ lift to diffeomorphisms of $[0,\infty) \times S^1 \times \bR^{2n-2}$ (see e.g.\ \cite[p.~824]{arone-kankaanrinta10} or \cite[Lemma 2.5.1]{kronheimer-mrowka}). Hence, lifting charts from $(M,D)$ leads to a canonical smooth structure on $N$. 
Choose the following data:
\begin{itemize} \itemsep.5em
\item on $D$, a Morse function $f_D$ and pseudo-gradient vector field $X_D$, which is Morse-Smale (see Section \ref{subsec:morse} for terminology and conventions);
\item a hermitian metric and connection on $\nu D$, where the connection is flat on a small ball $U_c \subset D$ around each critical point $c$;
\item for each critical point, an isomorphism $(\nu D)_c \iso \bC$, compatible with the hermitian metric and orientation;
\item a Morse function $g_{S^1}$ on the circle, with two critical points $q^{\mathrm{min}}$ and $q^{\mathrm{max}}$. Let $X_{S^1}$ be its gradient for the standard metric.
\end{itemize}
The unit circle bundle of $\nu D$ is canonically identified with $\partial N$. Using the flat connection, and the chosen isomorphisms $(\nu D)_c \iso \bC$, one obtains diffeomorphisms
\begin{equation} \label{eq:trivialization} 
\pi_{\partial N}^{-1}(U_c) \cong S^1 \times U_c. 
\end{equation} 
Choose a function $g_{\partial N}$ on $\partial N$ which, in each trivialization \eqref{eq:trivialization}, is (the pullback of) $g_{S^1}$. For a small constant $\delta>0$, set 
\begin{equation} \label{eq:fpartialN}
f_{\partial N} =  \delta \cdot g_{\partial N} + \pi_{\partial N}^* f_D.
\end{equation}
Take the horizontal lift $X_{\partial N}^h$ of $X_D$ to $\partial N$, determined by our connection. Choose a vector field $X_{\partial N}^v$ tangent to the circle fibres, and which in each local trivialization \eqref{eq:trivialization} is $(X_{S^1},0)$. Set
\begin{equation}
X_{\partial N} =  \delta \cdot X_{\partial N}^{v} + X_{\partial N}^{h}. 
\end{equation}

\begin{lemma}
(i) All critical points of $f_{\partial N}$ lie in the fibres over critical points $c$ of $f_D$. There are two in each fibre: $c^{\mathrm{min}} = (c,q^{\mathrm{min}})$ and $c^{\mathrm{max}} = (c,q^{\mathrm{max}})$, in the trivialization \eqref{eq:trivialization}.

(ii) For sufficiently small $\delta$, $X_{\partial N}$ is a pseudo-gradient vector field for $f_{\partial N}$.
\end{lemma}

\begin{proof}
We'll only prove (ii), since the nontrivial part of (i) (the absence of critical points in other circle fibres) follows from that argument. By construction, 
\begin{equation}
df_{\partial N}(X_{\partial N})= df_D(X_{D})+ O(\delta). 
\end{equation}
Outside the union of $\pi_{\partial N}^{-1}(U_c)$, the first term is positive and bounded away from zero; therefore, the entire expression will be positive if $\delta$ is sufficiently small. On $\pi^{-1}_{\partial N}(U_c)$, the function and vector field are split:
\begin{equation} \label{eq:split-vector-field}
\begin{aligned}
& f_{\partial N} = \delta \cdot g_{\partial S^1} + f_D, \\
& X_{\partial N} = (\delta \cdot X_{S^1}, X_D);
\end{aligned}
\end{equation}
hence $X_{\partial N}$ is the gradient vector field, in the product metric.
\end{proof}

\begin{lemma} \label{th:stable-unstable}
For each critical point $c$ in $D$, and its preimages $c^{\mathrm{min}}$, $c^{\mathrm{max}}$ in $\partial N$, we have:

(i) $\pi_{\partial N}$ induces diffeomorphisms
\begin{equation} \label{eq:section1}
\begin{aligned} 
& W^u(\partial N, c^{\mathrm{min}}) \cong W^{u}(f_D, c), \\ 
& W^s(\partial N, c^{\mathrm{max}}) \cong W^{s}(f_D, c). 
\end{aligned}
\end{equation}
In other words, $W^u(\partial N, c^{\mathrm{min}})$ and $W^s(\partial N, c^{\mathrm{max}})$ are sections of the circle bundle over $W^{u}(f_D,c)$ and $W^{s}(f_D,c)$, respectively. 

(ii) The other stable and unstable manifolds in $\partial N$ are
\begin{equation} \label{eq:minussection1} 
\begin{aligned} 
& W^u(\partial N, c^{\mathrm{max}}) = \pi_{\partial N}^{-1}(W^u(D, c)) 
\setminus W^u(\partial N, c^{\mathrm{min}}), \\ 
& W^s(\partial N, c^{\mathrm{min}}) = \pi_{\partial N}^{-1}(W^s(D, c))
\setminus W^s(\partial N, c^{\mathrm{max}}).
\end{aligned}
\end{equation}
\end{lemma} 

\begin{proof} 
(i) We prove $W^s(\partial N, c^{\mathrm{max}}) \cong W^s(D, c)$, the other case being similar. Choose a neighbourhood $\tilde{U}_c$ of $c$ with the following property: for every point of $W^s(c) \cap \tilde{U}_c$, the half-flow line from that point to $c$ lies inside $U_c$. In the local trivialization \eqref{eq:trivialization}, we have \eqref{eq:split-vector-field} and therefore
\begin{equation}
\label{eq:localsectionUc} 
W^s(\partial N, c^{\mathrm{max}}) \cap (S^1 \times \tilde{U}_c) = q^{\mathrm{max}} \times (W^s(D, c) \cap \tilde{U}_c).
\end{equation}
Let $(\phi_D^s)$ be the flow of $-X_D$, and similarly $(\phi_{\partial N}^s)$ for $-X_{\partial N}$. By construction, the flow on $\partial N$ covers that on $D$:
\begin{equation} \label{eq:flow-projection}
\pi_{\partial N} \circ \phi_{\partial N}^s = \phi_D^s \circ \pi_{\partial N}.
\end{equation}
Given $x \in W^s(D, c)$, there is some $s \gg 0$ such that $\phi^s(x) \in \tilde{U}_c$. By \eqref{eq:localsectionUc} the fibre of $\pi_{\partial N}$ over $\phi^s_D(x)$ contains a unique point $y \in W^s(\partial N, c^{\mathrm{max}})$. Therefore, the fibre over $x$ also contains a unique point, namely $(\phi_{\partial N}^s)^{-1}(y)$.

(ii) Again, we'll only do the second case. Because of \eqref{eq:flow-projection},
\begin{equation} \label{eq:stable-sub}
W^s(\partial N, c^{\mathrm{min}}) \subset \pi_{\partial N}^{-1}(W^s(D, c)) \setminus W^s(\partial N, c^{\mathrm{max}}). 
\end{equation}
Take $x \in W^s(D,c)$, and a point $y \in \pi_{\partial N}^{-1}(x)$ 
which does not lie on $W^s(\partial N, c^{\mathrm{max}})$. Taking $s$ as before, we find that 
\begin{equation}
\phi_{\partial N}^s(y) \in (S^1 \setminus q^{\mathrm{max}}) \times (W^s(D,c) \cap \tilde{U}_c).
\end{equation}
But then, continuing the flow will take the $D$ component asymptotically to $c$, and the $S^1$ component to $q^{\mathrm{min}}$, which means that $y \in W^s(\partial N, c^{\mathrm{min}})$; hence equality holds in \eqref{eq:stable-sub}.
\end{proof}

A transversality argument \cite[\S 6.1]{hutchings08} shows that one can achieve the Morse-Smale condition within the class of $(f_{\partial N}, X_{\partial N})$ we have constructed. The resulting Morse complex can be written as
\begin{equation} \label{eq:decomposition} 
\mathit{CM}^*(\partial N) = \mathit{CM}^*(D)^{\mathrm{min}} \oplus CM^{*-1}(D)^{\mathrm{max}}.
\end{equation}

\begin{lemma} \label{lem:matrixmorse}
With respect to \eqref{eq:decomposition}, the Morse differentials $d_{\partial N}$ and $d_D$ are related by
\begin{equation} \label{eq:d-dn}
d_{\partial N} = \begin{pmatrix}
d_D & \chi \\
0 & -d_D
\end{pmatrix}.
\end{equation}
\end{lemma}

Here, $\chi$ is a chain map $CM^*(D) \rightarrow CM^{*+2}(D)$ (on cohomology, it describes the cup product with the Chern class of $\partial N \rightarrow D$); and the switch to $-d_D$ is a standard Koszul sign (which, in our context, comes from writing \eqref{eq:trivialization} with the $S^1$ factor first).

\begin{proof} 
Take critical points $c_{\pm}$ on $D$ with $\mathrm{deg}(c_-) = \mathrm{deg}(c_+) + 1$.
Then $c_-^{\mathrm{min}}$ and $c_+^{\mathrm{max}}$ have the same Morse index, and by the Morse-Smale condition we must therefore have
\begin{equation} \label{eq:emptystableunstable}
W^u(\partial N, c_-^{\mathrm{min}}) \cap W^s(\partial N, c_+^{\mathrm{max}}) = \emptyset.
\end{equation}
Take a point $x \in W^u(D, c_-) \cap W^s(D, c_+)$. By Lemma \ref{th:stable-unstable}(i) there is a unique preimage $y \in W^u(\partial N, c_-^{\mathrm{min}})$. By \eqref{eq:emptystableunstable} $y$ does not lie in $W^s(\partial N, c_+^{\mathrm{max}})$, which by Lemma \ref{th:stable-unstable}(ii) implies that $y \in W^s(\partial N, c_+^{\mathrm{min}})$. A similar argument produces a unique preimage of $x$ in $W^u(\partial N, c_-^{\mathrm{max}}) \cap W^s(\partial N, c_+^{\mathrm{max}})$. To summarize, we have shown that $\pi_{\partial N}$ induces bijections
\begin{equation}
\left.
\begin{aligned}
& W^u(\partial N, c_-^{\mathrm{min}}) \cap W^s(\partial N, c_+^{\mathrm{min}}) \\
& W^u(\partial N, c_-^{\mathrm{max}}) \cap W^s(\partial N, c_+^{\mathrm{max}}) 
\end{aligned} \right\}
\iso W^u(D, c_-) \cap W^s(D, c_+).
\end{equation}
When combined with sign considerations (which we omit here), this explains the two occurrences of $d_D$ in \eqref{eq:d-dn}.

Finally, consider two points with $\mathrm{deg}(c_-) = \mathrm{deg}(c_+)$. If $c_- \neq c_+$, the Morse-Smale assumption implies that $W^u(D, c_-) \cap W^s(D, c_+) = \emptyset$. Hence, even though $\mathrm{deg}(c_-^{\mathrm{max}}) = \mathrm{deg}(c_+^{\mathrm{min}}) + 1$, there are no flow lines in $\partial N$ connecting those critical points. In the remaining case $c_- = c_+$, there are two flow lines in the fibre of $\partial N$ over that point, but those cancel by the usual Morse theory for $S^1$. This explains the $0$ entry in \eqref{eq:d-dn}. 
\end{proof}

We record the following immediate consequence of Lemma \ref{lem:matrixmorse}:

\begin{cor} \label{cor:pidN}
(i) The inclusion $\mathit{CM}(D) = \mathit{CM}(D)^{\mathrm{min}} \hookrightarrow \mathit{CM}(\partial N)$ is a chain map. We denote it by $\pi_{\partial N}^*$.

(ii) The projection $\mathit{CM}(\partial N) \twoheadrightarrow \mathit{CM}(D)^{\mathrm{max}} = \mathit{CM}(D)$ is a chain map of degree $-1$. We denote it by $\pi_{\partial N,*}$.
\end{cor}

Clearly $\pi_{\partial N,*} \pi_{\partial N}^* = 0$. Composition in the other order yields an endomorphism of degree $-1$,
\begin{align} 
\pi_{\partial N}^* \pi_{\partial N,*}: \mathit{CM}(\partial N) \longrightarrow \mathit{CM}(\partial N),
\end{align}
whose square is zero. One can use it to introduce a $q$-deformed version of the Morse complex (with $q$ of degree $2$ as usual),
\begin{equation} \label{eq:q-deform-dn}
\begin{aligned}
& \mathit{CM}(\partial N)_q = \mathit{CM}(\partial N)[q], \\
& d_{\partial N,q} = d_{\partial N} +q\,\pi_{\partial N}^* \pi_{\partial N,*}.
\end{aligned}
\end{equation}

\begin{lem} \label{qdeformedisD} 
$\mathit{CM}(\partial N)_q$ is quasi-isomorphic to $\mathit{CM}(D)$, by the map
\begin{equation} \label{eq:qdefo}
\mathit{CM}(D) \xrightarrow{\pi_{\partial N}^*} \mathit{CM}(\partial N)q^0 \hookrightarrow \mathit{CM}(\partial N)_q.
\end{equation}
\end{lem}

\begin{proof} 
The analogue of \eqref{eq:decomposition} is
\begin{equation}
\begin{aligned}
& \mathit{CM}^*(\partial N)_q = \mathit{CM}^*(D)^{\mathrm{min}}[q] \oplus
\mathit{CM}^{*-1}(D)^{\mathrm{max}}[q], \\
& d_{\partial N,q} = \begin{pmatrix} d_D & \chi + q\,\mathit{id} \\ 0 & -d_D \end{pmatrix},
\end{aligned}
\end{equation}
which is the mapping cone of the degree $2$ map $\chi + q\,\mathit{id}: \mathit{CM}(D)[q] \rightarrow \mathit{CM}(D)[q]$. That map is injective,
so the projection 
\begin{equation}
\mathit{CM}(\partial N)_q \longrightarrow \mathit{CM}(D)^{\mathrm{min}}[q]/\mathit{im}(\chi + q \,\mathit{id})
\end{equation}
is a quasi-isomorphism. Composing \eqref{eq:qdefo} with that projection yields an isomorphism.
\end{proof}

\subsection{The closed manifold}
Fix a tubular neighbourhood of $D$,
\begin{equation} \label{eq:tubular}
\nu D \supset \{\|\xi\| \leq \epsilon\} \longrightarrow M.
\end{equation}
On $M$, choose a Morse function $f_M$ and pseudo-gradient $X_M$, whose restrictions to \eqref{eq:tubular} are
\begin{equation} \label{eq:fmxm}
f_M = \half \|\xi\|^2 + \pi_{\nu D}^*f_D, \quad X_M = \xi\partial_\xi + X_{\nu D}^h;
\end{equation}
Here, $\pi_{\nu D}: \nu D \rightarrow D$ is the projection; $X_{\nu D}^h$ is the horizontal lift of $X_M$ for our connection; and $\xi \partial_\xi$ is the infinitesimal radial expansion vector field on the fibres. In particular, the critical points of $f_M$ lying in $D$ are precisely those of $f_D$. Near such a critical point, one can use the flatness of the connection to get local charts $\{z \in \bC\;:\;|z| \leq \epsilon\} \times U_c \hookrightarrow M$, in which
\begin{equation} \label{eq:near-critical-points}
f_M = \half |z|^2 + f_D, \quad X_M = (z\partial_z, X_D).
\end{equation}

\begin{lemma} \label{th:tangent}
Let $c$ be a critical point of $f_D$. Then $W^u(D,c) = W^u(M,c)$; and $W^s(M,c)$ intersects $D$ transversally, with the intersection being $W^s(D,c)$.
\end{lemma}

\begin{proof}
From the split form of \eqref{eq:near-critical-points} one sees that if $b: (-\infty,0] \rightarrow M$ is a flow line with asymptotics $c$, then necessarily
$b(s) \in D$ for $s \ll 0$. Since $X_M$ is tangent to $D$, the same must hold for all $s$, which proves the desired statement about $W^u(D,c)$. The same local analysis shows that near $c$, the manifold $W^s(M,c)$ intersects $D$ transversally, and the flow then allows one to carry over that insight to the entire manifold. The statement about the intersection again just follows from the fact that $X_M$ is tangent to $D$.
\end{proof}

As one consequence, the Morse-Smale condition can be achieved within this class of $(f_M,X_M)$. Write the resulting Morse cochain space as
\begin{equation} \label{eq:m-d-splitting}
\mathit{CM}^*(M) = \mathit{CM}^*(D) \oplus \mathit{CM}^*(M \setminus D),
\end{equation}
where the second summand contains the critical points lying outside $D$. Lemma \ref{th:tangent} implies:

\begin{corollary} \label{th:restrict-to-d}
The projection $\mathit{CM}(M) \rightarrow \mathit{CM}(D)$ is a chain map.
\end{corollary}

On cohomology, this map gives the restriction $H^*(M) \rightarrow H^*(D)$; the kernel $\mathit{CM}(M \setminus D)$ correspondingly computes $H^*(M,D) = H^*_c(M \setminus D)$. 

\subsection{The blowup\label{subsection:Morseblowup}}
The tubular neighbourhood \eqref{eq:tubular} canonically determines a collar neighbourhood on the blowup, 
\begin{equation} \label{eq:collar}
[0,\epsilon] \times \partial N \hookrightarrow N,
\end{equation}
where the coordinate $\rho \in [0,\epsilon]$ corresponds to the previous $\|\xi\|$. Take $g_{\partial N}$ and a small $\delta>0$, as in \eqref{eq:fpartialN}. Moreover, take a cutoff function $\kappa: [0,\epsilon] \rightarrow \bR$ with $\kappa(\rho) = 1$ for $\rho$ close to $0$, and $\kappa(\rho) = 0$ near $\rho = \epsilon$. Define a function $f_N$ on $N$ by
\begin{equation} \label{eq:fNdefinition}
f_N = 
\begin{cases}
\half\rho^2 + \delta\cdot \kappa(\rho) g_{\partial N} + \pi_{\partial N}^* f_D & \text{on \eqref{eq:collar}}, \\
f_M & \text{outside that.}
\end{cases}
\end{equation}
Similarly, consider the vector field 
\begin{equation} \label{eq:xn}
X_N = \begin{cases}
 \rho\partial_\rho + \delta \cdot \kappa(\rho)X_{\partial N}^v + X_{\partial N}^h
& \text{on \eqref{eq:collar}}, \\
X_M & \text{outside that.}
\end{cases}
\end{equation}
Iit follows from the definition that $(f_N,X_N)$ restricts to $(f_{\partial N},X_{\partial N})$ on the boundary. In particular, the flow of $X_N$ is defined for all times; and the critical points of $f_N$ lying on the boundary are exactly those of $f_{\partial N}$. If $c$ is a critical point of $f_D$, then near $\pi_{\partial N}^{-1}(c) \iso S^1$ we have a local chart $[0,\epsilon] \times S^1 \times U_c \hookrightarrow N$, in which
\begin{equation}
f_N = \half\rho^2 + \delta\cdot \kappa(\rho)g_{S^1} + f_D, \;\;
X_N = (\rho\partial_\rho, X_{S_1}, X_D).
\end{equation}
Critical points lying in the interior of $N$ must be disjoint from the collar, hence $X_N = X_M$ near them. This shows that $X_N$ is a pseudo-gradient for $f_N$. The analogue of Lemma \ref{th:tangent}, with essentially the same proof, is:

\begin{lemma} \label{th:n0}
Let $c$ be a critical point of $f_N$ lying in $\partial N$. Then $W^u(N,c) = W^u(\partial N,c)$. Moreover, $W^s(N,c)$ intersects $\partial N$ transversally, and the 
intersection is $W^s(\partial N,c)$.
\end{lemma}

It again follows that there are no issues with achieving the Morse-Smale condition in this framework (by which we include perturbing $X_M$, within the same class of vector fields).

\begin{lemma} \label{th:n1}
Take a critical point $c$ of $f_D$, and the corresponding critical points $c^{\mathrm{min}}$, $c^{\mathrm{max}}$ of $f_{\partial N}$, hence of $f_N$. The diffeomorphism $\pi_N: (N \setminus \partial N) \rightarrow (M \setminus D)$ induces an identification
\begin{equation} \label{eq:MNequal} 
\big( W^s(N, c^{\mathrm{min}}) \cup 
W^s(N, c^{\mathrm{max}}) \big)
\cap (N \setminus \partial N) =
W^s(M, c) \cap (M \setminus D).
\end{equation}
\end{lemma}

\begin{proof}
On the collar, $X_N$ maps to $(\rho\partial_\rho, X_D)$ under projection to $[0,\epsilon] \times D$. Therefore,
\begin{equation} \label{eq:ws0}
\big( W^s(N, c^{\mathrm{min}}) \cup 
W^s(N, c^{\mathrm{max}}) \big) \cap ([0,\epsilon] \times \partial N)
= [0,\epsilon] \times \pi_{\partial N}^{-1}(W^s(D,c)).
\end{equation}
From the definition of $X_M$ in \eqref{eq:fmxm} it follows that in the tubular neighbourhood \eqref{eq:tubular}, 
\begin{equation} \label{eq:ws1}
W^s(M,c) \cap \{\|\xi\| \leq \epsilon\} = \pi_{\nu D}^{-1}(W^s(D,c)) \cap \{\|\xi\| \leq \epsilon\}.
\end{equation}
By comparing the right hand sides of \eqref{eq:ws0} and \eqref{eq:ws1}, one sees that the desired equality \eqref{eq:MNequal} holds in $(0,\epsilon] \times \partial N$, respectively its image in $M$. But outside those neighbourhoods, $X_N = X_M$ by definition, so the equality carries over by by applying the flow.
\end{proof}

We summarize part of our discussion as follows:

\begin{corollary} \label{th:project-stable}
Let $c$ be a critical point of $f_M$. 

(i) If $c$ lies in $M \setminus D$, then $W^s(M,c) = W^s(N,c)$. More precisely, these submanifolds are contained in $N \setminus \partial N$ respectively $M \setminus D$, and $\pi_N$ induces an isomorphism between them.

(ii) If $c$ lies in $D$, then $\pi_N^{-1}(W^s(M,c)) = W^s(N,c^{\mathrm{min}}) \cup W^s(N,c^{\mathrm{max}})$.
\end{corollary}

\begin{proof}
(i) follows from the fact that $W^s(M,c)$ is disjoint from the tubular neighbourhood \eqref{eq:tubular}, and $W^s(N,c)$ correspondingly from the collar \eqref{eq:collar}.

(ii) is obtained by combining: Lemma \ref{th:n0}; the corresponding relation between stable manifolds in $\partial N$ and $D$, which follows from \eqref{eq:flow-projection} (or if you like, is part of Lemma \ref{th:stable-unstable}); and Lemma \ref{th:n1}.
\end{proof}

Inside the cochain space $\mathit{CM}(N)$ defined by $(f_N,X_N)$, consider the subspace $\mathit{CM}(N \setminus \partial N) \subset \mathit{CM}(N)$ generated by critical points lying in the interior (as a reminder, all such points automatically lie outside the collar; and for $M$, one similarly has that critical points of $f_M$ not lying in $D$ must lie outside the tubular neighbourhood). One can identify
\begin{equation}
\mathit{CM}(N \setminus \partial N) \iso \mathit{CM}(M \setminus D),
\end{equation}
and therefore write
\begin{align}
\label{eq:n-1}
& \mathit{CM}^*(N) = \mathit{CM}^*(N \setminus \partial N) \oplus \mathit{CM}^*(\partial N);\;\;\text{ or } \\
\label{eq:n-2}
& \mathit{CM}^*(N) = \mathit{CM}^*(M \setminus D) \oplus \mathit{CM}^*(D)^{\mathrm{min}} \oplus \mathit{CM}^{*-1}(D)^{\mathrm{max}}.
\end{align}
So far, this concerned just the spaces of cochains; now we'll proceed to the differential.

\begin{lemma} \label{th:n2}
(i) In \eqref{eq:n-1}, $\mathit{CM}^*(N \setminus \partial N)$ is a subcomplex; and the induced differential on the quotient agrees with that previously defined on $C^*(\partial N)$. We denote the projection by 
\begin{equation} \label{eq:ipartialN}
i^*_{\partial N}: \mathit{CM}^*(N) \longrightarrow \mathit{CM}^*(\partial N). 
\end{equation}

(ii) In \eqref{eq:n-2}, $\mathit{CM}^*(M \setminus D) \oplus \mathit{CM}^*(D)^{\mathrm{min}}$ is a subcomplex; and the differential on that agrees with that on $\mathit{CM}^*(M)$, written as in \eqref{eq:m-d-splitting}. We denote the inclusion by
\begin{equation} \label{eq:mapMN} 
\pi_N^*: \mathit{CM}^*(M) \longrightarrow \mathit{CM}^*(N).
\end{equation}
\end{lemma}

\begin{proof}
(i) is a straightforward consequence of Lemma \ref{th:n0}.

(ii) We already know that $\mathit{CM}(M \setminus D)$ is a subcomplex. The component of the differential sending $\mathit{CM}(D)^{\mathrm{min}}$ to $\mathit{CM}(D)^{\mathrm{max}}$ can be computed entirely inside $\partial N$, and is therefore zero by Lemma \ref{lem:matrixmorse}. This shows that $\mathit{CM}(M \setminus D) \oplus \mathit{CM}(D)^{\mathrm{min}}$ is indeed a subcomplex.

Consider a critical point $c_+$ on $D$, and its preimages $c_+^{\mathrm{min}}, c_+^{\mathrm{max}}$. Take a critical point $c_-$ in $M \setminus D = N \setminus \partial N$, with $\mathrm{deg}(c_-) = \mathrm{deg}(c_+) + 1$. Since $c_-$ and $c_+^{\mathrm{max}}$ have the same index, the Morse-Smale condition says that
\begin{equation}
W^u(N, c_-) \cap W^s(N, c_+^{\mathrm{max}}) = \emptyset.
\end{equation}
Therefore, by Lemma \ref{th:n1} and the fact that $W^u(N, c_-) \subset N \setminus \partial N$, we have
\begin{equation}
W^u(N, c_-) \cap W^s(N, c_+^{\mathrm{min}}) = W^u(N,c_-) \cap W^s(M,c).
\end{equation}
Together with the necessary sign considerations (which we omit), this shows that the $(c_-,c_+^{\mathrm{min}})$-coefficient of $d_{\partial N}$ agrees with the $(c_-,c_+)$-coefficient of $d_M$.

By definition \eqref{eq:xn}, $X_N$ points in positive $\rho$-direction along the boundary $\{\rho = \epsilon\}$ of our collar; hence, if $c$ is a critical point in $N \setminus \partial N$, then $W^s(N,c)$ is disjoint from the collar, and therefore agrees with $W^s(M,c)$. In particular, if $c_{\pm}$ are two such points, the $(c_-,c_+)$-coefficient of $d_N$ is equal to its counterpart in $d_M$. 

Finally, for two critical points $c_{\pm}$ on $D$, the $(c_-^{\mathrm{min}},c_+^{\mathrm{min}})$-component of $d_N$ is computed inside $\partial N$. By Lemmas \ref{lem:matrixmorse} and Corollary \ref{th:restrict-to-d}, this is the same as the $(c_-,c_+)$-component of $d_M$. Together, the previous three computations demonstrate the claim concerning the differential.
\end{proof}

We also want to introduce a complex involving both $N$ and its boundary, which is a more complicated version of \eqref{eq:q-deform-dn}:
\begin{equation} \label{eq:n-q}
\begin{aligned} 
& \mathit{CM}(N)_q = \mathit{CM}(N) \oplus q \mathit{CM}(\partial N)[q], \\
& d_{N,q}(c) = 
\begin{cases}
d_Nc + q(\pi_{\partial N}^* \pi_{\partial N,*} i_{\partial N}^* c) & c \in \mathit{CM}(N), \\
d_{\partial N,q}c  =d_{\partial N}c + q\pi_{\partial N}^* \pi_{\partial N,*} & c \in q\mathit{CM}(\partial N)[q].
\end{cases}
\end{aligned}
\end{equation}
In this context, the analogue of Lemma \ref{qdeformedisD}, with a similar proof, is:

\begin{lem} \label{lem:NqdeformedisM}
The map 
\begin{equation} \label{eq:qdef1} 
\mathit{CM}(M) \xrightarrow{\pi_{N}^*} \mathit{CM}(N)q^0 \hookrightarrow \mathit{CM}(N)_q 
\end{equation} 
is a quasi-isomorphism. 
\end{lem}

\section{Filtered quasi-isomorphism\label{sec:proof}}
In a Morse-Bott picture, the symplectic cohomology of a smooth divisor complement \cite{diogo-lisi19} is constructed from chains on $M \setminus D$ and on the normal circle bundle to $D$ (which is $\partial N$ in our notation). In contrast, our description of the deformed symplectic cohomology only involves the cohomology of $M$ and $D$. In this section, we explain how to reconcile those two pictures, using thimbles with constraints in normal direction to $D$ (this is partly motivated by the toy model from \cite[Section 7.3]{pomerleano-seidel23}). The construction comes in two versions, corresponding to the moduli spaces from Sections \ref{section:thimbleswithtangency} and \ref{subsec:classical-thimble}; those two combine to form the proof of Theorem \ref{th:key}.

\subsection{Background}
We begin by reviewing the local picture (since there are several versions on the literature, under different assumptions on the almost complex structures).
Let $J$ be an almost complex structure on $\bC \times \bR^{2n}$ with the following properties: 
\begin{equation} \label{eq:general-almost-complex}
\parbox{36em}{
$J$ preserves $\bR^{2n} = \{0\} \times \bR^{2n}$; and along that submanifold, it splits as a product of the standard complex structure $i$ in normal direction and some almost complex structure $J_{\bR^{2n}}$ in tangent direction.}
\end{equation}
Take a $J$-holomorphic map $v = (v_{\bC},v_{\bR^{2n}}): \bC \rightarrow \bC \times \bR^{2n}$. Here, the domain carries its standard complex structure and coordinate $z = s+it$. We assume throughout the following discussion that at $z = 0$, the map intersects $\{0\} \times \bR^{2n}$ with multiplicity $\geq w$. Consider the derivative
\begin{equation} \label{eq:jetdef} 
\mathit{ev}^w_0(v) \stackrel{\mathrm{def}}{=} \frac{\partial^w v_{\bC}}{\partial s^{w}} (0) \in \bC.
\end{equation} 
Suppose we have a diffeomorphism $\phi: \bC \times \bR^{2n} \rightarrow \bC \times \bR^{2n}$ which preserves $\{0\} \times \bR^{2n}$; and such that the derivative $D\phi$ at any point of that submanifold is compatible with the splitting $\bC \times \bR^{2n}$, and is complex-linear on the first factor. Hence, $\phi$ takes $J$ to another almost complex structure in class \eqref{eq:general-almost-complex}. Then \cite[Corollary 6.3]{cieliebak-mohnke07}
\begin{equation}
\label{eq:normalvector}
\mathit{ev}^w_0(\phi \circ v) = D\phi(v(0)) \mathit{ev}^w_0(v). 
\end{equation}
The part of $v$ in normal direction to $\{0\} \times \bR^{2n}$ satisfies a differential equation
\begin{equation} \label{eq:a-equation}
(\partial_s + i\partial_t + A_{s,t})v_{\bC} = 0,
\end{equation}
where $A_{s,t} \in \mathit{Hom}_{\bR}(\bC,\bC)$ is smooth \cite[Proof of Theorem 2.88]{wendl}. By writing down the Taylor expansion \cite[Lemma 2.82]{wendl}, and comparing that with the definition \eqref{eq:jetdef}, one sees that
\begin{equation} \label{eq:wendl}
v_{\bC}(z) = (\mathit{ev}_0^w(v)/w!) z^w + O(|z|^{w+1}).
\end{equation}
In particular:
\begin{align}
\label{eq:nonvanishing}
&
\parbox{36em}{
$\mathit{ev}^w_0(v) \neq 0$ iff the local intersection multiplicity is exactly $w$. 
}
\\
& \label{eq:continuous-extension}
\parbox{36em}{$v_{\bC}(z)/z^w$ extends continuously to $z = 0$, and the value at that point is $(\mathit{ev}_0^w(v)/w!)$.}
\end{align}

The theory can be simplified by restricting the class of almost complex structures under considerations. Given an almost complex structure $J_{\bR^{2n}}$ on $\bR^{2n}$ and a one-form $\alpha \in \Omega^1(\bR^{2n})$, one can define an almost complex structure $J_{\alpha}$ on $\bC \times \bR^{2n}$ by \cite[Lemma 2.2]{zinger-basic}
\begin{equation} \label{eq:j-alpha}
J^{\alpha} = \begin{pmatrix} i & x(\alpha + i\alpha \circ J_{\bR^{2n}}) \\ 0 & J_{\bR^{2n}} \end{pmatrix},
\end{equation}
where $x$ is the $\bC$-coordinate. 

\begin{remark} \label{th:compatible}
If $J_{\bR^{2n}}$ is compatible with some $\omega_{\bR^{2n}}$, and $d\alpha$ is of type $(1,1)$ with respect to $J_{\bR^{2n}}$, then $J^\alpha$ is compatible with respect to a similarly constructed symplectic form, $\omega^{\alpha} = \omega_{\bC} + \omega_{\bR^{2n}} + d(\half |x|^2 \alpha)$. 
\end{remark}

Going through the argument from \cite[Theorem 2.88]{wendl} shows that \eqref{eq:a-equation} now has a complex-linear order zero term:
\begin{equation}
\text{for $J = J^\alpha$,}\;\;
A_{s,t} = \alpha(\partial_t u_{\bR^{2n}}) + i\alpha(J_{\bR^{2n}}\partial_t u_{\bR^{2n}})
\in \bC.
\end{equation}
As a consequence, one gets the following sharpening of \eqref{eq:continuous-extension} (by Taylor expansion, or alternatively by arguing as in \cite[Remark 2.80]{wendl}):
\begin{equation} \label{eq:smooth-wendl}
\parbox{36em}{for $J = J^\alpha$, the continuous extension of the map $v_{\bC}(z)/z^w$ to $z = 0$ is smooth.}
\end{equation}

\begin{remark}
There is also a range of intermediate possibilities: one can fix $d \in \{1,2,\dots,\infty\}$ and allow those $J$ which agree with some $J^\alpha$ to $d$-th order along $\{0\} \times \bR^{2n}$. The effect is that $A_{s,t}$ is complex-linear to $(d-1)$-st order around $z = 0$; and via Taylor expansion, that the extension of $v_{\bC}(z)/z^w$ is a $C^d$-function. For simplicity, we will not consider those intermediate options further, even though they could be practical: for instance, taking $d = \infty$ yields the same outcome as in \eqref{eq:smooth-wendl}, without constraining the almost complex structure away from $\{0\} \times \bR^{2n}$.
\end{remark}

Our next result is a technical one, used later to describe the local structure for certain points in a moduli space. We will first give the version for almost complex structures in \eqref{eq:general-almost-complex}, and then describe how it is affected by the more specific choice \eqref{eq:j-alpha}.

\begin{lemma} \label{th:winding-number-argument}
Take a family of almost complex structures $J_{\zeta,r}$ as in \eqref{eq:general-almost-complex}, smoothly depending on parameters $\zeta \in \bC$, $r \geq 0$. Let $v_{\zeta,r} = (v_{\zeta,r}^{\bC},v_{\zeta,r}^{\bR^{2n}}): \bC \rightarrow \bC \times \bR^{2n}$ be $J_{\zeta,r}$-holomorphic maps, again smooth in $(\zeta,r)$, such that 
\begin{equation}
\text{at $z = 0$, $v_{\zeta,r}$ intersects $\{0\} \times \bR^{2n}$ with multiplicity }
\begin{cases}
w & r>0, \\
w+1 & r = 0.
\end{cases}
\end{equation}
Fix a small $\epsilon>0$. Suppose that there are $r>0$ arbitrarily close to $0$, such that the solutions of
\begin{equation} \label{eq:zeta-zeta}
v_{\zeta,r}(\zeta) \in \{0\} \times \bR^{2n},\quad 0 < |\zeta| < \epsilon,
\end{equation}
are regular. For a small $r$ with that property, there are finitely many $\zeta$ satisfying \eqref{eq:zeta-zeta}; and algebraically (counting with signs) their number is $+1$.
\end{lemma}

\begin{proof}
Let's start by looking at the situation for $r = 0$. From \eqref{eq:wendl} (or rather a parametrized version of it) we see that $\zeta \mapsto v_{\zeta,0}^{\bC}(\zeta)/\zeta^{w+1}$ extends continuously over $\zeta = 0$, and is nonzero there. This implies $\zeta = 0$ is an isolated solution of $v_{\zeta,0}^{\bC}(\zeta) = 0$; and that it has multiplicity $(w+1)$, in the sense that for sufficiently small $\epsilon>0$, the loop
\begin{equation} \label{eq:winding-number-w-plus-one}
\theta \longmapsto v_{\epsilon e^{i\theta},0}^{\bC}(\epsilon e^{i\theta}) \in \bC^*
\end{equation}
has winding number $(w+1)$ around the origin. Now let's look at some small $r>0$. One can carry out the analogue of the argument above: the solution $\zeta = 0$ of $v_{\zeta,r}^{\bC}(\zeta) = 0$ is isolated; and it has multiplicity $w$, in the sense that for any sufficiently small $\delta$, the loop
\begin{equation}
\theta \longmapsto v_{\delta e^{i\theta},r}^{\bC}(\delta e^{i\theta}) \in \bC^*
\end{equation}
has winding number $w$. Note that here, $\delta$ depends on $r$. We want to take it to be less than the previous $\epsilon$, which is unproblematic. By continuity starting with \eqref{eq:winding-number-w-plus-one}, the loop
\begin{equation}
\theta \longmapsto v_{\epsilon e^{i\theta},r}^{\bC}(\epsilon e^{i\theta}) \in \bC^*
\end{equation}
still has winding number $(w+1)$. If $r$ is such that the nonzero solutions of $v_{\zeta,r}^{\bC}(\zeta) = 0$ are regular, then it follows from a comparison of winding numbers that the signed number of such solutions in the region $\delta < |\zeta| < \epsilon$ equals $+1$. By construction, there are no solutions with $0 < |\zeta| \leq \delta$, so the same signed count applies to solutions with $0 < |\zeta| < \epsilon$.
\end{proof}

\begin{lemma} \label{th:winding-number-argument-2}
In the situation of Lemma \ref{th:winding-number-argument}, suppose that the almost complex structures are of type \eqref{eq:j-alpha}. Then, for each sufficiently small $r$, there is a unique solution of \eqref{eq:zeta-zeta}. Moreover, those solutions, together with $(\zeta,r) = (0,0)$, form a family smoothly depending on $r$.
\end{lemma}

\begin{proof}
Unlike the previous proof, this one relies on the inverse function theorem rather than a topological (winding number) argument. Let's again start with $r = 0$. By a parametrized version of \eqref{eq:smooth-wendl}, the extension of $\zeta \mapsto v_{\zeta,0}^{\bC}(\zeta)/\zeta^{w+1}$ to $\zeta = 0$ is smooth, and its value at $\zeta = 0$ is nonzero. Hence, the map $\zeta \mapsto v_{\zeta,0}^{\bC}(\zeta)/\zeta^w$ has a regular zero at $\zeta = 0$. It follows that $\zeta \mapsto v_{\zeta,r}^{\bC}(\zeta)/\zeta^w$, for any small $r>0$, also has a unique zero close to $\zeta = 0$; these zeros are again regular and depend smoothly on $r$. By \eqref{eq:continuous-extension} the extension of $v_{\zeta,r}^\bC(\zeta)/\zeta^w $ at $\zeta = 0$ is nonzero for all $r>0$, so the solutions we've found must have nonzero $\zeta$.
\end{proof}

\subsection{Negative approximate action\label{subsec:negative-action}}
With this in place, we turn to our application. Throughout the section, we assume that the $q$-deformed telescope construction $C_q$ has been defined using slopes as in \eqref{eq:in-between}, and that the Hamiltonians involved satisfy the bounds from Lemma \ref{th:filtered-q-telescope}; similarly, Cauchy-Riemann equations on the thimble should be chosen as in Lemma \ref{th:filtered-thimble}. 

Fix $w>0$, $m \geq 0$. We use almost complex structures and inhomogeneous terms as in Section \ref{section:thimbleswithtangency}, with an extra consistency condition:
\begin{enumerate}[label=(AT\arabic*)] \itemsep.5em
\parindent0em \parskip.5em
\item \label{item:at-forget}
{\em (Marked point at $+\infty$)}
Take the embedding $\frakT_{w+1,m-1} \rightarrow \frakT_{w,m}$ given by adding a marked point at $+\infty$ to the divisor. (If one identifies $T \iso \bC$ so that $z= +\infty \in T$ corresponds to $\zeta = 0 \in \bC$, one can think of elements of the symmetric products as monic polynomials $p$; then, the embedding just multiplies a polynomial $p(\zeta)$ by $\zeta$.) The data defining the Cauchy-Riemann equations should be chosen compatibly with this embedding.
\end{enumerate}
For our thimbles, we will only allow the situation where the limiting one-periodic orbit $x_-$ has approximate action $-(w+m)$, which by \eqref{eq:thimble-energy} corresponds to $u$ having low energy. As a consequence, we do not have to worry about bubbling of holomorphic spheres. This limited setup is exactly what enters into the definition of the graded piece \eqref{eq:graded-t} of $t_{C_q,w}$. With that in mind, the Floer-theoretic part of our construction consists of pairs $(\Sigma,u)$ as in Section \ref{section:thimbleswithtangency}, which in particular means that the map $u$ satisfies \eqref{eq:sigma-plus}. The Gromov trick turns $u$ into a pseudo-holomorphic map $v: T \rightarrow T \times M$, for the almost complex structure \eqref{eq:Gromov}, and we can compute the $w$-th derivative \eqref{eq:jetdef} of $v$ in suitable local coordinates around $T \subset D \subset T \times M$. From \eqref{eq:normalvector} one sees that this yields a well-defined normal vector to $T \times D \subset T \times M$, independent of local coordinates. Since that normal bundle is the same as that of $D \subset M$, we can write the outcome as
\begin{equation} \label{eq:u-jet}
\mathit{ev}_{+\infty}^w(u) \in (\nu D)_{u(+\infty)}.
\end{equation}

Take a Morse function $f_{\partial N}$ and a corresponding pseudo-gradient vector field, of the kind studied in Section \ref{section:Morsecircle}. We generally denote the critical points of $f_{\partial N}$ by $c^{\partial N}$. Recall that each such point lies in the preimage of some critical point $c$ of $f_D$, and is fibrewise either a local minimum or maximum; we accordingly write $c^{\partial N} = c^{\mathrm{min}}$ or $c^{\partial N} = c^{\mathrm{max}}$. The Morse-theoretic part of our construction involves a negative half-flowline $b$ for $f_{\partial N}$, with limit $c_+^{\partial N}$. There is an additional variable $r \geq 0$, which enters into the jet incidence condition
\begin{equation} \label{eq:jet-incidence}
\mathit{ev}_{+\infty}^w(u) = r \cdot b(0).
\end{equation}
If no point of $\Sigma$ lies at $+\infty$, the intersection multiplicity is $\mu_{+\infty}(u) = w$, which by \eqref{eq:nonvanishing} means that $\mathit{ev}_{+\infty}^w(u) \neq 0$. As a consequence, one can then rewrite \eqref{eq:jet-incidence} in the more familiar form
\begin{equation} \label{eq:jet-incidence-2}
\frac{\mathit{ev}_{+\infty}^w(u)}{\|\mathit{ev}_{+\infty}^w(u)\|} = b(0), \;\;
r = \|\mathit{ev}_{+\infty}^w(u)\| > 0.
\end{equation}
On the other hand, if a marked point does lie at $+\infty$, the condition becomes
\begin{equation} \label{eq:jet-incidence-3}
\mathit{ev}_{+\infty}^r(u) = 0, \;\; u(+\infty) = \pi_{\partial N}(b(0)) \in D, \;\; r = 0.
\end{equation}

The space of $(\Sigma,u,b,r)$ satisfying \eqref{eq:jet-incidence} will be denoted by $\AT_{w,m}(x_-,c_+^{\partial N})$. By projecting $b$ to $D$, and forgetting $r$, one gets a map to the moduli space from Section \ref{section:thimbleswithtangency}:
\begin{equation} \label{eq:project-half-flow-line}
\AT_{w,m}(x_-,c_+^{\partial N}) \longrightarrow \frakT_{w,m}(x_-,c_+).
\end{equation}
We will need two regularity conditions which refine those in the definition of \eqref{eq:graded-t}:
\begin{enumerate}[label=(AT\arabic*)] \itemsep.5em
\parindent0em \parskip.5em \setcounter{enumi}{1}
\item \label{item:at-main}
{\em (Main stratum)}
Consider the subspace of $\AT_{w,m}(x_{-},c_{+}^{\partial N})$ where the points of $\Sigma$ are pairwise distinct, and none of them is equal to $+\infty$. We assume that this moduli space is regular; its dimension will then be $\operatorname{deg}(x_{-})-\operatorname{deg}(c_{+}^{\partial N})+2m$.

\item \label{item:at-collision}
{\em (Collision, no marked point at $+\infty$)}
This is the analogue of \ref{item:s-collision}, adding the condition that no marked point should lie at $+\infty$; the dimension is $\mathrm{deg}(x_-) - \mathrm{deg}(c_+^{\partial N}) + 2|\Pi|$.
\end{enumerate}
 
 

\begin{lem} \label{th:at0}
Consider spaces $\AT_{w,m}(x_{-},c_{+}^{\partial N})$ of dimension $\leq 1$. 

(i) Everywhere in that space, $\Sigma$ consists of $m$ pairwise distinct points.

(ii) If the dimension is $0$ or $c_+^{\partial N} = c_+^{\mathrm{min}}$ is a fibrewise minimum, no point of $\Sigma$ can lie at $+\infty$.
\end{lem}

\begin{proof}
(ii) If at least one the points of $\Sigma$ lies at $+\infty$, the local intersection multiplicity at that point increases. The condition \ref{item:at-forget} implies that after projecting the flow half-line to $D$, we get an element of $\frakT_{w+k,m-k}(x_-,c_+)$ for some $k>0$. That space satisfies
\begin{equation} \label{eq:diff-dim}
\begin{aligned}
& \mathrm{dim} \, \frakT_{w+k,m-k}(x_-,c_+) = 
\mathrm{deg}(x_-) - \mathrm{deg}(c_+) + 2m - 2k
\\ & \quad = \mathrm{dim} \, \AT_{w,m}(x_-,c_+^{\partial N}) - 2k +
\begin{cases} 
0 & \text{if } c_+^{\partial N} = c_+^{\mathrm{min}} \\
1 & \text{if } c_+^{\partial N} = c_+^{\mathrm{max}}.
\end{cases}
\end{aligned}
\end{equation}
Under the assumptions we have imposed, this dimension is negative (and a fortiori, so is that of the strata in $\frakT_{w+k,m-k}(x_-,c_+)$ where other marked points coincide). Hence, those spaces will be empty, as part of the transversality assumptions underlying the construction in Section \ref{section:thimbleswithtangency}.

(i) As long as no marked point lies at $+\infty$, the result follows immediately from \ref{item:at-collision}. For the remaining cases, one instead appeals to the same strategy as in (ii).
\end{proof}

\begin{lem} \label{lem:bijectionATT}
Take a space $\AT_{w,m}(x_-,c_+^{\mathrm{min}})$ of dimension $0$. Then \eqref{eq:project-half-flow-line} is an isomorphism.
\end{lem}

\begin{proof}
From Lemma \ref{th:at0}(ii) we see that $u$ determines $r$ and $b$ uniquely, which means that \eqref{eq:project-half-flow-line} is injective. Take an element of $\frakT_{w,m}(x_-,c_+)$. By \eqref{eq:diff-dim} this has $\mu_{+\infty}(u) = w$, hence we can lift the half-flow line to $\partial N$ so that \eqref{eq:jet-incidence-2} is satisfied. For dimension reasons, we have $\AT_{w,m}(x_-,c_+^{\mathrm{max}}) = \emptyset$, so the lifted half-flow-line must converge to $c_+^{\mathrm{min}}$. Hence, \eqref{eq:project-half-flow-line} is surjective as well.
\end{proof}

We will need to spend some time discussing the case omitted in Lemma \ref{th:at0}, of one-dimensional spaces $\AT_{w,m}(x_-,c_+^{\mathrm{max}})$. In such spaces, one can have points where \eqref{eq:jet-incidence-3} occurs. For dimension reasons, $\Sigma$ will still consist of pairwise distinct points, let's say $\Sigma = (z_1,\dots,z_m)$ with $z_m = +\infty$, and therefore $\mu_{+\infty}(u) = w+1$. We call these ``expected boundary points'', even though it is not a priori what the local structure of the moduli space is (in the standard Banach space setup for $\AT_{w,m}(x_-,c_+^{\mathrm{max}})$, the ``expected boundary points'' are not regular).

\begin{lemma} \label{th:expected-boundary}
There is a bijection between ``expected boundary points'' and points of the zero-dimensional space $\frakT_{w+1,m-1}(x_-,c_+)$ (hence also with $\AT_{w+1,m-1}(x_-,c_+^{\mathrm{min}})$, by Lemma \ref{lem:bijectionATT}).
\end{lemma}

\begin{proof}
The map to $\frakT_{w+1,m-1}(x_-,c_+)$ is defined as in \eqref{eq:project-half-flow-line}, by projecting the half-flow line. It is bijective by Lemma \ref{th:stable-unstable}(i).
\end{proof}

\begin{lemma} \label{th:tilde}
Near an ``expected boundary point'', consider the larger space $\widetilde{\AT}_{w,m}(x_-,c_+^{\mathrm{max}})$ where the condition $u(z_m) \in D$ has been dropped (but note that the Cauchy-Riemann equation satisfied by $u$ still depends on $z_m$, through our choice of auxiliary data). This larger space is smooth, having coordinates $(\zeta,r) \in \bC \times \bR^{\geq 0}$ where $\exp(2\pi i z_m) = 1/\zeta$ (so that the ``expected boundary point'' itself lies at the origin).
\end{lemma}

\begin{proof}
First, look at the subspace of $\widetilde{\AT}_{w,m}(x_-,c_+^{\mathrm{max}})$ where we set $\zeta = 0$ ($z_m = +\infty$) and $r = 0$. The ``expected boundary point'' is a regular point of this zero-dimensional subspace; this statement is equivalent to the regularity of its image in $\AT_{w+1,m-1}(x_-,c_+)$, by a comparison of linearized operators. Adding back the parameters $(\zeta, r)$ then yields a larger, but still regular, space.
\end{proof}

\begin{lemma} \label{th:at-nasty}
Near an ``expected boundary point'', make $\AT_{w,m}(x_-,c_+^{\mathrm{max}})$ smaller by requiring that $r \geq \epsilon$, for generic sufficiently small $\epsilon$. The outcome is a one-manifold with finitely many boundary points. Moreover, the algebraic count of such boundary points is $\pm 1$.
\end{lemma}

\begin{proof}
Denote the maps in $\widetilde{\AT}_{w,m}(x_-,c_+^{\mathrm{max}})$ by $u_{\zeta,r}$. Inside that space, $\AT_{w,m}(x_-,c_+^{\mathrm{max}})$ is the subset formed by $(0,0)$ and all those $(\zeta,r)$, $r>0$, such that 
\begin{equation}
u_{\zeta,r}(z_m) \in D.
\end{equation}
Recall that $\zeta$ is just $z_m$ in certain local coordinates near $+\infty \in T$. Hence, after applying the Gromov trick to reduce to pseudo-holomorphic maps, the structure of the boundary obtained by cutting off at $r = \epsilon$ is precisely that described by Lemma \ref{th:winding-number-argument}. There, the algebraic count was given as $+1$; but depending on the sign of the ``expected boundary point'' as an element of $\AT_{w+1,m-1}(x_-,c_+)$, the Floer-theoretic orientation of $\widetilde{\AT}_{w,m}(x_-,c_+^{\mathrm{max}})$ will be equal or opposite to the standard orientation of the local coordinates $\bC \times \bR^{\geq 0}$; hence there's an additional sign that occurs here.
\end{proof}

\begin{remark}
Readers who find the local structure near ``expected boundary points'' disturbingly complicated may want to pursue the following instead. There is a global analogue of \eqref{eq:j-alpha} involving a choice of unitary connection on $\nu D$, which yields a class of compatible almost structures with restricted behaviour near $D$ (compatibility works because one can take the curvature of the connection to be $(2\pi/i)(\omega|D)$, see Remark \ref{th:compatible}; or one could drop that condition and work with tame almost complex structures). For our Cauchy-Riemann equations, we then ask that locally around $D$, the families of almost complex structures $(J_z)_{z \in T}$ should, near $z = +\infty$, be $z$-independent and of such restricted type; and similarly, that the inhomogeneous term should vanish near $z = +\infty$. Then, one can replace the use of Lemma \ref{th:winding-number-argument} in the proof of Lemma \ref{th:at-nasty} with Lemma \ref{th:winding-number-argument-2}: as a consequence, the one-dimensional spaces $\AT_{m,w}(x_-,c_+^{\mathrm{max}})$ become manifolds with boundary, with the ``expected boundary points'' as their genuine boundaries. Ultimately, either approach yields the same counting formula \eqref{eq:grij}.
\end{remark}

Our spaces have obvious compactifications $\overline{\AT}_{w,m}(x_-,c_+^{\partial N})$. Their structure is easy to analyze, by dimension-counting, and bearing in mind that bubbling of holomorphic spheres is a priori ruled out by energy considerations. We only record the outcome:

\begin{lem} \label{lem:at1}
(i) In the zero-dimensional case, we have $\AT_{w,m}(x_-,c_+^{\partial N}) = \overline{\AT}_{w,m}(x_-,c_+^{\partial N})$.

(ii) In the one-dimensional case, $\overline{\AT}_{w,m}(x_-,c_+^{\partial N}) \setminus \AT_{w,m}(x_-,c_+^{\partial N})$ consists of two kinds of points:
\begin{itemize} \itemsep.5em
\item A pseudo-gradient trajectory splits off from $b$. On the cylinder component, we still have pairwise distinct marked points, none of which lies at $+\infty$.

\item The Riemann surface splits into two pieces, one a cylinder and the other a thimble. The marked points are still pairwise distinct, and (for those on the thimble) none of them lie at $+\infty$.
\end{itemize}
In both cases, the local structure near a point $\overline{\AT}_{w,m}(x_-,c_+^{\partial N}) \setminus \AT_{w,m}(x_-,c_+^{\partial N})$ is that of a one-manifold with boundary.
\end{lem}

Counting points in zero-dimensional moduli spaces $\AT_{w,m}(x_-,c_+^{\partial N})$ defines operations
\begin{align} \label{eq:thetawm} 
\mathit{at}_{w,m}: \mathit{CM}^*(\partial N) \longrightarrow \mathit{Gr}^{-m-w}\mathit{CF}^{*-2m}(m+w). 
\end{align}
The previous geometric considerations translate into the following algebraic properties, involving the maps between Morse complexes constructed in Section \ref{sec:morse}. Lemma \ref{lem:bijectionATT} says that:
\begin{equation} \label{eq:atpi}
\mathit{at}_{w,m} \circ \pi_{\partial N}^*= Gr^{-m-w}t_{w,m}.
\end{equation} 
Next, we use Lemma \ref{lem:at1} combined with Lemma \ref{th:at-nasty} (which means that we are applying the standard signed-count-of-boundary-points-is-zero argument after cutting out a neighbourhood of each ``expected boundary point''). The outcome is
\begin{equation} \label{eq:grij}
\sum_{i+j=m} (\mathit{Gr}^{-i}d_i) \circ \mathit{at}_{w,j} =
\mathit{at}_{w,m} \circ d_{\partial N} + \mathit{at}_{w+1,m-1} \circ \pi^*_{\partial N} \pi_{\partial N,*},
\end{equation}
where it's understood that the last term is omitted for $m = 0$. Because of \eqref{eq:grij}, the $\mathit{at}_{w,m}$ can be combined into a chain map from the complex \eqref{eq:q-deform-dn} to that in \eqref{eq:g-complex}, for any $w<0$:
\begin{equation} \label{eq:thetaw}
\begin{aligned}
& a_w : CM(\partial N)_q \longrightarrow G^{-w}, \\
& a_w(c^{\partial N} q^k) = \sum_{m\geq 0} q^{k+m} \mathit{at}_{w+k,m}(c^{\partial N}).  
\end{aligned}
\end{equation}
Furthermore, \eqref{eq:atpi} says that this fits into a commutative diagram
\begin{equation} \label{eq:negative-diagram}
\xymatrix{
CM(D) \ar[d]_-{\eqref{eq:qdefo}} \ar[drr]^-{\mathit{Gr}^{-w}t_{C_q,w}} \\
CM(\partial N)_q \ar[rr]_-{\eqref{eq:thetaw}} && G^{-w}
}
\end{equation}

\begin{proposition} \label{ref:thmnegativeaction} 
The map \eqref{eq:thetaw} is a quasi-isomorphism.
\end{proposition}

\begin{proof} 
The map is filtered with respect to the (decreasing, bounded above, complete) filtration by powers of $q$ on both sides. The associated graded with respect to this filtration consists of
\begin{equation}
\mathit{at}_{w+k,0}: \mathit{CM}^*(\partial N) \longrightarrow
\mathit{Gr}^{-w-k} \mathit{CF}(w+k)
\end{equation}
for all $k \geq 0$. By definition these are the low-energy versions of the log PSS maps studied in \cite{ganatra-pomerleano20}, which are shown to be quasi-isomorphisms in \cite[Theorem 4.30]{ganatra-pomerleano20}. 
\end{proof}

Proposition \ref{ref:thmnegativeaction}, Lemma \ref{qdeformedisD} and \eqref{eq:negative-diagram} together imply the second part of Theorem \ref{th:key}.

\subsection{Approximate action $0$\label{subsec:zero-action}}
For this, we use almost complex structures and inhomogeneous terms as in Section \ref{subsec:classical-thimble}, again with an extra consistency condition:
\begin{enumerate}[label=(AS\arabic*)] \itemsep.5em
\item
\label{item:as-consistency}
Take the embedding $\frakT_{1,m-1} \rightarrow \frakS_m$ which adds a marked point at $+\infty$. The data should be should be chosen consistently with this embedding.
\end{enumerate}
As in the definition of $s_{C_q}$ we consider pairs $(\Sigma,u)$, where $\Sigma \in \frakS_m$ and $u: T \rightarrow M$ satisfies \eqref{eq:intersect-sigma}; but we only allow only thimble maps where the limit $x_-$ has approximate action $-m$, which is what enters into the graded piece \eqref{eq:graded-s} of $s_{C_q}$. On the blowup $N$, choose a Morse function $f_N$ and pseudo-gradient vector field as in Section \ref{subsection:Morseblowup}. Given a critical point $c_+^N$, we consider half-flow lines $b$ converging to that point, joined to $u$ by requiring that
\begin{equation} \label{eq:pi-incidence}
u(+\infty) = \pi_N(b(0)).
\end{equation}
Denote the space of such $(\Sigma,u,b)$ by $\AS_m(x_-,c_+^N)$. By Corollary \ref{th:project-stable}, the condition \eqref{eq:pi-incidence} implies that $u(+\infty) \in W^s(M,c_+)$, where $c_+ = \pi_N(c_+^N)$. By replacing $b$ with the flow half-line in $M$ which starts at $u(+\infty)$, one gets a map
\begin{equation} \label{eq:as-s}
\AS_m(x_-,c_+^N) \longrightarrow \frakS_m(x_-,c_+).
\end{equation}
Furthermore, another look at Corollary \ref{th:project-stable}(i) shows that:

\begin{lem} \label{th:as-0}
Let $c_+$ be a critical point of $f_M$ lying in $M \setminus D$, and which therefore corresponds uniquely to some $c_+^N$. Then the map \eqref{eq:as-s} is an isomorphism.
\end{lem}

Transversality questions for $\AS_m(x_-,c_+^N)$ can be viewed as making the evaluation map $(\Sigma,u) \longmapsto u(+\infty)$ transverse to $\pi_N|W^s(c_+^N): W^s(c_+^N) \rightarrow M$. Recall that if $c_+^N$ lies in $\partial N$, then $W^s(c_+^N)$ is a manifold with boundary $\partial W^s(c_+^N) = W^s(c_+^N) \cap \partial N = W^s(\partial N,c_+^N)$ (Lemma \ref{th:n0}). Let's look at the transversality issue on the boundary, and more precisely this situation:

\begin{lemma} \label{th:boundary-regularity}
Take $c_+^N$ lying on $\partial N$, and consider the subspace 
\begin{equation} \label{eq:as-subspace}
\{(\Sigma,u,b) \;:\; \text{$\Sigma$ consists of distinct points, and }
b(0) \in \partial N\} \subset \AS(x_-,c_+^N).
\end{equation}
Assuming generic choices of auxiliary data, this will be regular; and hence, points in that subspace are smooth boundary points of $\AS(x_-,c_+^N)$.
\end{lemma}

\begin{proof}
The assumption on $b(0)$ requires one point of $\Sigma$ to lie at $+\infty$; let's write $\Sigma = (z_1,\dots,z_m)$ with $z_m = +\infty$. The linearization of the equations defining \eqref{eq:as-subspace} is an operator
\begin{equation} \label{eq:n-extended-linearized}
\xymatrix{
\scrE_{\partial N} = W^{2,2}(u^*TM) \oplus T_{z_1}T \oplus \cdots \oplus T_{z_m}T \oplus T_{b(0)}W^s(\partial N,c_+^N)
\ar[d] \\
\scrF_{\partial N} = W^{1,2}(u^*TM) \oplus \nu_{u(z_1)}D \oplus \cdots \oplus \nu_{u(z_m)}D \oplus T_{u(+\infty)}M.
}
\end{equation}
The notation needs some explaining (including an apology for the overlap between $T = \text{\it thimble}$ and $T = \text{\it tangent spaces}$). The $W^{2,2}(u^*TM) \rightarrow W^{1,2}(u^*TM)$ component is the standard linearized Cauchy-Riemann operator. On the domain, the direct sum of the $T_{z_k}M$ equals $T_{\Sigma}\frakS_m$; and $T_{b(0)}W^s(\partial N,c_+^N)$ expresses the freedom to move $b$ (only within the boundary). On the target space, the $\nu_{u(z_k)} D$ components measure the failure of an infinitesimal deformation to preserve $u(z_k) \in D$; and the $T_{u(+\infty)}M$ component linearizes \eqref{eq:pi-incidence}. Even though in our situation $z_m = +\infty$, we have kept distinguishing those points notationally, as they enter into the definition of the moduli space in different ways: notably, restricting \eqref{eq:n-universal-operator} to $T_{z_m}T$ gives a map with image in $\nu_{u(z_m)}D$ (the derivative of $u$) but with trivial $T_{u(+\infty)}M$-component, because moving $z_m$ doesn't affect the condition \eqref{eq:pi-incidence}.

To apply the classical Palais-Smale argument, one considers a universal version of \eqref{eq:n-extended-linearized}, which includes a suitable infinite-dimensional space of auxiliary data:
\begin{equation} \label{eq:n-universal-operator}
\text{\it (infinitesimal deformations of the data)} \oplus \scrE_{\partial N} \longrightarrow \scrF_{\partial N}.
\end{equation}
Standard transversality-of-evaluation shows that the composing \eqref{eq:n-universal-operator} with the obvious inclusions and projections yields a surjective map
\begin{equation} \label{eq:omit-m}
\begin{aligned}
& \text{\it (infinitesimal deformations of the data)} \oplus W^{2,2}(u^*TM) 
\\ & \qquad
\hookrightarrow \text{\it (infinitesimal deformations of the data)} \oplus \scrE_{\partial N}
\xrightarrow{\eqref{eq:n-universal-operator}} \scrF_{\partial N} 
\\ & \qquad
\twoheadrightarrow
W^{1,2}(u^*TM) \oplus \nu_{u(z_1)}D \oplus \cdots \oplus \nu_{u(z_{m-1})}D \oplus T_{u(+\infty)}M.
\end{aligned}
\end{equation}
For this, it is important that the evaluation points are pairwise distinct, which is why we have omitted the $\nu_{u(z_m)}D$ factor. Equivalently, if we restrict \eqref{eq:n-universal-operator} to {\it (infinitesimal deformations of the data)} $\oplus W^{2,2}(u^*TM)$, then its image equals the kernel of
\begin{equation} \label{eq:difference-z}
\scrF_{\partial N} \twoheadrightarrow \nu_{u(z_m)}D \oplus T_{u(+\infty)}M 
\xrightarrow{(\xi,\eta) \mapsto [\xi-\eta]} \nu_{u(z_m)}D.
\end{equation}
On the other hand, since only one of the points of $\Sigma$ lies at $+\infty$, we have $\mu_{+\infty}(u) = 1$, which means that the map $u$ intersects $D$ transversally at that point. Hence, the composition
\begin{equation}
T_{z_m}T \hookrightarrow \scrE_{\partial N} \xrightarrow{\eqref{eq:n-universal-operator}}
\scrF_{\partial N} \xrightarrow{\eqref{eq:difference-z}} \nu_{u(z_m)}D
\end{equation}
is onto. The two last-mentioned facts together imply surjectivity of \eqref{eq:n-universal-operator}.
\end{proof}

With that in mind, we can impose the following conditions:
\begin{enumerate}[label=(AS\arabic*)] \itemsep.5em
\parindent0em \parskip.5em \setcounter{enumi}{1}
\item \label{item:as-main}
{\em (Main stratum, includes a marked point at $+\infty$)}
Consider the subspace of $\AS_{w,m}(x_{-},c_{+}^N)$ where the points of $\Sigma$ are pairwise distinct.
We assume that this space is regular, with boundary where $b(0) \in \partial N$. Its dimension will then be $\operatorname{deg}(x_{-})-\operatorname{deg}(c_{+}^N)+2m$.
\item \label{item:as-collision}
{\em (Collision, no marked point at $+\infty$)}
This is the direct analogue of \ref{item:at-collision}, including the assumption that no marked point should lie at $+\infty$.
\end{enumerate}

\begin{lemma} \label{th:as1}
Consider spaces $\AS_m(x_-,c_+^N)$ of dimension $\leq 1$.

(i) Everywhere in such a space, $\Sigma$ consists of $m$ distinct points.

(ii) Suppose that the dimension is $0$, or $c_+^N$ lies in the interior, or $c_+^N = c_+^{\mathrm{min}}$ is a fibrewise minimum on $\partial N$; then no point of $\Sigma$ can lie at $+\infty$, hence \eqref{eq:pi-incidence} happens in $M \setminus \partial M$.
\end{lemma}

\begin{proof}
This is the counterpart of Lemma \ref{th:at0}, and the argument is basically similar, so we will omit most of it. For part (ii), if a point lies at $+\infty$ then the same procedure as in \eqref{eq:as-s} maps $(\Sigma,u,b)$ to an element of $\frakT_{1,m-1}(x_-,c_+)$; but the dimension of that space is negative.
\end{proof}

\begin{lemma} \label{lem:bijectionATT2}
Consider a space $\AS_m(x_-,c_+^{\mathrm{min}})$ of dimension $0$. Then \eqref{eq:as-s} is an isomorphism.
\end{lemma}

\begin{proof}
This is the counterpart of Lemma \ref{lem:bijectionATT}. Because of Lemma \ref{th:as1}(ii), $u$ determines $b$ uniquely in this situation, hence \eqref{eq:as-s} is injective. Conversely, take a point of $\frakS_m(x_-,c_+)$. As before, no point of $\Sigma$ can lie at $+\infty$ for dimension reasons, hence $u(+\infty)$ lies in $W^s(M,c_+) \cap (M \setminus D)$, which by Lemma \ref{th:n1} is identified with $(W^s(N,c_+^{\mathrm{min}}) \cup W^s(N,c_+^{\mathit{max}}) \cap (N \setminus \partial N)$. Since $\AS_m(x_-,c_+^{\mathrm{max}}) = \emptyset$ for dimension reasons, the only option is that $u(+\infty) \in W^s(c_+^{\mathrm{min}})$, which yields a point in $\AS_m(x_-,c_+^{\mathrm{min}})$.
\end{proof}

The counterpart of Lemma \ref{th:expected-boundary}, which follows in the same way from Lemma \ref{th:stable-unstable}(i), is:

\begin{lemma} \label{lem:as-boundary}
For a one-dimensional space $\AS_m(x_-,c_+^{\mathrm{max}})$, there is a bijection between the boundary $\partial \AS_m(x_-,c_+^{\mathrm{max}})$ and $\frakT_{1,m-1}(x_-,c_+)$ (or equivalently $\AT_{1,m-1}(x_-,c_+^{\mathrm{min}})$, in view of Lemma \ref{lem:bijectionATT}).
%
\end{lemma}

The Gromov compactification $\overline{\AS}_{m}(x_{-},c_{+}^{N})$ of spaces of dimension $\leq 1$ adds points in exactly the same way as in Lemma \ref{lem:at1}. We use the zero-dimensional spaces to define operations
\begin{equation}
\mathit{CM}^*(N) \longrightarrow \mathit{Gr}^{-m} \mathit{CF}^{*-2m}(m).
\end{equation}
In parallel with \eqref{eq:atpi}, Lemma \ref{lem:bijectionATT2} translates into
\begin{equation} \label{eq:thetabreaking2}
\mathit{as}_m \circ \pi_N^* = \mathit{Gr}^{-m}(s_m).
\end{equation}
Counting boundary points in one-dimensional spaces $\overline{\AS}_m(x_-,c_+^{\mathrm{max}})$ (those from Lemma \ref{lem:as-boundary}, as well as those added by the compactification) yields the following counterpart of \eqref{eq:grij}:
\begin{equation} \label{eq:identitythetavss}
\sum_{i+j=m} \mathit{Gr}^{-i}d_i \circ \mathit{as}_j = \mathit{as}_j \circ d_N + 
\mathit{at}_{1,m-1} \circ \pi^*_{\partial N} \pi_{\partial N,*} i_{\partial N}^*.
\end{equation}
Therefore, the following is a chain map from \eqref{eq:n-q} to the $K = 0$ case of \eqref{eq:g-complex}:
\begin{equation} \label{eq:theta0def}
\begin{aligned}
& a_0: \mathit{CM}(N)_q \longrightarrow G^0, \\
& a_0(c^N) = \sum_m q^m \mathit{as}_m(c^N), \\[-.5em]
& a_0(c^{\partial N} q^w) = \sum_m q^{w+k} \mathit{at}_{w,m}(c^{\partial N}) \text{ for $w>0$.}
\end{aligned}
\end{equation}
Moreover, by \eqref{eq:identitythetavss} this fits into a commutative diagram
\begin{equation} \label{eq:zero-diagram}
\xymatrix{
\mathit{CM}(M) \ar[d]_-{\pi_{N}^*} \ar[drr]^-{\eqref{eq:graded-s}} \\
\mathit{CM}(\partial N)_q \ar[rr]_-{\eqref{eq:theta0def}} && G^{0}
}
\end{equation}

\begin{proposition} \label{prop:theta0quasiiso}
The map \eqref{eq:theta0def} is a quasi-isomorphism. 
\end{proposition}

\begin{proof} 
This follows from essentially the same argument as in Proposition \ref{ref:thmnegativeaction}.  Namely, the map respects the $q$-filtrations, and its associated graded consists of
\begin{align}
\label{eq:classical-pss}
& \mathit{as}_0: \mathit{CM}^*(N) \longrightarrow \mathit{Gr}^0 \mathit{CF}(0), \\
\label{eq:as-before}
& \mathit{at}_{w,0}: \mathit{CM}^*(\partial N) \longrightarrow \mathit{Gr}^{-w}\mathit{CF}(w) \text{ for $w>0$.}
\end{align}
Here \eqref{eq:classical-pss} is a version of the classical Piunikhin-Salamon-Schwarz map for the manifold $M \setminus D$, hence a quasi-isomorphism; the other pieces \eqref{eq:as-before} are quasi-isomorphisms by \cite[Theorem 4.30]{ganatra-pomerleano20}. 
\end{proof}

Combining Proposition \ref{prop:theta0quasiiso}, Lemma \ref{lem:NqdeformedisM}, and \eqref{eq:zero-diagram} yields the first part of Theorem \ref{th:key}.


\section{The equivariant theory\label{sec:equivariant}}
The constructions from Sections \ref{sec:sh} and \ref{sec:thimbles} have $S^1$-equivariant (with respect to loop rotation) extensions, based on versions of the previously considered moduli spaces with added parameters; the added parameters can be thought of as representing cycles in $BS^1$. From a geometric or analytic viewpoint  there are no new issues here, and the same applies to most of the material in the sections after this one. For that reason, details will from now on be given only where they seem particularly important; otherwise, the exposition will be reduced to its structural skeleton.

\subsection{The equivariant differential} \label{sec:equivariantdiff}
The initial piece of the equivariant differential is the BV or loop-rotation operator, a chain map
\begin{equation} \label{eq:big-delta-1}
d_{C_q}^1: C_q \longrightarrow C_q[-1],
\end{equation}
constructed as follows.
\begin{itemize} \itemsep.5em
\item 
The first ingredient of \eqref{eq:big-delta-1} are maps
\begin{equation} \label{eq:big-delta-m}
\xymatrix{
\mathit{CF}^{*-1-2m}(w+m) &&
\ar[ll]^-{\displaystyle d_m^1}_-{\includegraphics[valign=c]{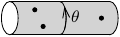}} \mathit{CF}^*(w)
}
\end{equation}
which interact with the previously defined differentials $d_m$ as follows:
\begin{equation}
\sum_{i+j=m} d_i d_j^1 + d_i^1 d_j = 0.
\end{equation}
The maps \eqref{eq:big-delta-m} are defined using parameter spaces of cylinders with marked points and a marked circle, where the circle (drawn with an arrow above) is decorated with an angle $\theta \in S^1$. The specific property of this construction is that, as the cylinder breaks up into pieces, those on the left of the angle-decorated component are rotated by $-\theta$, while those on the right are not; see Figure \ref{fig:twisted-gluing} for an example.
\item
The second ingredient is a version of the first one, where we do not divide by translation. The outcome are maps
\begin{equation} \label{eq:big-delta-m-dag}
\xymatrix{
\mathit{CF}^{*-2-2m}(w+m+1) &&
\ar[ll]^-{\displaystyle d_m^{1,\dag}}_-{\includegraphics[valign=c]{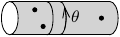}} \mathit{CF}^*(w)
}
\end{equation}
which satisfy the more complicated equation
\begin{equation}
\sum_{i+j=m} d_i d_j^{1,\dag} - d_i^{1,\dag} d_j - d_i^\dag d_j^1 + d_i^1 d_j^\dag = 0.
\end{equation}
The picture \eqref{eq:big-delta-m-dag} may be a bit confusing: following our usual habit, we have recorded the breaking of translation-invariance by formally drawing a circle on our cylinder. That circle does not carry any angle, and its relative position to the actual angle-decorated circle is arbitrary (in particular, the two can coincide). In the simplest case $m = 0$, the relevant parameter space is $\bR \times S^1$ (relative $s$-position of the two circles, and angle).
\end{itemize}
\begin{figure}
\begin{centering}
\includegraphics{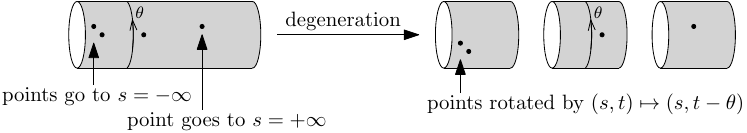}
\caption{\label{fig:twisted-gluing}An example of a degeneration from the parameter space underlying \eqref{eq:big-delta-m}, explaining how the angle $\theta$ affects the limiting configuration.}
\end{centering}
\end{figure}
One then defines \eqref{eq:big-delta-1} by setting, for $x \in \mathit{CF}^*(w)$,
\begin{equation} \label{eq:bv-components}
\begin{aligned}
& d^1_{C_q}(x) = \sum_m q^m d^1_m(x), \\
& d^1_{C_q}(\eta x) = \sum_m q^m (-\eta d^1_m(x) + d^{1,\dag}_m(x)).
\end{aligned}
\end{equation}

For the full equivariant differential, one uses cylinders that carry any number $l \geq 1$ of circles $\{s = \sigma^i\}$, with $\sigma^1 \leq \cdots \leq \sigma^l$, decorated with angles $\theta^1,\dots,\theta^l$. This approach was introduced, with slightly different language, in \cite[Section 4.3]{ganatra23}, and has been used widely since then (e.g.\ \cite[Section 4.1]{li24} or \cite[Sections 5.2c--5.2d]{pomerleano-seidel23}). We will not draw the pictures here, but only summarize the outcome.
\begin{itemize}
\item The relevant generalizations of \eqref{eq:big-delta-m} and \eqref{eq:big-delta-m-dag} are of the form
\begin{align}
\label{eq:dml} & d_m^l: \mathit{CF}^*(w) \longrightarrow \mathit{CF}^{*+1-2l-2m}(w+m), \\
& d_m^{l,\dag}: \mathit{CF}^*(w) \longrightarrow \mathit{CF}^{*-2l-2m}(w+m+1).
\end{align}
\end{itemize}
For notational simplicity, we also include $d_m$ and $d_m^\dag$ as the special case $l = 0$. Add up those components to endomorphisms $d^l_{C_q}$ as in \eqref{eq:bv-components}. The equivariant differential on 
\begin{equation}
C_{u,q} = C[[u,q]]
\end{equation}
extends $d_{C_q}$ with higher order $u$ terms,
\begin{equation}
d_{C_{u,q}} = \sum_{l \geq 0} u^l d^l_{C_q}.
\end{equation}

\subsection{Equivariant thimble maps}
We return to the situations from Section \ref{subsec:classical-thimble} and \ref{section:thimbleswithtangency}.
\begin{itemize} \itemsep.5em
\item 
Generalizing \eqref{eq:d-t} (which is the special case $l = 0$) one defines, for all $m,l \geq 0$, maps
\begin{equation} \label{eq:equivariant-d-s}
\xymatrix{
\mathit{CF}^{*-2m-2l}(m) 
&&& \ar[lll]_-{\includegraphics{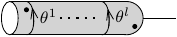}}^{\displaystyle s_m^l} \mathit{CM}^*(M)
}
\end{equation}
This is the thimble version of \eqref{eq:dml}, meaning that the thimbles come with $l$ angle-decorated circles as well as $m$ marked points. The counterpart of \eqref{eq:d-s-morse} is
\begin{equation}
\sum_{\substack{i+j = m \\ u+v = l}} d_i^u s_j^v = s_m^l d_M.
\end{equation}

\item 
There is a parallel generalization of \eqref{eq:tt}, using tangency conditions and a Morse function on the divisor. It's hardly necessary to draw a picture, but here it is, for completeness:
\begin{equation} \label{eq:equivariant-d-t}
\xymatrix{
\mathit{CF}^{*-2m-2l}(w+m)
&&& \ar[lll]_-{\!\!\!\includegraphics[valign=c]{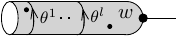}}^{\displaystyle t_{w,m}^l} \mathit{CM}^*(D)
}
\end{equation}
%
\end{itemize}
One combines \eqref{eq:equivariant-d-s} and \eqref{eq:equivariant-d-t} into equivariant extensions of \eqref{eq:total-s-map} and \eqref{eq:t-chain}:
\begin{equation} \label{eq:equivariant-quasi-iso}
\begin{aligned}
& s_{C_{u,q}} = \sum_{l,m} u^l q^m s_m^l: \mathit{CM}^*(M) \longrightarrow C_{u,q}, \\
& t_{C_{u,q},w} = \sum_{l,m} u^l q^m t_{w,m}^l : \mathit{CM}^*(D) \longrightarrow C_{u,q}. 
\end{aligned}
\end{equation}
Extend these $(u,q)$-linearly, and then take their direct sum. The outcome is the chain level map underlying \eqref{eq:equivariant-main}. By a $u$-filtration argument, Theorem \ref{th:main} implies that this is a quasi-isomorphism, which proves Corollary \ref{th:equivariant-main}.

\subsection{A brief look at the parameter spaces\label{subsec:angle-decorated}}
We consider two sample cases, the first being \eqref{eq:dml} (see \cite[Sections 5.2d and 5.3c]{pomerleano-seidel23} for a more detailed explanation of a closely related construction). The underlying parameter space (for $m+l > 0$) is
\begin{equation} \label{eq:d-space}
\frakD_m^l = \big(\Theta^l \times \mathit{Sym}_m(\bR \times S^1)\big) \;/\; \bR,
\end{equation}
where 
\begin{equation} \label{eq:delta-l}
\Theta^l = \{ \sigma^1 \leq \cdots \leq \sigma^l\} \times (S^1)^l 
\end{equation}
parametrizes the positions ($\sigma^i \in \bR$) of the angle-decorated circles, as well as their angles (the remaining variables, which we write at $\theta^i \in S^1$). The group $\bR$ acts by translation on the $\sigma^i$, leaving the $\theta^i$ fixed. As usual, the $l = 0$ case reduces to the previously used parameter spaces underlying the Floer differential. The choice of inhomogeneous data is set up as follows:
\begin{enumerate} [label=($\Delta$\arabic*)]  \itemsep.5em
\item \label{item:rotated-data}
For any point in $\frakD_m^l$, the associated data are chosen so that at the $s \ll 0$ end of the cylinder, they agree with the rotated Floer data $(\bar{H}_{m+w, t-(\theta^1 + \cdots + \theta^l)},\bar{J}_{m+w, t-(\theta^1 + \cdots + \theta^l)})$ (whereas there is no such rotation at the $s \gg 0$ end).

\item \label{item:coincident-circles}
The boundary stratum where $\sigma^i = \sigma^{i+1}$ comes with a map to $\frakD_m^{l-1}$, obtained by passing to $\theta^i + \theta^{i+1}$. We ask that the data are compatible with that map (in other words, on that boundary stratum, they are pulled back from the choices for $\frakD_m^{l-1}$). By definition, the pullback data are invariant under the $S^1$-action which increases $\theta^i$ and decreases $\theta^{i+1}$. As a consequence of that, these boundary components will not contribute to moduli spaces of dimension $\leq 1$.
\end{enumerate}
The compactification of the parameter space is, as a set,
\begin{equation}
\bar\frakD_m^l = \coprod_{\substack{r \geq 1 \\ m^1 + \cdots + m^R = m \\ l^1 + \cdots + l^R = l}} \frakD_{m^1}^{l^1} \times \cdots \times \frakD_{m^R}^{l^R}.
\end{equation}
The topology (or maybe more precisely, the identification of boundary strata with products of lower-dimensional moduli spaces) involves angle-twisting. Namely, as a sequence in $\frakD_m^l$ approaches the $(m^1,\dots,m^R,l^1,\dots,l^R)$ stratum, we take the naive $i$-th component of the limit and rotate it in $(-t)$-direction, by the sum of the angles in the components to its right (in particular, the $R$-th component is never rotated). The consistency condition for inhomogeneous data also follows that idea, and that is compatible with how we have set up \ref{item:rotated-data} above.

Let's carry over the previous discussion to the slightly more complicated case of \eqref{eq:equivariant-d-s}, where the relevant parameter space is
\begin{equation}
\frakS_m^l = \big(\{ \sigma^1 \leq \cdots \leq \sigma^l \} \times (S^1)^l \times \mathit{Sym}_m(T)\big) \;/\; \bR,
\end{equation}
for $\sigma^i \in (-\infty,+\infty]$ (for $\sigma^i = +\infty$, one can imagine the circle to have shrunk into the point at the tip of the thimble, but it still carries an angle). To make this compatible with the Riemann surface structure of the thimble at $+\infty$, we use the differentiable structure on the parameter space in which $\exp(-\sigma^i) \in [0,\infty)$ is the coordinate. The compactification is
\begin{equation}
\bar\frakS_m^l = \coprod_{\substack{R \geq 1 \\ m^1 + \cdots + m^R = m \\ l^1 + \cdots + l^R = l}} \frakD_{m^1}^{l^1} \times \cdots \times \frakD_{m^{R-1}}^{l^{R-1}} \times \frakS_{m^R}^{l^R},
\end{equation}
with angle-twisting as before. Concerning the data underlying the associated Cauchy-Riemann equations, we have:
\begin{enumerate} [label=($\Sigma$\arabic*)]  \itemsep.5em
\item As in \ref{item:rotated-data} for $s \ll 0$.
\item \label{item:sigma2}
We also have the analogue of \ref{item:coincident-circles} when $\sigma^i = \sigma^{i+1}$, pulling back data from $\frakS_m^{l-1}$.
\item \label{item:sigma3}
On the boundary stratum where $\sigma^l = +\infty$, we rotate the thimble (with its marked points) by $-\theta^l$, and then pull back the data from $\frakS_m^{l-1}$ by the map which forgets $\sigma^l$. As a consequence, the pullback data are invariant under the $S^1$-action which rotates the thimble, and simultaneously adds the same angle to $\theta^l$. Again, the effect is that those boundary strata cannot contribute to moduli spaces of dimension $\leq 1$.
\end{enumerate}
Of course, one has to check that the various conditions do not contradict each other at higher-codimension corner strata; see Figure \ref{fig:multiple-rotations} for an example.
\begin{figure}
\begin{centering}
\includegraphics{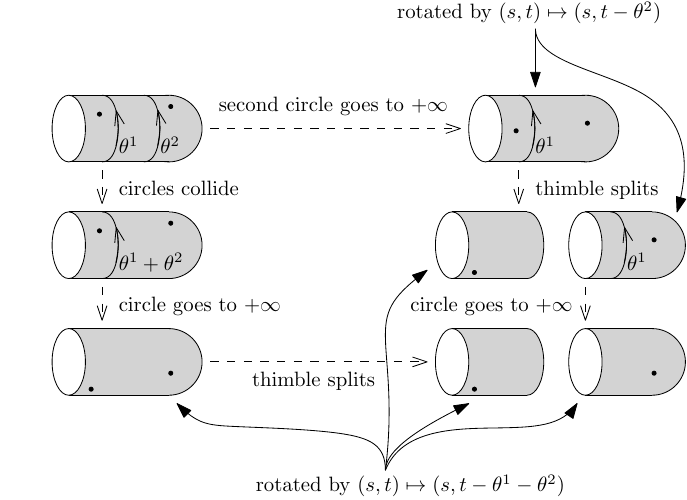}
\caption{\label{fig:multiple-rotations}
An example of \ref{item:sigma2} and \ref{item:sigma3}, showing how the inhomogeneous data are pulled back from thimbles with less data.}
\end{centering}
\end{figure}

\subsection{Compatibility with the action filtration}
Take Hamiltonians as in Lemma \ref{th:filtered-q-telescope}, and extend the condition on continuation maps imposed there to the equivariant differential. The outcome is that $d_{C_{u,q}}$ is compatible with the filtration $F^{\geq K}C_{u,q} = F^{\geq K}C_{q}[[u]]$ (which is again exhaustive, and bounded below in each degree). Similarly, one can arrange that the equivariant thimble maps take values in suitable pieces of the filtration. We will only need one such property, namely that 
\begin{equation} \label{eq:equivariant-nonnegative-energy}
s_{C_{u,q}}: \mathit{CM}(M)[u,q] \longrightarrow F^{\geq_0}C_{u,q} \subset C_{u,q}.
\end{equation}
A $u$-filtration argument based on the corresponding non-equivariant property, for which see the discussion following \eqref{eq:filtered-total-thimble}, shows that \eqref{eq:equivariant-nonnegative-energy} is a quasi-isomorphism. On the cohomology level, we therefore get a commutative diagram
\begin{equation} \label{eq:filtered-equivariant-thimble}
\xymatrix{
H^*(M)[u,q] \ar@{^{(}->}[d] \ar[rr]_-{\iso} && H^*(F^{\geq 0}C_{u,q}) \ar[d] \\
H^*(M)[u,q] \oplus \bigoplus_{w \geq 1} H^*(D)[u]z^w \ar[rr]_-{\iso} 
&& H^*(C_{u,q}) = \mathit{SH}^*_{u,q}(M,D).
}
\end{equation}

\begin{proof}[Proof of Lemma \ref{th:equivariant-q-torsion}]
Take a cocycle $x \in C_{u,q}$. Since multiplication by $q$ increases the action by $1$, see \eqref{eq:q-action}, there is some $r \geq 0$ such that $q^rx \in F^{\geq 0}C_{u,q}$. In particular, the cohomology class $q^r[x]$ can be lifted to $H^*(F^{\geq 0}C_{u,q})$. The desired result then follows by carrying over that fact from the right column of \eqref{eq:filtered-equivariant-thimble} to the left one.
\end{proof}

\section{Operations\label{sec:operations}}
This section discusses a cap-product-type endomorphism on deformed symplectic cohomology. Geome\-trically, this involves parameter spaces where one of the marked points is singled out, and its position is constrained, still keeping the incidence condition with $D$. We investigate the interaction of this operation with the BV operator and, via thimble maps, relate it to the quantum product with $[D]$. This serves mainly as a toy model for the equivariant theory in the subsequent section (the same expository strategy was used in \cite[Sections 4--5]{seidel18}). 

\subsection{The cap product endomorphism} 
The most basic operation,
\begin{equation}
\label{eq:iota-endomorphism} 
a_{C_q}: C_q \longrightarrow C_q,
\end{equation}
is defined as follows.
\begin{itemize}
\itemsep.5em
\item 
For every $w,m \geq 0$ we construct a map
\begin{equation} \label{eq:iota-m-map}
\xymatrix{
\mathit{CF}^{*-2m}(w+m+1) 
&& \ar[ll]_-{\includegraphics[valign=c]{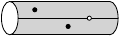}}^-{\displaystyle a_m} \mathit{CF}^*(w)
}
\end{equation}
such that
\begin{equation} \label{eq:iotam-equation}
\sum_{i+j = m} d_i a_j - a_i d_j = 0.
\end{equation}
The construction involves solutions of the usual kind of continuation map equations, having intersection number $m+1$ with $D$, and where one of the intersection points (drawn in white above) is singled out. The position of the distinguished point is constrained to lie on the line $\bR \times \{0\} = \{t = 0\} \subset \bR \times S^1$. In other words, the preimage of $D$ is written as a degree $m$ divisor $\Sigma$ plus one additional point $(s_*,0)$. The space of pairs $(\Sigma,(s_*,0))$ is then divided by translation, to form the parameter space. Equivalently, one can break translation invariance by putting the distinguished point at $(0,0)$, and then the parameter space is just $\mathit{Sym}_m(\bR \times S^1)$. 

\item There is the usual variant construction, where one does not divide by translation; because of the distinguished marked point, the parameter space is now $\mathit{Sym}_m(\bR \times S^1) \times \bR$. We use that to define maps
\begin{equation}
\label{eq:lambdam-map-dagger}
\xymatrix{
\mathit{CF}^{*-2m-1}(w+m+2) 
&& \ar[ll]_-{\includegraphics[valign=c]{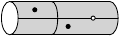}}^-{\displaystyle a_m^\dag} \mathit{CF}^*(w)
}
\end{equation}
such that
\begin{equation} \label{eq:iotam-equation-2}
\sum_{i+j=m} d_i a_j^\dag + d_i^\dag a_j + a_j^\dag d_j - a_j d_j^\dag = 0.
\end{equation}
\end{itemize}
For $x \in \mathit{CF}^*(w)$, one then sets
\begin{equation} \label{eq:define-a}
\begin{aligned}
& a_{C_q}(x) = \sum_m q^m a_m(x), \\
& a_{C_q}(\eta x) = \sum_m q^m (\eta a_m(x) + a_m^\dag(x)).
\end{aligned}
\end{equation}
The following is a version of \cite[Proposition 4.5 and (5.37)]{seidel18}:

\begin{proposition} \label{th:iota-and-differential}
There is a chain homotopy involving $a_{C_q}$ and the BV operator,
\begin{equation} \label{eq:differentiating-the-differential}
d_{C_q}^1 a_{C_q} - a_{C_q} d_{C_q}^1 \htp \partial_q d_{C_q}.
\end{equation}
\end{proposition}

\begin{corollary}
On the cohomology level, $d_{C_q}^1$ and $a_{C_q}$ commute after multiplying with $2q$.
\end{corollary}

\begin{proof}[Proof of Corollary]
Take the $q$-linear endomorphism $\Gamma_{C_q}$ of $C_q$ which multiplies $\mathit{CF}^k(w)$ by $k$, and similarly $\eta \mathit{CF}^k(w)$ by $(k-1)$. An elementary calculation, using only the degrees of various components of the differential, shows that
\begin{equation}
d_{C_q}\Gamma_{C_q} - \Gamma_{C_q} d_{C_q} =  2q\partial_q(d_{C_q}) - d_{C_q}.
\end{equation}
Hence, if we take a cocycle in $C_q$, its image under $2q\partial_q(d_{C_q})$ is always a coboundary. In view of Proposition \ref{th:iota-and-differential}, this implies the desired result.
\end{proof}

The homotopy from Proposition \ref{th:iota-and-differential} is a sum of three pieces,
\begin{equation} \label{eq:ab-sum}
a^{1,1}_{C_q} + a^{1,0}_{C_q} + b^1_{C_q}: C_q \longrightarrow C_q[-2],
\end{equation}
each of which is fundamentally similar to $a_{C_q}$ itself. The relevant parameter spaces are versions of those in \cite[Sections 5.2e--5.2f]{pomerleano-seidel23}, as follows.
\begin{itemize} \itemsep.5em
\item Consider cylinders with an angle-decorated circle, and where a distinguished point lies on the part of the line $\{t = 0\}$ to the right (larger $s$ values) of that circle. As usual, there is a variant without translation-invariance. Denote the resulting operations by 
\begin{equation} \label{eq:a11}
\xymatrix{
\mathit{CF}^{*-2-2m}(w+m+1) &&
\ar[ll]^-{\displaystyle a_m^{1,1}}_-{\includegraphics[valign=c]{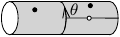}} \mathit{CF}^*(w)
}
\end{equation}
and
\begin{equation} \label{eq:a11-dag}
\xymatrix{
\mathit{CF}^{*-3-2m}(w+m+2) &&
\ar[ll]^-{\displaystyle a_m^{1,1,\dag}}_-{\includegraphics[valign=c]{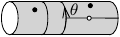}} \mathit{CF}^*(w)
}
\end{equation}
One combines them as in \eqref{eq:define-a} to get $a_{C_q}^{1,1}$.
\end{itemize}
Here, as the distinguished marked point goes to $+\infty$, it splits off a cylinder carrying that point anywhere on the line $\{t = 0\}$, which is how the relationship with $d_{C_q}^1 a_{C_q}$ is established.
\begin{itemize}
\item
We switch to the distinguished marked point lying to the left of the angle-decorated circle, but now constraining it to the line $\{t = \theta\}$ instead. Obligatory pictures:
\begin{equation} \label{eq:a10}
\xymatrix{
\mathit{CF}^{*-2-2m}(w+m+1) &&
\ar[ll]^-{\displaystyle a_m^{1,0}}_-{\includegraphics[valign=c]{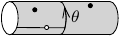}} \mathit{CF}^*(w)
}
\end{equation}
and
\begin{equation} \label{eq:a10-dag}
\xymatrix{
\mathit{CF}^{*-3-2m}(w+m+2) &&
\ar[ll]^-{\displaystyle a_m^{1,0,\dag}}_-{\includegraphics[valign=c]{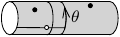}} \mathit{CF}^*(w)
}
\end{equation}
\end{itemize}
This time, as the distinguished marked point goes to $-\infty$, we split off a cylinder which (because of the way the data are setup on the compactification of the parameter space, in parallel with those for the BV operator) carries a distinguished point on the line $\{t = 0\}$; this is the origin of the term $a_{C_q} d_{C_q}^1$ in \eqref{eq:differentiating-the-differential}.
\begin{itemize}
\item
In our final moduli spaces, the angle is lifted to $\theta^{\mathit{lift}} \in [0,1]$. The distinguished marked point lies on the angle-decorated circle, and the coordinate $t_*$ which is the nontrivial part of its position also comes with a lift $t_*^{\mathit{lift}} \in [\theta^{\mathit{lift},1}]$. As usual, there are two versions:
\begin{align} \label{eq:bm11}
&
\xymatrix{
\mathit{CF}^{*-2-2m}(w+m+1) &&
\ar[ll]^-{\displaystyle b_m^{1}}_-{\includegraphics[valign=c]{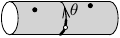}} \mathit{CF}^*(w)
}
\\ \label{eq:bm11-anchor}
& 
\xymatrix{
\mathit{CF}^{*-3-2m}(w+m+2) &&
\ar[ll]^-{\displaystyle b_m^{1,\dag}}_-{\includegraphics[valign=c]{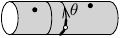}} \mathit{CF}^*(w)
}
\end{align}
\end{itemize}
These parameter spaces have three kinds of boundary strata of particular interest. When the distinguished marked point satisfies $t_*^{\mathit{lift}} = \theta^{\mathit{lift}}$, the geometry agrees with one of the boundary strata in the definition of $a_m^{1,0}$ or $a_m^{1,0,\dag}$; and for $t_*^{\mathit{lift}} = 1$, correspondingly with $a_m^{1,1}$ or $a_m^{1,1,\dag}$. In the sum \eqref{eq:ab-sum}, this will cause the contributions of those boundary strata to cancel. The third stratum is where $\theta^{\mathit{lift}} = 0$, in which case the $t$-coordinate of the distinguished marked point is unconstrained. We can then forget the circle, as well as the fact that the marked point was distinguished in the first place, the outcome being a cylinder with $(m+1)$ points; and we arrange that the continuation map data are compatible with that forgetting operation, which explains the appearance of $\partial_q d_{C_q}$ in \eqref{eq:differentiating-the-differential}.

The last-mentioned property is sufficiently important to deserve a little more discussion. The parameter space underlying \eqref{eq:bm11} is explicitly given by
\begin{equation}
\frakB_m^{1} = \big( \mathit{Sym}_m(\bR \times S^1) \times \bR \times \{ 0 \leq \theta^{\mathit{lift}} \leq t_*^{\mathit{lift}} \leq 1\} \big) / \bR.
\end{equation}
Write $\partial_{\theta^{\mathit{lift}}=0}\frakB_m^{1}$ for the boundary stratum where $\theta^{\mathit{lift}} = 0$. Via the forgetting process outlined above, one defines a map to the parameter space from Section \ref{subsec:dm}:
\begin{equation} \label{eq:marked-forget}
\begin{aligned}
& \partial_{\theta^{\mathit{lift}}=0}\frakB_m^{1} \longrightarrow \frakD_{m+1}, \\
& (\Sigma, s_*, \theta^{\mathit{lift}}=0, t_*^{\mathit{lift}}) \longmapsto \Sigma + (s_*,t_*);
\end{aligned}
\end{equation}
by $+$ we mean adding the distinguished marked point $(s_*,t_*)$ to the existing divisor. This is an $(m+1)$-fold cover over the open dense subset in $\frakD_{m+1}$ where the marked points are pairwise distinct and have nonzero $t$-coordinate. We require that the data setting up the continuation map equations should be compatible with \eqref{eq:marked-forget}. Assuming suitable genericity properties, the effect is to have $(m+1)d_{m+1}$ appearing in the equation arising from one-dimensional moduli spaces, which is one of the terms in $\partial_q d_{C_q}$. The same discussion applies to the parameter space $\frakB_m^{1,\dag}$ for \eqref{eq:bm11-anchor} and its map to $\frakD_{m+1}^\dag$. This concludes our outline of Proposition \ref{th:iota-and-differential}.

\subsection{The thimble revisited\label{subsec:switch-to-cycles}}
In a technical departure, we now replace Morse functions with pseudo-cycles, since those arise more naturally in Gromov-Witten theory. This unfortunately means we have to retread some of the previous trajectory. Pick a pseudo-cycle $c_P: P \rightarrow M$, assumed to be transverse to $D$; this means that the map $c_P$ is transverse to $D$, and additionally that the restriction to $c_P^{-1}(D)$ is a pseudo-cycle in $D$. Write $|P|$ for the codimension with respect to $M$. 
\begin{itemize} \itemsep.5em
\item In the construction from Section \ref{subsec:classical-thimble}, we now consider pairs $(u,p)$ where $p \in P$, with the adjacency condition
\begin{equation} \label{eq:p-incidence}
u(+\infty) = c_P(p). 
\end{equation}
This yields elements
\begin{equation} \label{eq:s-m-of-p}
\mathit{CF}^{|P|-2m}(m) \ni s_m(P)\; \includegraphics[valign=c]{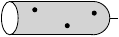} \hspace{.5em} P
\end{equation}
for which the analogue of \eqref{eq:d-s-morse} is simply 
\begin{equation} \label{eq:d-p-cycle}
\sum_{i+j=m} d_i s_j(P) = 0.
\end{equation}
By summing over all $m$, we therefore get a cocycle
\begin{equation}
s_{C_q}(P) = \sum_{m \geq 0} q^m s_m(P) \in C_q^{|P|}.
\end{equation}
On cohomology, under the usual map that relates pseudo-cycles and Morse cocycles, this is equivalent to \eqref{eq:total-s-map}; we will not discuss the proof of that fact, since it follows a standard pattern.

\item
Similarly, given a pseudo-cycle $R$ in $D$, one can adapt the construction from Section \ref{section:thimbleswithtangency} to pairs consisting of a map $u$ with $w$-fold tangency to $D$ at $+\infty$, and a point $r \in R$, satisfying the same incidence condition as before (but which now takes place in $D$):
\begin{equation} \label{eq:r-incidence}
u(+\infty) = c_R(r).
\end{equation}
This defines, for any $w>0$ and $m \geq 0$, elements
\begin{equation} \label{eq:twpseudoversion}
\mathit{CF}^{|R|-2m}(w+m) \ni t_{w,m}(R) \;
\includegraphics[valign=c]{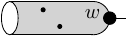} R
\end{equation}
(here, the codimension $|R|$ is with respect to $D$). They satisfy the counterpart of \eqref{eq:d-p-cycle}, and we can assemble them into cocycles
\begin{equation} 
t_{C_q,w}(R) \in C_q^{|R|}.
\end{equation}
\end{itemize}
In fact, we will only be using the $w = 1$ case of the latter construction.

\begin{lemma} \label{th:replace-t-by-s}
Take a pseudo-cycle $R$ in $D$, and a perturbation $P$ of that pseudo-cycle into $M$, which is transverse to $D$. Then 
\begin{equation} \label{eq:s-is-t}
q [t_{C_q,1}(R)] = [s_{C_q}(P)] \in
\mathit{SH}_q^{|P| = |R|+2}(M,D).
\end{equation}
\end{lemma}
%

For further discussion, see Section \ref{subsec:perturb}. Eventually, that Lemma will be used to remove $t_{C_q,1}$ from our formulae, up to $q$-torsion.

\begin{itemize}
\item We will also need a simple instance of Gromov-Witten invariants. Namely, for any $m \geq 0$ and any $P$ as before, define a pseudo-cycle $g_m(P)$, with $|g_m(P)| = |P|-2m$, by considering pseudo-holomorphic spheres with three distinguished points, which have intersection number $m+1$ with $D$. The first point goes through $D$; the second one through $P$ in the sense of \eqref{eq:p-incidence}; and we use evaluation at the third point to define our pseudo-cycle. Here's the picture for easier memorization (as usual, all intersection points with $D$ are drawn as dots, with the distinguished one in white):
\begin{equation}
g_m(P) \;\includegraphics[valign=c]{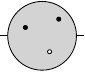}\; P
\end{equation}
\end{itemize}
In cohomology, summing over $m$ yields an instance of the small quantum product, minus its classical cup-product part:
\begin{equation} \label{eq:small-quantum-2}
\sum_m q^m [g_m(P)] = q^{-1}\big([D] \ast_q [P] - [D \cap P]\big) \in (H^*(M)[[q]])^{|P|}.
\end{equation}
On the level of actual cycles, the following, which is proved by standard methods, will be useful (see Section \ref{subsec:gw}):

\begin{lemma} \label{th:again-transverse}
For a generic choice of almost complex structure which respects $D$, the pseudo-cycle $g_m(P)$ is again transverse to $D$.
\end{lemma}

\subsection{The cap product on the thimble\label{subsec:thimble-cap}}
With those preliminaries in mind, we continue the main thread of our discussion.
\begin{itemize}
\item Take the thimble with one distinguished point, whose position must satisfy $t = 0$ (and, as usual, $m$ other points in arbitrary position). We impose an incidence condition with the pseudo-cycle $P$ at $+\infty$, and with $D$ at the $(m+1)$ marked points, including the distinguished one. For each $m \geq 0$, this gives rise to chains
\begin{equation} \label{eq:alpha-elements}
\mathit{CF}^{|P|-2m-1}(m) \ni y_m(P)\;
\includegraphics[valign=c]{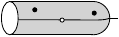}\hspace{-1em} \hspace{1em} P
\end{equation}
which one combines as usual into $y_{C_q}(P) \in C_q^{|P|+1}$.
\end{itemize}

\begin{lemma} \label{th:alpha-property}
The chains \eqref{eq:alpha-elements} satisfy
\begin{equation} \label{eq:alpha-property}
\sum_{i+j = m} d_i y_j(P) +  a_i(s_j(P))= 
t_{1,m}(D \cap P) + \sum_{i+j = m}
s_i(g_j(P))
\end{equation}
\end{lemma}

On the left hand side of \eqref{eq:alpha-property}, a cylinder splits off at $-\infty$; and in the second term, the distinguished point moves into that cylinder. On the right hand side, the first term corresponds to the distinguished point reaching $+\infty$; in the second term, one then additionally has sphere bubbling happening at that point (technically, that term makes sense thanks to Lemma \ref{th:again-transverse}). We postpone further discussion to Section \ref{subsec:alpha}. As a consequence, one gets the following version of the classical relation \cite{piunikhin-salamon-schwarz95} between the cap product on Hamiltonian Floer cohomology and the quantum product on ordinary cohomology.

\begin{corollary} \label{th:pss-iota}
There is a commutative diagram involving \eqref{eq:iota-endomorphism} and 
(the $q$-linear extension of) \eqref{eq:total-s-map}:
\begin{equation}
\xymatrix{
H^*(M)[[q]] \ar[rr]^{s_{C_q}} \ar[d]_-{[D] \ast_q}
&& \mathit{SH}_q^*(M,D) \ar[d]^-{q\,a_{C_q}} 
\\
H^{*+2}(M)[[q]] \ar[rr]^-{s_{C_q}}
&& \mathit{SH}_q^{*+2}(M,D). 
}
\end{equation}
\end{corollary}

\begin{proof}[Proof of the Corollary]
The relations \eqref{eq:alpha-property} yield the following cohomology level equation:
\begin{equation} 
[t_{C_q,1}(D \cap P)] + s_{C_q}( q^{-1}([D] \ast_q [P] - [D \cap P])) - [a_{C_q}(s_{C_q}(P))] = 0.
\end{equation}
Multiplying by $q$ leads to cancellation between the first and third term, due to Lemma \ref{th:replace-t-by-s}.
\end{proof}

\subsection{Proof sketch for Lemma \ref{th:replace-t-by-s}\label{subsec:perturb}}
Consider a point $(\Sigma,u,r)$ in a moduli space underlying $t_{1,m}(R)$, with the added condition that the divisor $\Sigma = \{z_1,\dots,z_m\}$ is a collection of $m$ pairwise distinct points, none of them equal to $+\infty$ (for a generic choice made in the construction, this will be true in moduli spaces of dimension $\leq 1$). 
Regularity of $(\Sigma,u,r)$ means surjectivity of an operator
\begin{equation} \label{eq:r-extended-linearized}
\xymatrix{
\scrE_R = W^{2,2}(u^*TM) \oplus T_{z_1}T \oplus \cdots \oplus T_{z_m}T \oplus T_rR
\ar[d] \\
\scrF_R = W^{1,2}(u^*TM) \oplus \nu_{u(z_1)}D \oplus \cdots \oplus \nu_{u(z_m)}D \oplus T_{u(+\infty)}M
}
\end{equation}
The notation here largely follows \eqref{eq:n-extended-linearized}. On the target space, the the $T_{u(+\infty)}M$-component combines two roles; its part normal to $D$ measures the failure of $u(+\infty) \in D$, and its part tangent to $D$ the failure of \eqref{eq:r-incidence} as an incidence condition inside $D$. 

Take the family of continuation map data over $\frakT_{1,m} = \mathit{Sym}_m(T)$ underlying $t_{1,m}(R)$, and choose data on $\mathit{Sym}_{m+1}(T)$ which are consistent with the embedding $\frakT_{1,m} \rightarrow \mathit{Sym}_{m+1}(T)$ which adds a point at $+\infty$. Let's use these to define a degenerate version of $s_{m+1}(P)$, for $P = R$ (degenerate because the pseudo-cycle fails to satisfy our usual requirement of transversality to $D$). Take a point $(\Sigma, u, r)$ in the resulting moduli space, and assume that $\Sigma = \{z_1,\dots,z_m, z_{m+1} = +\infty\}$ for pairwise disjoint points (it is automatically true that the divisors which occur in the moduli space must contain $+\infty$). This time, regularity means invertibility of 
\begin{equation} \label{eq:p-extended-linearized}
\xymatrix{
\scrE_P = W^{2,2}(u^*TM) \oplus T_{z_1}T \oplus \cdots \oplus T_{z_m}T \oplus T_{z_{m+1}}T \oplus T_rR 
\ar[d] \\
\scrF_P = W^{1,2}(u^*TM) \oplus \nu_{u(z_1)}D \oplus \cdots \oplus \nu_{u(z_m)}D \oplus \nu_{u(z_{m+1})}D \oplus T_{u(+\infty)}M.
}
\end{equation}
In comparison with \eqref{eq:r-extended-linearized}, the domain has been enlarged by $T_{z_{m+1}}T$, corresponding to the freedom of moving the new marked point $z_{m+1}$ away from $+\infty$. Correspondingly, in the target space, there is a new copy of the normal bundle expressing the constraint $u(z_{m+1}) \in D$, which is independent of \eqref{eq:p-incidence}. The relation between the two operators is expressed by a commutative diagram
\begin{equation}
\xymatrix{
0 \ar[r] & 
\scrE_R \ar[r] \ar[d]_-{\eqref{eq:p-extended-linearized}} &
\scrE_P \ar[r] \ar[d]_-{\eqref{eq:r-extended-linearized}} &
T_{+\infty}T \ar[d]^-{\eqref{eq:small-map}} \ar[r] & 0
\\
0 \ar[r] & 
\scrF_R \ar[r]_-{\eqref{eq:graph-inclusion}} &
\scrF_P \ar[r]_-{\eqref{eq:graph-projection}} &
\nu_{u(+\infty)}D \ar[r] & 
0
}
\end{equation}
The top row consists of the obvious inclusions and projection. We also have:
\begin{align}
\label{eq:graph-inclusion}
& 
\parbox{34em}{The injective map $\scrF_R \rightarrow \scrF_P$ where the extra component $\nu_{u(z_{m+1})}D$ is obtained by projection from $T_{u(+\infty)}M$.}
\\
\label{eq:graph-projection}
& 
\parbox{34em}{The map $\scrF_P \rightarrow \nu_{u(+\infty)}D$ that takes the $\nu_{u(z_{m+1})}D$ component of $\scrF_P$ and subtracts from it the projection of the $T_{u(+\infty)}M$ component.}
\\
\label{eq:small-map}
&
\parbox{34em}{The derivative of $u$ at $+\infty$, projected to the normal bundle.}
\end{align}
Under our assumptions, where $u$ intersects $D$ with multiplicity $1$ at $+\infty$, the map \eqref{eq:small-map} is an isomorphism (we have seen this argument before, in the proof of Lemma \ref{th:boundary-regularity}). Therefore: 

\begin{lemma} \label{th:operator-comparison}
The operator \eqref{eq:p-extended-linearized} is onto if and only if \eqref{eq:r-extended-linearized} is. Moreover, in that case, the kernels of the two operators are the same (under the inclusion $\scrE_R \subset \scrE_P$).
\end{lemma}

The upshot is that one can in fact define $s_{m+1}(P)$ for $P = R$ in this way, using the transversality theory for $t_{1,m}(R)$ and Lemma \ref{th:operator-comparison} to establish the necessary regularity results; and that the outcome agrees with $t_{1,m}(R)$.

After that, Lemma \ref{th:replace-t-by-s} is proved simply by slightly perturbing $R$, which because of regularity does not change the count of points in zero-dimensional moduli spaces. There is an additional wrinkle here: the desired result requires considering all $m$ simultaneously, but the naive argument involving a small perturbation can only work for finitely many $m$ at a time. The proper solution is to replace the perturbation by a homotopy of pseudo-cycles, and to introduce another associated moduli space. However, there's a simpler workaround, which will do for us: one can arrange for the Floer complex to be bounded below, by Lemma \ref{th:satisfy-epsilon-bound}(ii), and then only finitely many $m$ can contribute anyway, because of the power $q^{m+1}$ involved.

\begin{remark}
It seems that the argument above could be used to define $s_{C_q}(P)$ for any pseudo-cycle $P$ in $M$, irrespectively of how it intersects $D$, thus rendering Lemma \ref{th:again-transverse} unnecessary. (Presumably, this approach could be also applied in Section \ref{subsec:classical-thimble}, making the use of an arbitrary Morse function possible, which means \eqref{eq:stable-transverse} could be dropped.) We have not explored all the details, since there seems to be no major gain to balance the increased technical complexity. 
\end{remark}

\subsection{Proof sketch for Lemma \ref{th:again-transverse}\label{subsec:gw}} Let $S = (\bR \times S^1) \cup \{\pm \infty\}$ be the Riemann sphere. Take an almost complex structure $J$ which respects $D$. For $m \geq 0$, define $\frakG_m(P)$ to be the space of pairs $(v,p)$, where $p \in P$ and $v: S \rightarrow M$ is a $J$-holomorphic map with the following properties. It has intersection number $(m+1)$ with $D$; is not multiply-covered; is not contained in $D$; and satisfies $v(0,0) \in D$, as well as $v(+\infty) = c_P(p)$. For generic $J$, the space $\frakG_m(P)$ is smooth of dimension $2(n+m) -|P|$, and its evaluation map at $-\infty$ is transverse to $D$. Look at the Gromov compactification, by stable maps with three marked points $(z_{\pm} = \pm\infty,\,z_* = (0,0))$. For a point of $\overline{\frakG}_m(P) \setminus \frakG_m(P)$, consider the following possible phenomena:
\begin{enumerate}[label=(G\arabic*)] \itemsep.5em
\item
Suppose that 
the images of $z_{\pm}$ lie on the same simple $J$-holomorphic sphere, which is contained in $D$. This gives codimension $2m+4 \geq 4$.
the codimension is 2m+4.

\item \label{item:g2}
Similarly, suppose that 
the images of $z_{\pm}$ lie on the same simple $J$-holomorphic sphere, which is not contained in $D$ but has Chern number $\leq m$. This gives codimension $\geq 2$, and codimension $\geq 4$ if the image of $\zeta_-$ is contained in $D$.

\item \label{eq:infinities-collide}
Suppose that the images of $z_{\pm}$ agree (for instance, because they lie on the same constant component of the limiting stable curve). That image point lies on our pseudo-cycle or its limiting set, and also on a simple $J$-holomorphic sphere of Chern number $\leq m+1$. This gives codimension $\geq 2$; and if the point additionally lies in $D$, codimension $\geq 4$ (here, we are using the fact that $P$ is transverse to $D$).

\item
If none of the previous cases applies, there must be a nonempty chain of $K \geq 2$ simple pairwise distinct $J$-holomorphic spheres $(v^1,\dots,v^K)$, as in \eqref{eq:chain-incidence}, with $\sum_k v^k \cdot D \leq m+1$, such that $v^1(-\infty)$ agrees with the image of $z_-$, and $v^K(+\infty)$ with the image of $z_+$. This leads to the same conclusion as in \ref{item:g2}.
\end{enumerate}
By inspection, one sees that the evaluation at $-\infty$ on $\frakG_m(P)$ indeed satisfies Lemma \ref{th:again-transverse}. A further transversality argument shows that:

\begin{lemma} \label{th:distinct-sphere-points}
Suppose that we additionally assume that $v^{-1}(D)$ consists of $(m+1)$ distinct points. This still gives the same pseudo-cycle (in the sense that the subset we have removed has codimension $\geq 2$).
\end{lemma}


\subsection{Proof sketch for Lemma \ref{th:alpha-property}\label{subsec:alpha}}
The operation \eqref{eq:alpha-elements} uses a parameter space
\begin{equation} \label{eq:x-space}
\frakY_m = \mathit{Sym}_m(T) \times (-\infty,+\infty],
\end{equation}
where the last factor, which we denote by $s_*$, determines the position of the distinguished marked point. As usual, one chooses data underlying a Cauchy-Riemann equation on $T$ which vary over $\frakY_m$. We impose the following assumptions (which don't interfere with transversality or contradict each other, hence can be satisfied at the same time).
\begin{enumerate}[label=(Y\arabic*)] 
\itemsep.5em
\item Take the subset of $\frakY_m$ where the distinguished marked point lies at $+\infty$. On that subset, we want the data to agree with those used to define $t_{1,m}(P \cap D)$.

\item \label{item:sub-symmetric-product}
For any $i < m$, Consider the subspace $\mathit{Sym}_i(T) \hookrightarrow \frakY_m$ obtained by adding $m-i$ times the point $+\infty$ to the divisor, and also setting $s_* = +\infty$. On that subset, we want to impose two conditions. First of all, for the relevant families of almost complex structures, we want that at $+\infty$ to be independent of $i$ and of where we are in $\mathit{Sym}_i(T)$. Moreover, this almost complex structure, denoted simply by $J$, will be suitable for defining the Gromov-Witten pseudo-cycle $g_{m-i}(P)$. Having done that, the data on the entire thimble should be that used to define $s_i(g_{m-i}(P))$. 
\end{enumerate}
Let $\frakY_m(x_-,P)$ be the associated moduli space of parametrized solutions of the Cauchy-Riemann equation, with the usual incidence condition \eqref{eq:p-incidence}. As usual, the main stratum consists of those $(\Sigma,u,s_*,p)$ where the points of $\Sigma$ are distinct, and not equal to either $\infty$ or to the distinguished point $(s_*,0)$, for $s_*<\infty$. We are concerned with what happens as $s_* \rightarrow +\infty$. The most straightforward behaviour is that the limit remains within $\frakY_m(x_-,P)$, with no bubbling or splitting. Of course, in that case the maps satisfy $u(+\infty) \in D \cap P$. One in fact recovers the space underlying $t_{1,m}(D \cap P)$, which explains the appearance of that term in \eqref{eq:alpha-property}.

There is another type of limiting configurations which appears in \eqref{eq:alpha-property}, for geometrically slightly less obvious reasons. For that, we consider a sequence $(\Sigma_k,u_k,s_{*,k},p_k)$ in the main stratum of the moduli space, such that:
\begin{enumerate}[label=(Y\arabic*)] 
\setcounter{enumi}{2} \itemsep.5em
\item $p_k \rightarrow p \in P$.

\item The $\Sigma_k$ converge to the union of: a collection $\Sigma$ of $i<m$ distinct points in $T \setminus \{+\infty\}$; and the point $+\infty$, with multiplicity $(m-i)$.

\item $s_{*,k} \rightarrow +\infty$. Moreover, if we look at the sequence of translations of the cylinder by $-s_{*,k}$, then in the limit, we get a configuration $\Pi$ consisting of the distinguished point $(0,0)$ together with $(m-i)$ distinct points in $(\bR \times S^1) \setminus \{(0,0)\}$. (Those are the limits of the points in $\Sigma_k$ which originally converged to $+\infty$, so we are making a finer assumption on how that convergence works).

\item The sequence $u_k$ converges to a limit $u$ together with a bubble $v: S \rightarrow M$. Here, $v$ is a simple pseudo-holomorphic map and satisfies
\begin{equation} \label{eq:uv-conditions}
v^{-1}(D) = \Pi + \{(0,0)\},\; v(+\infty) = c_P(p),\; v(-\infty) = u(+\infty).
\end{equation}
\end{enumerate}
In the parameter space \eqref{eq:x-space}, the limit point lies in the stratum $\mathit{Sym}_i(T) \times \{+\infty\}$, of codimension $2i+1 \geq 3$, since that limit does not take $\Pi$ into account. Nevertheless, the full limiting data $(\Sigma,u,\Pi,v,p)$ are easily seen to have codimension $1$. Moreover, in the case where the main stratum is $1$-dimensional, a gluing argument shows that such limits represent boundary points of a suitable compactification. Above, we have formulated the sphere part of the limiting space as a parametrized moduli space of $(\Pi,v,p)$ satisfying the first two conditions in \eqref{eq:uv-conditions}. Nevertheless, see Lemma \ref{th:distinct-sphere-points}, evaluation $v(-\infty)$ reproduces the Gromov-Witten pseudocycle $g_{m-i}(P)$ (more precisely, it reproduces it up to pieces that are of codimension $\geq 2$, hence irrelevant).

There are two other obvious codimension $1$ degenerations, which correspond to the first and last term in \eqref{eq:alpha-property}. They correspond to the stratum of $\bar\frakY_m(x_-,P)$ where a cylinder has split off on the left; that cylinder may contain the distinguished point, giving rise to the last term, or it may not, giving rise to the first term. The end of the argument consists in checking that all other phenomena that appear in the moduli space and its compactification have codimension $\geq 2$. We will not discuss that in detail, as it closely parallels such issues encountered earlier on.

\section{The connection\label{sec:connection}}
This section proves the last of the results stated at the beginning of the paper, Proposition \ref{th:connection}. By design, both the definition of the objects involved, and their use, are extensions of material in Section \ref{sec:operations}. 

\subsection{Definition\label{subsec:define-connection}}
The connection on $S^1$-equivariant deformed symplectic cohomology was defined in \cite[Section 5.3f]{pomerleano-seidel23}. We summarize the definition here, with the minor adjustments demanded by our adoption of the telescope construction for symplectic cohomology. The formal structure is that one constructs a $(u,q)$-linear map
\begin{equation} \label{eq:r-map}
\begin{aligned}
& \alpha_{C_{u,q}}: C_{u,q} \longrightarrow C_{u,q}, \\
& d_{C_{u,q}}\alpha_{C_{u,q}} - \alpha_{C_{u,q}} d_{C_{u,q}} = u\partial_q d_{C_{u,q}}.
\end{aligned}
\end{equation}
Here $\partial_q d_{C_{u,q}}$ is the derivative of the equivariant boundary operator with respect to the standard basis of $C_{u,q}$ (for any pair of one-periodic orbits, one takes the $q$-derivative of the relevant coefficient in $d_{C_{u,q}}$). The equation \eqref{eq:r-map} is equivalent to saying that
\begin{equation}
\nabla_{u\partial_q} = u\partial_q + \alpha_{C_{u,q}}
\end{equation}
is an endomorphism of the chain complex; this endomorphism underlies the connection on cohomology \cite[Definition 5.3.10]{pomerleano-seidel23}. Digging a bit into the details:
\begin{itemize} \itemsep.5em
\item
for each $m,w \geq 0$, $l \geq 0$, and $0 \leq i \leq l$, one defines operations
\begin{equation} \label{eq:equivariant-a}
\xymatrix{
\mathit{CF}^{*-2l-2m}(w+m+1) &&
\ar[ll]^-{\displaystyle a_m^{l,i}}_-{\includegraphics[valign=c]{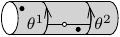}} \mathit{CF}^*(w)
}
\end{equation}
For $m = 0$ this reduces to \eqref{eq:iota-m-map}, and for $m = 1$ to \eqref{eq:a11}, \eqref{eq:a10}. The general case is defined by having $l$ angle-decorated circles, as in the definition of the equivariant differential; and requiring the distinguished marked point to satisfy
\begin{equation} \label{eq:s-interval}
(s_*,t_*) \in \begin{cases} (-\infty,\sigma^1] \times \{\theta^1 + \cdots + \theta^l\} & i = 0, \\
[\sigma^i,\sigma^{i+1}] \times \{\theta^{i+1} + \cdots + \theta^l\} & 1 \leq i < l, \\
[\sigma^l,\infty) \times \{0\} & i = l.
\end{cases}
\end{equation}

\item
We also have, for $1 \leq i \leq l$, operations
\begin{equation} \label{eq:equivariant-b}
\xymatrix{
\mathit{CF}^{*-2l-2m}(w+m+1) &&
\ar[ll]^-{\displaystyle b_m^{l,i}}_-{\includegraphics[valign=c]{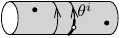}} \mathit{CF}^*(w)
}
\end{equation}
which reduce to \eqref{eq:bm11} for $l = 1$. As in that special case, the definition involves lifting the angle $\theta^i$ to $\theta^{i,\mathit{lift}}$, and then asking that
\begin{equation} \label{eq:b-position}
(s_*,t_*) \in \{\sigma^i\} \times [\theta^{i,\mathit{lift}},1].
\end{equation}
\end{itemize}
Let $\alpha_m^l$ be the sum of $a_m^{l,i}$ and $b_m^{l,i}$ over all $i$ (for $l = 0$, this means we set $r_m^l = a_m$). It satifies
\begin{equation} \label{eq:derivative-of-d}
\sum_{\substack{i+j=m \\ u+v = l}} d_i^u \alpha_j^v - \alpha_i^u d_j^v = 
\begin{cases} 
(m+1) d_{m+1}^{l-1} & l > 0, \\
0 & l = 0.
\end{cases}
\end{equation}
The mechanism which produces the right hand side is the same as in the $l = 1$ case, involving the stratum were $\theta^{i,\mathit{lift}} = 0$ and a forgetful map from that to the parameter space underlying $d_{m+1}^{l-1}$. As a consequence, on the subcomplex $(\bigoplus_w \mathit{CF}^*(w))[[u,q]] \subset C_{u,q}$, one can define 
\begin{equation}
\alpha_{C_{u,q}}(x) = \sum_{m,l} u^l q^m \alpha_m^l,
\end{equation}
and this will satisfy \eqref{eq:r-map}. Of course, this is not the entire picture: there are variants $a_m^{l,i,\dag}$ and $b_m^{l,i,\dag}$ defined according to the usual principle, and which complete the construction of the connection on the entire complex $C_{u,q}$. 

\subsection{Differentiating the thimble maps}
Following the same strategy as in Section \ref{subsec:thimble-cap}, we introduce analogues of the previous moduli space for the thimble surface, with an incidence constraint at $+\infty$ to a pseudo-cycle $P$ transverse to $D$.
\begin{itemize} \itemsep.5em
\item
For each $m,l \geq 0$ and $0 \leq i \leq l$, we define operations
\begin{equation} \label{eq:equivariant-alpha}
\mathit{CF}^{|P|-2m-2l-1}(m+1) \ni y_m^{l,i}(P) \;
\includegraphics[valign=c]{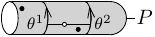}
\end{equation}
For $m = 0$ this reduces to \eqref{eq:alpha-elements}. The general case has $l$ angle-decorated circles, which are allowed to shrink to the point $+\infty$. The distinguished marked point is required to be in position \eqref{eq:s-interval}, with the only modification that for $i = l$ it can also reach $+\infty$.

\item
Also for $m,w \geq 0$, $l \geq 1$, and $1 \leq i \leq l$, we have
\begin{equation} \label{eq:equivariant-beta}
\mathit{CF}^{|P|-2m-2l-1}(m+1) \ni z_m^{l,i}(P) \;
\includegraphics[valign=c]{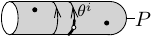}
\end{equation}
This mimics \eqref{eq:equivariant-b} on the thimble, which means that the distinguished marked point must satisfy \eqref{eq:b-position}.
\end{itemize}
Let $\xi_m^l(P)$ be the sum of all $y_m^{l,i}(P)$ and $z_m^{l,i}(P)$. This satisfies
\begin{equation} \label{eq:derivative-of-thimble}
\sum_{\substack{i+j=m \\ u+v = l}} d_i^u \xi_j^v + \alpha_i^u s_j^v(P)
= \sum_{i+j=m} s_i^l(g_j(P)) + t_{1,m}^l(D \cap P)
-
\begin{cases} 
(m+1) s_{m+1}^{l-1}(P) & l>0, \\
0 & l = 0.
\end{cases}
\end{equation}
On the left hand side of \eqref{eq:derivative-of-thimble}, a cylinder split off at $-\infty$; the distinguished marked point can either remain on the thimble (first term) or move into the cylinder (second term). So far, this has been entirely analogous to \eqref{eq:derivative-of-d}, and that also applies to the last term on the right hand side. The remaining terms of \eqref{eq:derivative-of-thimble} (the first two on the right) express the same kind of behaviour as in \eqref{eq:alpha-property}, which is indeed the $l = 1$ special case of \eqref{eq:derivative-of-thimble}. Both times, all $l$ angle-decorated circles stay in the thimble. One could think that there are other limits, in which some angle-decorated circles go to $+\infty$. However, in that case the angle becomes irrelevant (meaning that one makes the choices of data compatible with a map to a lower-dimensional space, which forgets the angle), which means that these phenomena are of codimension $\geq 2$. 

For the overall expression 
\begin{equation}
\xi_{C_{u,q}}(P) = \sum_{m,l} u^l q^m \xi_m^l(P), 
\end{equation}
the equation \eqref{eq:derivative-of-thimble} means that
\begin{equation}
d_{C_{u,q}}\xi_{C_{u,q}}(P) + \alpha_{C_{u,q}} s_{C_{u,q}}(P)
= s_{C_{u,q}}\big(\sum_m q^m g_m(P)\big) + t_{C_{u,q},1}(D \cap P) -
u \partial_q s_{C_{u,q}}(P).
\end{equation}
On the cohomology level, in view of \eqref{eq:small-quantum-2}, we get
\begin{equation} 
\label{eq:equation-for-connection}
u\nabla_q s_{C_{u,q}}(P) = s_{C_{u,q}}(q^{-1}([D] \ast_q [P] - [D \cap P])) + t_{1,C_{u,q}}(D \cap P).
\end{equation}
The equivariant analogue of Lemma \ref{th:replace-t-by-s}, which follows exactly the same strategy with added equivariant parameters, is:

\begin{lemma} \label{th:equivariant-replace-t-by-s}
Take a pseudo-cycle $R$ in $D$, and a perturbation $P$ of that pseudo-cycle into $M$, which is transverse to $D$. Then 
\begin{equation} \label{eq:s-is-t-2}
q [t_{C_{u,q},1}(R)] = [s_{C_{u,q}}(P)] \in
\mathit{SH}_{u,q}^{|P| = |R|+2}(M,D).
\end{equation}
\end{lemma}

After multiplying \eqref{eq:equation-for-connection} by $q$ and applying \eqref{eq:s-is-t-2}, one gets Proposition \ref{th:connection}.

\section{Pulling out the marked points\label{sec:pullout}}

This final section explains how to go from the framweork used in this paper (Riemann surfaces with added marked points, whose images go through the divisor $D$) to that in \cite{pomerleano-seidel23} and other papers involving more abstract deformations of symplectic cohomology (Riemann surfaces with added punctures, where one inserts a Maurer-Cartan element in the symplectic cochain $L_\infty$-algebra at the punctures). In our application, the statement is made simpler by the fact that the Maurer-Cartan element has only one nonzero term (the $q$-linear one), for degree reasons. The equivalence between the two approaches is intuitively plausible; and indeed, the basic strategy is the obvious one of stretching the surfaces near the marked points. Nevertheless, the moduli spaces that occur in the construction may have wider applications. For that reason, our discussion focuses on setting up those spaces, while subsequent steps receive a much shorter shrift.

\subsection{Maurer-Cartan spaces\label{subsec:woodward}}
These spaces are a variant of the complexified multiplihedra from \cite[Section 2.3]{woodward15}. For $m \geq 1$, consider pairs consisting of an ordered collection of not necessarily distinct points $(z_1,\dots,z_m)$ in $\bC$ and a constant one-form $\alpha = a \, \mathit{dz}$, with $a > 0$. Two such pairs are identified if they are related by translation and real rescaling, so the parameter space is
\begin{equation} \label{eq:ordered-configuration}
\MC_m = (\bC^m \times \bR^{>0}) / (\bC \rtimes \bR^{>0}) \iso 
\bC^m/\bC \iso \bC^{m-1}.
\end{equation}
The strata in the compactification $\overline{\MC}_m$ are labeled by (isomorphism classes of) trees $T$ with $(m+1)$ semi-infinite edges, of which one is singled out, and the others labeled by $\{1,\dots,m\}$. Our convention is that the distinguished semi-infinite edge is an output, oriented towards infinity, and all other edges are oriented towards that output; in particular, the other $m$ semi-infinite edges are inputs, oriented away from infinity. The main condition is that the edges pointing towards any given vertex should either all be semi-infinite (in which case the vertex is called a leaf), or all finite. The unique vertex adjacent to the output is called the root; it can be a leaf only in the degenerate case of a single-vertex tree, which corresponds to the interior of our moduli space. All non-leaf vertices must have valence $|v| \geq 3$, while for leaves the condition is $|v| \geq 2$. The corresponding stratum is 
\begin{equation} \label{eq:m-t-space}
\MC_T \iso \prod_{v \text{ leaf}} \MC_{|v|-1} \times \prod_{v \text{ not a leaf}} \FM_{|v|-1},
\end{equation}
where 
\begin{equation} \label{eq:fm}
\FM_m = \mathit{Conf}_{m}^{\mathrm{ord}}(\bC)/(\bC \rtimes \bR^{>0}),
\end{equation}
with $\mathit{Conf}^{\mathrm{ord}}$ for ordered configuration space. In the occurrence of these factors in \eqref{eq:m-t-space}, it is better to say that the points of each configuration are labeled by the incoming edges at that vertex. Intuitively, it is useful to think of the $\MC_{|v|-1}$ factors as being at finite scale (let's say, normalized to $a = 1$), and the $\mathit{Conf}_{|v|-1}^{\mathrm{ord}}$ at infinite scale (obtained as limits where one shrinks a sequence of finite scale configurations more and more). The codimension of \eqref{eq:m-t-space} is the number of non-leaf vertices. In particular, the codimension one strata are
\begin{equation} \label{eq:mc-codim-1}
\MC_{m_1} \times \cdots \times \MC_{m_r} \times \FM_r
\end{equation}
for any partition of $\{1,\dots,m\}$ into $r \geq 2$ subsets of sizes $m_1,\dots,m_r \geq 1$. This notation inflates the number of such strata somewhat (different orderings of the partition $(m_1,\dots,m_r)$ give rise to isomorphic trees, hence describe the same stratum).

\begin{remark}
The construction above is closely related to the well-known Fulton-MacPherson compactification $\overline{\FM}_m$, $m \geq 2$, of \eqref{eq:fm}. Indeed, one can write
\begin{equation} \label{eq:mc-fm}
\overline{\MC}_m = 
\MC_m \sqcup \Big( \coprod
\MC_{m_1} \times \cdots \times \MC_{m_l} \times \overline{\FM}_l \Big) / S_l,
\end{equation}
where the disjoint union is over all partitions of $\{1,\dots,m\}$ into $l \geq 2$ subsets of sizes $(m_1,\dots,m_l)$; and $S_l$ is the symmetric group, acting freely on such partitions (so it exchanges different strata, making up for the fact that the partitions are not naturally ordered). In terms of \eqref{eq:m-t-space}, what we have done in \eqref{eq:mc-fm} is simply to collect all the non-leaf vertices into one expression.
\end{remark}

\begin{remark}
Let's clarify the relation with complexified multiplihedra. We allow the points in \eqref{eq:ordered-configuration} to coincide, whereas in \cite{woodward15} they would split off into a separate limiting configuration (at scale zero, so to speak). We illustrate this in Figure \ref{fig:multiplihedra} by looking at part of the real locus of $\overline{\MC}_3$, comparing it with the multiplihedron from \cite[Figure 8]{mau-woodward08}. A second departure from \cite{woodward15} is that in the non-leaf factors of \eqref{eq:m-t-space}, we do not quotient out by rotations. This leads to the appearance of Fulton-MacPherson spaces rather than Deligne-Mumford spaces in \eqref{eq:mc-fm} (and makes sense because our intended application is to symplectic cohomology, rather than Gromov-Witten theory).
\end{remark}
\begin{figure}
\begin{centering}
\includegraphics{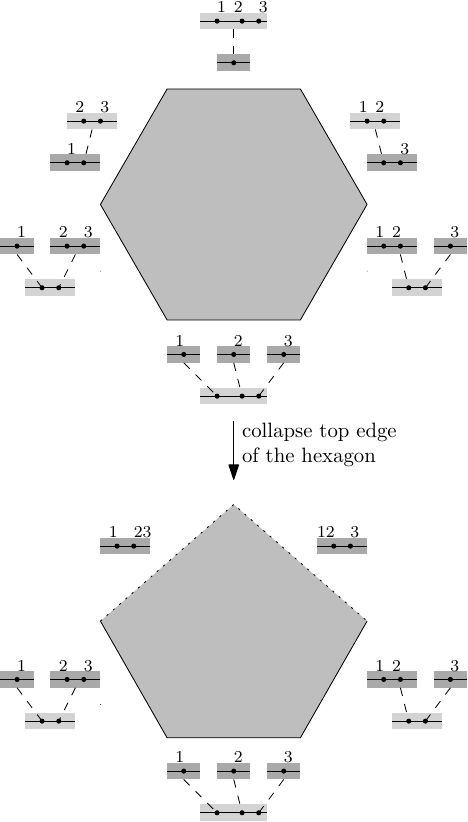}
\caption{\label{fig:multiplihedra}The multiplihedron $\bar{M}_{1,3}$ (in the notation from \cite{mau-woodward08}, at the top) compared to the corresponding part of $\overline{\MC}_3(\bR)$ (bottom). The darker components are those where one does not divide by rescaling. The dotted edges of the pentagon belong to the interior $\MC_3$ (and the corner where they meet is simply where the three numbered points coincide).}
\end{centering}
\end{figure}

Consider a point in a stratum \eqref{eq:m-t-space}. In each factor, choose a representing configuration, and in the case of the leaves, let that representative have $a = 1$ (meaning $\alpha = \mathit{dz}$). When gluing the pieces together, one has a parameter $\lambda_e > 0$ for each finite edge $e$; the geometric datum at the source vertex of $e$ is rescaled by $\lambda_e$ before being inserted into the corresponding datum at the target vertex. However, these gluing parameters are not independent of each other, because we will require that the one-forms inherited from the leaves should agree after gluing:
\begin{equation} \label{eq:equal-gluing} \parbox{34em}{
take a vertex of the tree. The product of gluing parameters along the path from a leaf to that vertex must be the same for all leaves.}
\end{equation}
The standard way of encoding such constraints goes as follows. Take the abelian group $G_T = \bZ^{\mathit{Ed}_{\mathit{fin}}(T)}$ generated by finite edges, and let $G_{T,\geq 0} \subset G_T$ be the obvious nonnegative (unital) monoid. Any path in $T$ connecting two leaves gives an element of $G_T$ (by counting the edges where the path direction is compatible with the orientation of the tree as $+1$, and the others with $-1$). Those elements generate a subgroup $R_T \subset G_T$. Write $Q_T = G_T/R_T$, and let $Q_{T,\geq 0} \subset Q_T$ be the image of $G_{T,\geq 0}$. The space of permitted gluing parameters, including the degenerate situation where only some components are glued, can be written as 
\begin{equation} \label{eq:gluing-corner}
\mathit{Hom}(Q_{T,\geq 0},(\bR^{ \geq 0},\cdot)). 
\end{equation}
One can reformulate the situation as follows. For each vertex $v$, take a path going from some leaf to $v$, and the corresponding element in $Q_T$. These elements are zero if $v$ is a leaf, and the others give a basis of $Q_T$; so we get
\begin{equation} \label{eq:non-leaf}
Q_T \iso \bZ^{\mathit{Ve}(T)}/\bZ^{\mathit{Ve}_{\mathit{leaf}}(T)} \iso
\bZ^{\mathit{Ve}_{\mathit{non-leaf}}(T)}.
\end{equation}
(The notation in \eqref{eq:non-leaf} is for the sets of all vertices, leaf vertices, and non-leaf vertices, respectively.) In these terms, an edge $e$ is given by the difference between its endpoint and starting point, so $Q_{T,\geq 0}$ is the submonoid of \eqref{eq:non-leaf} generated by such differences.

\begin{example} \label{th:mau}
(This is \cite[Example 6.3]{mau-woodward08}, reproduced here to help explain our terminology.) Consider the tree from Figure \ref{fig:4-tree}. We have
\begin{equation}
Q_T = \bZ e_1 \oplus \cdots \oplus \bZ e_6 / (e_1-e_5, e_4-e_6, e_1+e_2-e_3-e_4)
\end{equation}
which is freely generated by $a = e_1 = e_5$, $b = e_1+e_2 = e_3+e_4 = \cdots$, $c = e_4 = e_6$. In those terms, the generators of $Q_{T,\geq 0}$ are $x_1 = a$, $x_2 = b-a$, $x_3 = b-c$, $x_4 = c$, with the relation $x_1+x_2 = x_3+x_4$. The gluing parameters \eqref{eq:gluing-corner} are correspondingly
\begin{equation}
\{ (\lambda_1,\lambda_2,\lambda_3,\lambda_4) \in (\bR^{\geq 0})^4 \;:\; \lambda_1\lambda_2 = \lambda_3\lambda_4\}.
\end{equation}
This is the set of real nonnegative points of a three-dimensional toric variety. By the moment map $(\lambda_2^2-\lambda_1^2, \lambda_4^2-\lambda_3^2, \lambda_1^2 + \lambda_3^2)$,
it is mapped homeomorphically to the cone in $\bR^3$ spanned by $(1,0,0)$, $(0,1,0)$, $(-1,0,1)$, $(0,-1,1)$, which is an infinite four-sided pyramid.
\begin{figure}
\begin{centering}
\includegraphics{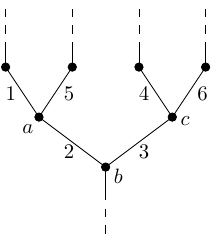}
\caption{\label{fig:4-tree}The tree from Example \ref{th:mau}.}
\end{centering}
\end{figure}
\end{example}

\begin{lemma}
Take an arbitrary element $q = \sum_v q_v\,v \in Q_T$, where following \eqref{eq:non-leaf} the sum is over non-leaf vertices, $q_v \in \bZ$. If $q \in Q_{T,\geq 0}$, the following holds:
\begin{equation} \label{eq:sub-tree-inequalities}
\parbox{34em}{for any rooted sub-tree (a sub-tree of $T$ containing the root), the sum of the $q_v$ over all $v$ which are vertices of the sub-tree is nonnegative.}
\end{equation}
\end{lemma}

\begin{proof}
This is clear, since the generators of $Q_{T,\geq 0}$ satisfy those inequalities.
\end{proof}

\begin{lemma} \label{th:inequalities}
Take a nonzero element $q \in Q_T$ satisfying \eqref{eq:sub-tree-inequalities}. Then, one can subtract one of the generators of $Q_{T,\geq 0}$ from $q$, so that the outcome still satisfies \eqref{eq:sub-tree-inequalities}.
\end{lemma}

\begin{proof}
{\em Suppose that the coefficients in $q$ are all nonnegative}. By assumption, $q_{v_1} > 0$ for at least one (non-leaf) $v_1$. Take any edge which goes from another vertex $v_0$ to $v_1$, and subtract the corresponding generator. The outcome still has nonnegative coefficients, hence trivially satisfies \eqref{eq:sub-tree-inequalities}.

{\em Now, suppose that $q_{v_0} < 0$ for some $v_0$} (which by definition can't be a leaf; it also can't be the root, since that coefficient is always nonnegative). Consider the unique edge starting at $v_0$, with its other endpoint denoted by $v_1$. Subtract the generator corresponding to that edge, forming $q' = q - v_1 + v_0$. If we have a rooted subtree of $T$ which either contains none of $(v_0,v_1)$, or both of them, the associated sum of coefficients for $q'$ is the same as that of $q$, hence still nonnegative. If we have a rooted subtree which contains $v_1$ but not $v_0$, the corresponding sum of coefficients for $q'$ is 
\begin{equation}
\Big(\sum_{\substack{v \text{ in our}\\ \text{subtree}}} q_v\Big) - 1
\geq \Big(\sum_{\substack{v \text{ in our}\\ \text{subtree}}} q_v\Big) + q_{v_0};
\end{equation}
the sum on the right corresponds to a larger rooted subtree of $T$, hence is still nonnegative by assumption.
\end{proof}

\begin{lemma} \label{th:inequalities-are-enough}
The monoid $G_{T,\geq 0} \subset G_T$ is characterized by \eqref{eq:sub-tree-inequalities}.
\end{lemma}

\begin{proof}
Start with a nonzero element of $G_T$ satisfying \eqref{eq:sub-tree-inequalities}, and apply Lemma \ref{th:inequalities}. An inspection of the proof shows that for the new element constructed there, the coefficient sum over any rooted sub-tree does not increase, and must decrease for at least one sub-tree. By an easy argument starting with the root, the only element for which all inequalities are equalities is zero. Hence, after finitely many iterations, we must reach zero. We have then written our original element of $G_T$ as a sum of generators of $G_{T,\geq 0}$.
\end{proof}

\begin{example}
In the example from Figure \ref{fig:4-tree}, the inequalities \eqref{eq:sub-tree-inequalities} are
\begin{equation}
q_b \geq 0, \;\; q_a + q_b \geq 0, \;\; q_b + q_c \geq 0, \;\; q_a + q_b + q_c \geq 0.
\end{equation}
In terms of \eqref{eq:gluing-corner}, these four inequalities correspond to the edges of the pyramid.
\end{example}

Lemma \ref{th:inequalities-are-enough} implies that the monoid $Q_{T,\geq 0}$ is toric (sharp and saturated) \cite[Definition 3.1]{joyce16}, which by definition means that \eqref{eq:gluing-corner} is a generalized corner. Using that as the main ingredient, one proves (compare \cite[Proposition 6.2]{mau-woodward08}):

\begin{lemma} \label{th:rejoyce}
$\overline{\MC}_m$ is a compact manifold with generalized corners, in the sense of \cite{joyce16}.
\end{lemma}

The action of the symmetric group $S_m$ on $\MC_m$ extends to the compactification. Generally, this action permutes boundary strata. The subgroup $S_T \subset S_m$ which preserves a stratum \eqref{eq:m-t-space} is the group of those automorphisms of the tree $T$ which fix the output edge (this becomes a subgroup of $S_m$ by thinking of its action on the other $m$ semi-infinite edges).

\begin{lemma} \label{th:codim-2-non-free}
The subset of points of $\MC_T$ fixed by any nontrivial element of $S_T$ has codimension $\geq 2$ in that stratum.
\end{lemma}

\begin{proof}
The action of $S_r$ on $\mathit{Conf}_r^{\mathrm{ord}}(\bC)/(\bC \rtimes \bR^{>0})$ is free. Hence, if an element $\sigma \in S_T$ has nontrivial fixed point set, it must act trivially on the set of edges that end at the root vertex. By propagating that argument upwards along the tree, one sees that $\sigma$ acts trivially on vertices, and only permutes the semi-infinite edges adjacent to each leaf. In other words, our element must lie in $\prod_{v \text{ leaf}} S_{|v|-1} \subset S_T$, and it's sufficient to consider the action of that group on $\prod_{v \text{ leaf}} \MC_{|v|-1}$; there, the statement is obvious.
\end{proof}

\begin{remark}
For our purpose, it would be possible to work with the unordered version $\MC_m/S_m$, whose compactification is also a manifold with generalized corners (one can see that from the analysis of the group action performed above). The quotient has less boundary strata; each such stratum is 
\begin{equation}
\MC_T/S_T = 
\frac{
\prod_{v \text{\rm { }leaf}} (\MC_{|v|-1}/\Sigma_{|v|-1}) \times \prod_{v \text{\rm { }not a leaf}} \mathit{Conf}_{|v|-1}^{\mathrm{ord}}(\bC)/(\bC \rtimes \bR^{>0})
}
{ 
S_T/(\prod_{v \text{\rm { }leaf}} S_{|v|-1})
},
\end{equation}
with the group $S_T/(\prod_{v \text{\rm { }leaf}} S_{|v|-1})$ acting freely. Ultimately, because of Lemma \ref{th:codim-2-non-free}, working with $\MC_m/S_m$ leads to the same Floer-theoretic data as working with $\MC_m$ and dividing by $m!$ (where one can avoid having actual denominators by using $S_m$-invariant data). We have preferred the latter formulation, since having an inductive structure of boundary strata as products \eqref{eq:m-t-space} is a more familiar kind of setup.
\end{remark}

\subsection{Extraction spaces\label{sec:extractionspaces}}
We now introduce a version of the previous moduli spaces for the cylinder. Here, since the cylinder does not carry a rescaling automorphism, the scale (assumed to be bounded below, because of our intended application) will essentially be a separate variable. Fix $m>0$. The basic objects are pairs of $z_1,\dots,z_m \in \bR \times S^1$ and $\alpha = a \,\mathit{dz}$, where $\mathit{dz} = \mathit{ds} + i\mathit{dt}$ is the standard complex one-form on the cylinder, and $a \geq 1$. We consider such pairs up to translation in $s$-direction. The parameter space is therefore
\begin{equation}
\EX_m = \big( (\bR \times S^1)^m \times [1,\infty) \big)/\bR =
\bR^{m-1} \times (S^1)^m \times [1,\infty).
\end{equation}
The structure of the compactification $\overline{\EX}_m$ is straightforward as long as the scale variable $a$ stays bounded: in that case, the original cylinder can split into several ones, each of which carries the same scale, giving a boundary stratum which is a fibre product
\begin{equation} \label{eq:p-split}
\EX_{m_1} \times_{[1,\infty)} \EX_{m_2} \times_{[1,\infty)} \cdots 
\times_{[1,\infty)} \EX_{m_p}. 
\end{equation}
There is one such stratum for each ordered partition of $\{1,\dots,m\}$ into $p$ subsets of sizes $m_1,\dots,m_p$. Now let's look at limits as $a \rightarrow \infty$, assuming for simplicity that the cylinder itself does not split (which means that the marked points on it remain in a bounded subset, up to an overall translation). There is a limiting principal component, which is a cylinder with $r \geq 1$ pairwise distinct marked points (again up to translation in the $\bR$-direction). We choose a partition $m = m_1 + \cdots + m_r$, and for each $j$, a tree $T_j$ with $j+1 \geq 2$ semi-infinite edges, as in Section \ref{subsec:woodward}. The corresponding stratum is 
\begin{equation} \label{eq:p-bubble}
\EX_{T_1,\dots,T_r} = \prod_j \MC_{T_j} \times 
\mathit{Conf}_{r}^{\mathrm{ord}}(\bR \times S^1)/\bR;
\end{equation}
its codimension is the number of non-leaf vertices in all the trees, plus one (because of the condition $a=\infty$). The general situation combines that with cylinder-splitting, which means that the trees are attached to one of several cylindrical components. The number of extra (more than $1$) cylinder components gets added to the codimension. See Figures \ref{fig:extractihedra-0} and \ref{fig:extractihedra-1} for examples. In particular, the codimension one boundary strata are:
\begin{enumerate}[label=(EX\arabic*)] \itemsep.5em
\item \label{item:scale-1}
$(\bR \times S^1)^m/\bR$, which appears when the scale reaches $a = 1$.

\item \label{item:equal-scale-splitting}
$\EX_{m_1} \times_{[1,\infty)} \EX_{m_2}$, the $p = 2$ case of \eqref{eq:p-split}. More precisely, there is one such stratum for each decomposition of $\{1,\dots,m\}$ into two nonempty subsets.

\item \label{item:mc-bubble}
The case of \eqref{eq:p-bubble} in which each tree has a single vertex, leading to $a=\infty$ and
\begin{equation} \label{eq:leaf-or-root}
\prod_{j=1}^r \MC_{m_j} \times \mathit{Conf}_r^{\mathrm{ord}}(\bR \times S^1)/\bR.
\end{equation}
\end{enumerate}
\begin{figure}
\begin{centering}
\includegraphics{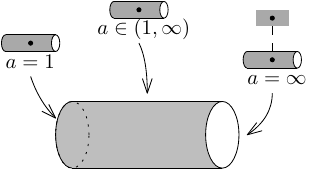}
\caption{\label{fig:extractihedra-0}The space $\overline{\EX}_1$ is an annulus (note the overall picture represents the moduli space, while the smaller parts are the relevant Riemann surfaces with marked points). On the interior, the variables are the scale $a \geq 1$ and the $t$-coordinate of the marked point on the cylinder. As $a \rightarrow \infty$ a new factor $\MC_1 = \mathit{point}$ appears in the relevant case of \eqref{eq:leaf-or-root}. This motivates the word ``extraction'' in the name, as the $a \rightarrow \infty$ limit causes the marked point to pop out of the cylinder.}
\end{centering}
\end{figure}%
\begin{figure}
\begin{centering}
\includegraphics{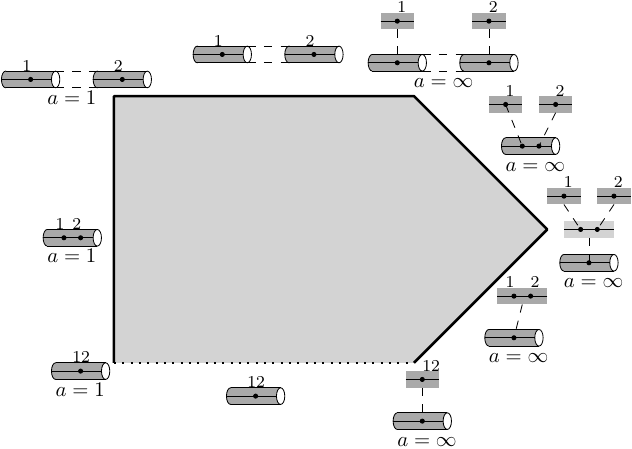}
\caption{\label{fig:extractihedra-1}A piece of $\overline{\EX}_2$, with the configurations associated to its boundary sides and corners. More precisely, this is the closure of the (codimension $2$) subset of $\EX_2$ where $z_k = (s_k,t_k)$ with $t_k = 0$ and $s_1 \leq s_2$. The dotted line at the bottom is not part of $\partial\overline{\EX}_2$, it's just where we stopped drawing because $s_1 = s_2$. Note also the unique component that's divided by rescaling, at the rightmost vertex.}
\end{centering}
\end{figure}%

The discussion of the basic properties of these spaces follows the same pattern as in Section \ref{subsec:woodward}. Consider a point in a stratum \eqref{eq:p-bubble}. Take the disjoint union of the $T_j$, connect all their output edges to a new vertex, and add a new output to that vertex, which yields a combined tree $T$ with $(m+1)$ semi-infinite edges. The gluing parameters $\lambda_e$ are associated to the finite edges of $T$, with the same constraints \eqref{eq:gluing-corner} as before (the scale parameter $a$ after gluing is the inverse product of all the $\lambda_e$ along a path from a leaf to the root of $T$). More generally, if the cylinder splits, we have additional parameters for gluing the cylindrical pieces to each other. By borrowing the arguemnt from our previous discussion of $Q_{T,\geq 0}$, one gets:

\begin{lemma}
$\overline{\EX}_m$ is a compact generalized manifold with corners.
\end{lemma}

A stratum-by-stratum analysis shows the analogue of Lemma \ref{th:codim-2-non-free}:

\begin{lemma} \label{th:codim-2-non-free-2}
Inside each stratum of $\overline{\EX}_m$, the subset which is fixed by any nontrivial element of the subgroup of $S_m$ preserving that stratum has codimension $\geq 2$.
\end{lemma}

In parallel with our usual strategy, we also need a version of the parameter spaces where one does not divide by translation,
\begin{equation}
\EX_m^\dag = (\bR \times S^1)^m \times [1,\infty).
\end{equation}
In the compactification $\overline{\EX}_m^\dag$, when the cylinder splits into several ones, all those except for the original are divided by translation (and each still carries the same scale). The counterpart of the stratum \eqref{eq:p-bubble}, for the case when the marked points remain in a bounded subset, is
\begin{equation} \label{eq:p-bubble-2}
\EX_{T_1,\dots,T_r}^\dag = \prod_j \MC_{T_j} \times \mathit{Conf}_r^{\mathrm{ord}}(\bR \times S^1).
\end{equation}
The general $a \rightarrow \infty$ boundary strata are products of one \eqref{eq:p-bubble-2} factor and an arbitrary number of \eqref{eq:p-bubble} factors, with the latter ones corresponding to other cylindrical components. One notable difference is that now, the empty cylinder with no marked points ($m = 0$) is allowed, with $\overline{\EX}_0^\dag = [1,\infty]$ and no bubbling happening as $a \rightarrow \infty$. This affects the structure of the higher-dimensional moduli spaces as well, because the empty cylinder can appear as a component in the limit when all the marked points go to infinity, either on the left or right; see Figure \ref{fig:extractihedra-2} for an example.
\begin{figure}
\begin{centering}
\includegraphics{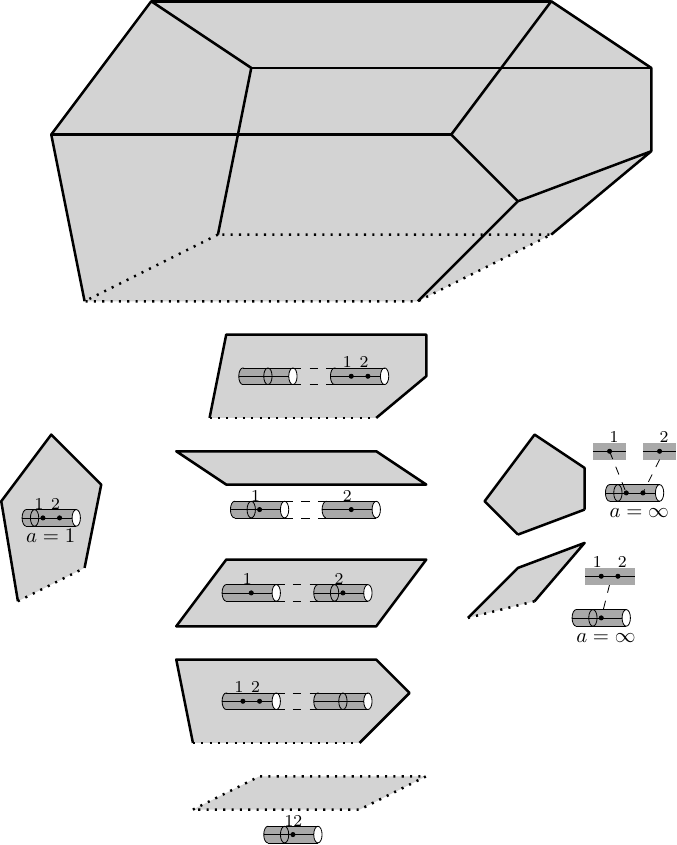}
\caption{\label{fig:extractihedra-2}A piece of $\overline{\EX}_2^\dag$, chosen in parallel with Figure \ref{fig:extractihedra-1}. For simplicity, only the surfaces associated to codimension $1$ boundary faces have been listed. As usual, the principal cylinder components (those not divided by translation) are distinguished by drawing an extra circle on them. Those principal cylinders can be empty (carry no marked points). Two of the pentagonal faces are copies of $\overline{\EX}_2$, with such an empty cylinder inserted on the left or right.}
\end{centering}
\end{figure}%

\subsection{Deformations by Maurer-Cartan elements\label{subsec:punctures}}
We will modify the construction of deformed symplectic cohomology from \cite[Section 5.3e]{pomerleano-seidel23}, so that it takes place on the telescope complex from Section \ref{subsec:sh}. To lighten the notation, we want to use the same sequence of Hamiltonians throughout; the simplest way to do that is to assume that the slopes $\sigma_w$ of the Hamiltonians $\bar{H}_w$, for $w \geq 1$, satisfy
\begin{equation} 
\sigma_w = w \sigma_1 \;\; \text{for some irrational $\sigma_1>1$.}
\end{equation} 
The downside is that this makes no sense for $w = 0$. For that reason, we will start the telescope complex $C$ with $\mathit{CF}(1)$. This is largely irrelevant, since the labeling of slopes by natural numbers is somewhat arbitrary (and the quasi-isomorphism type of the telescope complex is in any case independent of any initial segment). 

For $m \geq 2$, take $(z_1,\dots,z_m) \in \mathit{Conf}_m^{\mathrm{ord}}(\bC)$, as well as $w_1,\dots,w_m \geq 1$. Set $w_- = w_1 + \cdots + w_m$. Equip the punctured plane 
\begin{equation} \label{eq:punctured-plane}
S = \bC \setminus \{z_1,\dots,z_m\} 
\end{equation}
with cylindrical ends 
\begin{equation} \label{eq:fm-ends}
\left\{
\begin{aligned}
& (-\infty,0] \times S^1 \longrightarrow S && \text{ around $\infty$}, \\
& [0,\infty) \times S^1 \rightarrow S && \text{ around $z_1,\dots,z_m$.}
\end{aligned}
\right.
\end{equation}
These ends should be chosen as in \cite[Equation (5.2.3)]{pomerleano-seidel23}. This corresponds to a specific convention for the ``asymptotic markers'': 
\begin{itemize}
\item (\em Aligned asymptotic markers \cite[Figure 5.2]{pomerleano-seidel23})
At the $z_k$, the markers point in negative real direction; at $\infty$, the marker points in direction of the path $[0,1) \rightarrow \bC \cup \{\infty\}$, $r \mapsto -1/r$. 
\end{itemize}
A choice of auxiliary inhomogeneous data consists of a family of almost complex structures 
\begin{equation} \label{eq:j-data}
(J_z)_{z \in S}
\end{equation}
and a one-form with values in Hamiltonians,
\begin{equation} \label{eq:k-data}
K \in \Omega^1(S,\smooth(M,\bR)).
\end{equation}
The general $J_z$ should preserve $D$; additionally, over the ends \eqref{eq:fm-ends} they should equal $\bar{J}_{w_k}$ around $z_k$, and $\bar{J}_{w_-}$ around $\infty$. Similarly, \eqref{eq:k-data} must have the general property that for each $\xi \in TS$, $K(\xi)$ respects $D$; and over the ends it should equal $\bar{H}_{w_k} \mathit{dt}$ near $z_k$, respectively $\bar{H}_{w_-} \mathit{dt}$ near infinity.

For each $w_1,\dots,w_m$, we make choices of ends \eqref{eq:fm-ends} and data \eqref{eq:j-data}, \eqref{eq:k-data} universally over the space $\FM_m$ from \eqref{eq:fm}. This is subject to two constraints: one is equivariance with respect to the symmetric group $S_m$, and the other is consistency with respect to gluing surfaces together. The equivariance condition is unproblematic insofar as transversality is concerned, because $S_m$ acts freely. Take periodic orbits $x_k$ of $\bar{H}_{w_k}$, $k = 1,\dots,m$, and $x_-$ of $\bar{H}_{w_-}$, lying outside $D$; we consider solutions of Cauchy-Riemann equations with those periodic orbits as limits, and which also remain disjoint from $D$. Write the resulting moduli space as $\FM_m(x_-,x_1,\dots,x_m)$. The intersection and winding number arguments from Section \ref{subsec:basics} (in this case, corresponding to the special situation from Example \ref{th:gromov-limit-2}) can be extended to this construction. The outcome is that in the compactification $\overline{\FM}_m(x_-,x_1\dots,x_m)$, all one-periodic orbits that appear lie outside $D$; the maps never intersect $D$; and there is no sphere bubbling. Counting points in zero-dimensional moduli spaces yields operations
\begin{equation} \label{eq:linfinitybigraded} 
\ell^m: \mathit{CF}^*(w_1)\otimes \mathit{CF}^*(w_2)\otimes \cdots \otimes \mathit{CF}^*(w_m) \longrightarrow \mathit{CF}^{*+1-2m}(w_-),
\end{equation}
which are graded symmetric with respect to permutation of inputs. Set
\begin{equation} \label{eq:lie}
L = \bigoplus_{w \geq 1} \mathit{CF}(w).
\end{equation}
The maps \eqref{eq:linfinitybigraded}, combined with the Floer differential $\ell^1 = d_0$, make $L[1]$ into an $L_\infty$-algebra (see e.g.\ \cite[Equation (3.2.3)]{pomerleano-seidel23} for sign conventions). 

\begin{remark} \label{th:linfty-remark}
The (operadic) relation between Fulton-Macpherson spaces and $L_\infty$-algebras is classical. For the implementation in symplectic cohomology, using Hamiltonians with infinite slope, see e.g.\ \cite[Proposition 5.3.2]{pomerleano-seidel23} (of course, this is also part of the much more general theory of operations in \cite{abouzaid-groman-varolgunes22}). The construction given here is only partial, because it does not include continuation maps between different slopes; what one should really do is to extend the $L_\infty$-structure to the telescope complex, by introducing further parameter spaces (which would also destroy the bigrading present in $L$). This is carried out in \cite{el-alami-sheridan24a}.
\end{remark}

Since we are working with integer coefficients, the following divisibility property is useful:

\begin{lemma} \label{th:divisible}
Consider the subspace $(L^{\otimes m})^{S_m}$ fixed by the symmetric group $S_m$ (acting with Koszul signs). The restriction of $\ell^m$ to that subspace is divisible by $(m!)$.
\end{lemma}

\begin{proof}
Take an $m$-tuple of periodic orbits $(x_1,\dots,x_m)$. Let $G \subset S_m$ be the group which preserves that $m$-tuple ($G \iso S_{m_1} \times \cdots \times S_{m_k}$ for $m_1+\cdots+m_k = m$, depending on the coincidences between the $x_i$). Consider the expression
\begin{equation} \label{eq:average}
\frac{1}{|G|} \sum_{\sigma \in S_m} (-1)^\S\, x_{\sigma(1)} \otimes \cdots \otimes x_{\sigma(m)},
\end{equation}
where $\S$ is the standard Koszul sign. If there are two equal $x_i$ of odd degree, the expression is zero. Otherwise, it is an element of $(L^{\otimes m})^{S_m}$. One can show its integrality by rewriting the formula as a sum over shuffles, which are permutations such that $\sigma(i) < \sigma(j)$ whenever $x_i = x_j$ for any $i<j$; this combines the terms in \eqref{eq:average} into groups of $|G|$ each. The nonzero elements \eqref{eq:average} form a basis for $(L^{\otimes m})^{S_m}$. Applying the $L_\infty$-operation to such a basis element yields
\begin{equation} \label{eq:ell-g}
\frac{m!}{|G|}\, \ell^m(x_1,\dots,x_m).
\end{equation}
Since the underlying zero-dimensional moduli spaces $\FM_m(x_-,x_1,\dots,x_m)$ carry a free action of $G$, $\ell^m(x_1,\dots,x_m)$ is divisible by $|G|$, hence \eqref{eq:ell-g} is a multiple of $(m!)$.
\end{proof}

The next construction involves two parameter spaces (for $m>0$)
\begin{equation} \label{eq:c-c-hat}
\begin{aligned} 
& \frakC_m = \mathit{Conf}_m^{\mathrm{ord}}(\bR \times S^1)/\bR, \;\; m \geq 1, \\
& \frakC_m^\dag = \mathit{Conf}_m^{\mathrm{ord}}(\bR \times S^1), \;\; m \geq 0,
\end{aligned}
\end{equation}
and their Fulton-MacPherson style compactifications. Fix $w_+,w_1,\dots,w_m \geq 1$, and set $w_- = w_+ + w_1 + \cdots + w_m$, respectively $w_-^\dag = w_+ + w_1 + \cdots + w_m+1$.. Take a point configuration. For the resulting punctured cylinder $Z = (\bR \times S^1) \setminus \{z_1,\dots,z_m\}$, one adopts the following convention:
\begin{itemize}
\item {(\em $S^1$-invariant asymptotic markers \cite[Figure 5.3]{pomerleano-seidel23})}
The asymptotic markers at $\pm\infty$ point in direction of the paths $[0,1) \rightarrow \bR \times S^1$, $r \mapsto (\pm 1/r,0)$. At the points $z_k$, the markers point towards $-\infty$.
\end{itemize}
The corresponding choices of ends are written down in \cite[Equations (5.2.8)--(5.2.10)]{pomerleano-seidel23}. One also chooses analogues of \eqref{eq:j-data} and \eqref{eq:k-data}, which over the ends reduce to: $(\bar{H}_{w_k} \mathit{dt}, \bar{J}_{w_k})$ near the $z_k$; $(\bar{H}_{w_+} \mathit{dt}, \bar{J}_{w_+})$ near $+\infty$; and the same for either $w_-$ or $w_-^\dag$ near $-\infty$, depending on the type of moduli space. These should satisfy the same conditions as before, namely $S_m$-equivariance and consistency; but where consistency now refers to gluing in Fulton-MacPherson configurations at the $z_k$, as well as to gluing several cylinders together at $\pm\infty$. Moreover, in the case of $\frakC_m^\dag$, we want to use the same data as in the definition of the continuation map $d_0^\dag$.

As before, we only consider maps $u: Z \rightarrow M$ whose limits lie outside $D$, and which avoid $D$ altogether. The outcome of counting points in zero-dimensional moduli spaces are operations
\begin{align}
\label{eq:cm-operation} & 
c^{m+1}: \mathit{CF}^*(w_1) \otimes \cdots \mathit{CF}^*(w_m) \otimes \mathit{CF}^*(w_+) \longrightarrow \mathit{CF}^{*+1-2m}(w_-), \\
\label{eq:cm-dag-operation} & 
c^{m+1,\dag}: \mathit{CF}^*(w_1) \otimes \cdots \mathit{CF}^*(w_m) \otimes \mathit{CF}^*(w_+) \longrightarrow \mathit{CF}^{*-2m}(w_-^\dag).
\end{align}
To these, we add $c^1 = d_0$ (recall also that $c^{1,\dag} = d_0^\dag$, by construction). These operations are graded symmetric, and they satisfy the following equations. First,
\begin{equation}
\begin{aligned}
& 0 = \sum_{k,\sigma} (-1)^\S \, c^{m-k+2}(\ell^k(x_{\sigma(1)},\dots,x_\sigma{(k)}),x_{\sigma(k+1)},\dots,x_{\sigma(m)},x) \\
&  + \sum_{k,\sigma} (-1)^{\S + |x_{\sigma(1)}| + \cdots + |x_{\sigma(k)}|} c^{k+1}(x_{\sigma(1)},\dots,x_{\sigma(k)}, c^{m-k+1}(
x_{\sigma(k+1)},\dots,x_{\sigma(m)},x)); \\
\end{aligned}
\end{equation}
Here, the first sum is over all $1 \leq k \leq m$ and all $(k,m-k)$-shuffles $\sigma$, which means $\sigma(1) < \cdots < \sigma(k)$ and $\sigma(k+1) < \cdots < \sigma(m)$; and the second sum is similar, but allowing $k = 0$ as well. The second equation is similar, but slightly more complicated:
\begin{equation}
\begin{aligned}
& 0 = \sum_{k,\sigma} (-1)^\S c^{m-k+2,\dag}(\ell^k(x_{\sigma(1)},\dots,x_{\sigma(k)}),x_{\sigma(k+1)},\dots,x_{\sigma(m)},x) \\
&  - \sum_{k,\sigma} (-1)^\S c^{k+1}(x_{\sigma(1)},\dots,x_{\sigma(k)}, c^{m-k+1,\dag}(x_{\sigma(k+1)},\dots,x_{\sigma(m)},x)) \\ & 
+ \sum_{k,\sigma} (-1)^{\S + |x_{\sigma(1)}| + \cdots + |x_{\sigma(k)}|} c^{k+1,\dag}(x_{\sigma(1)},\dots,x_{\sigma(k)}, c^{m-k+1}(x_{\sigma(k+1)},\dots,x_{\sigma(m)},x)).
\end{aligned}
\end{equation}
The combination of these operations makes the telescope complex $C$ into an $L_\infty$-module over $L[1]$ (see \cite[Equation (3.4.9)]{pomerleano-seidel23} for sign conventions). We denote this module structure by $c^{m+1}_C$, where $c^1_C = d_C$. On the subcomplex $\bigoplus_m \mathit{CF}(m)$, it is straightforwardly given by \eqref{eq:cm-operation}; on the other part one has, for $m>0$,
\begin{equation}
c^{m+1}_C(x_1,\dots,x_m,\eta x) = (-1)^{|x_1|+\cdots+|x_m|} (
c^{m+1,\dag}(x_1,\dots,x_m,x) -\eta c^{m+1}(x_1,\dots,x_m,x)).
\end{equation}

\begin{remark}
Recall that the differential $d_C$ includes $-\mathit{id}: \eta \mathit{CF}(w) \rightarrow \mathit{CF}(w)$. In the $L_\infty$-module equation, this gives rise to terms
\begin{equation}
\begin{aligned}
& 
(-1)^{|x_1|+\cdots+|x_m|} c^{m+1}_C(x_1,\dots,x_m,c^1_C(\eta x)) \\ & \qquad =
-(-1)^{|x_1|+\cdots+|x_m|} c^{m+1}(x_1,\dots,x_m,x) + \cdots
\end{aligned}
\end{equation}
and
\begin{equation}
\begin{aligned}
&
c^1_C(c^{m+1}_C(x_1,\dots,x_m, \eta x)) \\ & \qquad 
= -(-1)^{|x_1|+\cdots+|x_m|} c^1_C( \eta c^{m+1}(x_1,\dots,x_m,x)) \\ & \qquad
= (-1)^{|x_1|+\cdots+|x_m|} c^{m+1}(x_1,\dots,x_m,x) + \cdots
\end{aligned}
\end{equation}
which cancel each other.
\end{remark}

The analogue of Lemma \ref{th:divisible}, with the same proof, is:

\begin{lemma} \label{th:divisible-2}
The restriction of $c^{m+1}_C$ to $(L^{\otimes m})^{S_m} \otimes C$ is divisible by $(m!)$.
\end{lemma}

\begin{definition} \label{th:mc}
A Maurer-Cartan element in $L$ is a collection of elements $g_m \in \mathit{CF}^{2-2m}(m)$, $m \geq 1$, which satisfy the sequence of equations 
\begin{equation} \label{eq:maurer-cartan}
\sum_{\substack{r \geq 1 \\ m_1 + \cdots + m_r = m\!\!\!}} (1/r!)\, \ell^k(g_{m_1},\dots,g_{m_r}) = 0.
\end{equation}
Note that $\sum_{m_1+\cdots+m_r = m} g_{m_1} \otimes \cdots \otimes g_{m_r} \in L^{\otimes r}$ is $S_r$-invariant. By Lemma \ref{th:divisible}, this means that the equation \eqref{eq:maurer-cartan} takes place in $L$, without denominators.
\end{definition}

The grading is that required by our application: if $|q| = 2$ as usual, then
\begin{equation}
g_q = \sum_m q^m g_m \in qL[[q]]
\end{equation}
is a Maurer-Cartan element for $L[1]$ in the standard (degree $1$) sense. One can use such an element to define a deformation $C_g$ of $C$, by equipping $C[[q]]$ with the differential
\begin{equation} \label{eq:d-g}
\begin{aligned}
& d_g(x) = \sum_{r \geq 0} (1/r!)\, c^{r+1}_C(g_q,\dots,g_q,x) \\
& \qquad
= \sum_{\substack{r \geq 0\\ m_1,\dots,m_r\!\!\!}} (q^{m_1+\cdots+m_r}/r!) \, c^{r+1}_C(g_{m_1},\dots,g_{m_r},x). 
\end{aligned}
\end{equation}
As before, Lemma \ref{th:divisible-2} ensures that the denominators are only apparent: the formula yields a well-defined differential on $C[[q]]$.

\begin{remark}
We have gone some length to make sure that everything works integrally, to match the setup in the earlier parts of the paper. Obviously, readers happy with $\bQ$-coefficients (or indeed $\bC$-coefficients as in \cite{pomerleano-seidel23}) can skip all of that.
\end{remark}



\subsection{Extracting marked points\label{subsection:finiteslope}}
Take the Maurer-Cartan spaces from Section \ref{subsec:woodward}. For any point in $\MC_m$ choose a representative $(z_1,\dots,z_m,\alpha)$. We equip the complex plane with an end at infinity, chosen according to the same conventions as in Section \ref{subsec:punctures}, and also with data $(K,J)$. These choices, carried out smoothly over $\MC_m$, are subject to the usual conditions: invariance with respect to the action of $S_m$; and consistency with the previous choices of data on $\FM_m$, with respect to \eqref{eq:mc-fm} (while the technical setup is a little different than before, because $\overline{\MC}_m$ is a manifold with generalized corners, it is unproblematic to define and construct consistent data in that context). Given a periodic orbit $x_-$ of $\bar{H}_m$ lying outside $D$, one considers the space $\MC_m(x_-)$ which consists of a point in $\MC_m$, represented by $(z_1,\dots,z_m, \alpha)$, and a map $u: \bC \rightarrow M$ asymptotic to $x_-$ over the end, satisfying the associated Cauchy-Riemann equation as well as the equality (as divisors)
\begin{equation} \label{eq:mc-intersect}
u^{-1}(D) = z_1 + \cdots + z_m.
\end{equation}
To repeat for clarity: this time, the points $z_i$ are not removed from the surface (unlike the constructions from Section \ref{subsec:punctures}, but rather used for an intersection condition, as in previous parts of the paper. 

In the compactification $\overline{\MC}_m(x_-)$ of these spaces, one gets components which correspond to the factors in \eqref{eq:m-t-space}, as well as additional Floer cylinders. By a topological argument as in Section \ref{subsec:basics}, all one-periodic orbits that occur will lie outside $D$; all components labeled by $\MC$ spaces will satisfy the analogue of \eqref{eq:mc-intersect}; and all other components are disjoint from $D$. As for transversality, we first treat the strata in $\MC_m$ which have nontrivial isotropy with respect to the symmetric group; because of \eqref{eq:mc-intersect}, they also involve higher orders of tangency with $D$. Since those strata have codimension $\geq 2$, they will not contribute to spaces $\MC_m(x_-)$ of dimension $\leq 1$. Over the rest of $\MC_m$, one can achieve regularity while keeping $S_m$-equivariance. The outcome is that the compactified spaces $\overline{\MC}_m(x_-)$ of dimension $\leq 1$ are smooth, and carry a free action of the symmetric group. Hence, if we count points in zero-dimensional moduli spaces $\MC_m(x_-)$, the outcome
\begin{equation} \label{eq:tilde-g}
\tilde{g}_m \in \mathit{CF}^{2-2m}(m)
\end{equation}
is divisible by $(m!)$. We write 
\begin{equation} \label{eq:geometric-mc}
g_m = \tilde{g}_m/(m!).
\end{equation}

\begin{lemma}
The geometrically defined \eqref{eq:geometric-mc} form a Maurer-Cartan element, in the sense of Definition \ref{th:mc}.
\end{lemma}

\begin{proof}
As standard in such contexts, the argument is based on looking at the codimension $1$ boundary strata \eqref{eq:mc-codim-1}. When such a stratum occurs as limit in $\overline{\MC}_m(x_-)$, each $\MC_{m_k}$ factor corresponds to a map $\bC \rightarrow M$ satisfying the analogue of \eqref{eq:mc-intersect}, while the $\FM_r$ factor corresponds to a map from an $r$-punctured plane \eqref{eq:punctured-plane} to $M \setminus D$. As already pointed out after \eqref{eq:mc-codim-1}, one has to be careful not to overcount the strata: the number of partitions of $\{1,\dots,m\}$ into $r \geq 2$ ordered subsets of sizes $m_1 + \cdots + m_r = m$ is $m!/(m_1! \cdots m_r!)$, but one has to divide by renumbering the subsets. The outcome, also including the obvious limit where a Floer trajectory split off over the end (as the term $r = 1$, with $\ell^1 = d_0$ the Floer differential), is
\begin{equation}
\sum_{\substack{r \geq 1 \\ \!\!\! m_1+\cdots+m_r = m \!\!\!}} \frac{m!}{m_1! \cdots m_r! r!} \ell^r(\tilde{g}_{m_1},\dots,\tilde{g}_{m_r}) = 0;
\end{equation}
which is equivalent to \eqref{eq:maurer-cartan} for \eqref{eq:geometric-mc}.
\end{proof}

\begin{remark} \label{th:integral}
The $m = 1$ case of \eqref{eq:maurer-cartan} just says that 
\begin{equation}
g_1 = \tilde{g}_1 \in \mathit{CF}(1)^0
\end{equation}
is a Floer cocycle. By using a parametrized version of the construction, one can show that its cohomology class, $[g_1] \in \mathit{HF}(1)$, is independent of all choices. Indeed, it agrees with a special instance of the maps \eqref{eq:tt}. Namely, $t_{1,0}: \mathit{CM}^*(D) \rightarrow \mathit{CF}^*(1)$ is a chain map, by a special case of \eqref{eq:t-map-equation}, and $[g_1]$ is the image of the identity class under that map. Equivalently, in terms of the pseudo-cycle definition from Section \ref{subsec:switch-to-cycles}, one has
\begin{equation} \label{eq:gin-and-tonic}
[g_1] = [t_{1,0}(D)].
\end{equation}
The image of $[g_1]$ under $\mathit{HF}^*(1) \rightarrow \mathit{SH}^*(M \setminus D)$ was called the Borman-Sheridan class in \cite{pomerleano-seidel23}.
\end{remark}

\begin{remark} \label{th:simplify-mc}
One can simplify the situation by assuming, as in Lemma \ref{th:satisfy-epsilon-bound}(ii), that all the $\mathit{CF}^*(w)$ are concentrated in degrees $\geq 0$. Then, a degree zero cocycle is entirely determined by its cohomology class; and the higher order terms of the Maurer-Cartan element are automatically zero. In other words, the Maurer-Cartan element is totally characterized by its property \eqref{eq:gin-and-tonic}. Obviously, vanishing of the higher order terms also simplifies many of the formulae that we are encountering, such as \eqref{eq:d-g}. We have avoided relying on this shortcut in our construction, because parts of the general argument developed here could be useful in other circumstances, where those degree considerations don't apply; but it was used in \cite{pomerleano-seidel23} to bypass the full construction of the Maurer-Cartan element.
\end{remark}

\begin{theorem} \label{th:pullout}
The deformation $C_g$ of the telescope given by the Maurer-Cartan element \eqref{eq:geometric-mc} is isomorphic to $C_q$ (as defined in Section \ref{subsec:sh}, except for our current convention of starting with $w = 1$).
\end{theorem}

This uses the extraction spaces $\EX_m$ and their variants $\EX_m^\dag$. As in the definition of the Maurer-Cartan element, we equip the cylinders parametrized by $\EX_m$ with ends and auxiliary data, modelled near $+\infty$ on $(\bar{H}_{w_+}, \bar{J}_{w_+})$, and near $-\infty$ on either $(\bar{H}_{w_-}, \bar{J}_{w_-})$ for $w_- = w_+ + m$, or $(\bar{H}_{w_-^\dag}, \bar{J}_{w_-^\dag})$ for $w_-^\dag = w_+ + m + 1$. While the basics are by now repetitious, we want to make sure that the consistency conditions are understood. For simplicity, we look at them principally for codimension $1$ boundary faces (adding occasional remarks about higher codimension behaviour, but without aiming for completeness).
\begin{itemize} \itemsep.5em
\item As a special case, for $\EX_0^\dag$ we always equip the cylinder with the same data as for $\frakD_0^\dag$, irrespective of the scale $a$. (Looking slightly ahead, $a$-independence means that there can't be isolated points in the associated moduli spaces, hence the contribution is zero; in spite of that, we need to have this case set up for consistency reasons.)

\item (This is the situation \ref{item:scale-1} in Section \ref{sec:extractionspaces}.)
The boundary of the (uncompactified) extraction spaces occurs where $a = 1$. Dividing by the symmetric group yields maps
\begin{align}
& \partial \EX_m \longrightarrow \frakD_m = \mathit{Sym}_m(\bR \times S^1)/\bR, \\
\label{eq:ex-quotient}
& \partial \EX_m^\dag \longrightarrow \frakD_m^\dag = \mathit{Sym}_m(\bR \times S^1).
\end{align}
and all data are supposed to be pulled back from those previously chosen to define the differentials $d_m$, $d_m^\dag$. 

\item \ref{item:equal-scale-splitting}
When the cylinder splits into two (or in higher codimension, several) cylinders, carrying equal finite scales, we inherit the data inductively from each of the pieces.

\item \ref{item:mc-bubble}
When the scale goes to $\infty$, the principal component belongs to a configuration space of (distinct) points on the cylinder, and the associated punctured cylinder carries data inherited from our previous choices for the spaces \eqref{eq:c-c-hat}. The other components belong to $\MC$ spaces, and inherit those choices (in particular, unlike what happened to the cylinder, we do not remove the marked points on those components). Finally (in higher codimension) there are other components which are $\FM$ spaces, and again already have choices of auxiliary data prescribed on the associated punctured planes \eqref{eq:punctured-plane}.
\end{itemize}
One has associated spaces $\EX_m(x_-,x_+)$ and $\EX_m^\dag(x_-,x_+)$ of solutions of Cauchy-Riemann equations. These spaces can be used to define maps
\begin{align} \label{eq:ex-operations}
& \mathit{ex}_m: \mathit{CF}^*(w_+) \longrightarrow \mathit{CF}^{*-2m}(w_-), \\
\label{eq:ex-dag-operations}
& \mathit{ex}_m^\dag: \mathit{CF}^*(w_+) \longrightarrow \mathit{CF}^{*-2m-1}(w_-^\dag).
\end{align}

\begin{example}
Let's look at the spaces for $m = 1$, and how they can be applied to solving Theorem \ref{th:pullout} at first order in $q$. We have
\begin{equation}
\begin{aligned}
& \EX_1 = (\bR \times S^1)/\bR \times [1,\infty) = S^1 \times [1,\infty), \\
& \partial \EX_1 = S^1 \times \{1\} = \frakD_1, \\
& \overline{\EX}_1 \setminus \EX_1 = S^1 \times \{\infty\} = \frakC_1 \times \MC_1. 
\end{aligned}
\end{equation}
Algebraically, the outcome is that
\begin{equation} \label{eq:ex-equation}
d_0 \mathit{ex}_1 - \mathit{ex}_1 d_0 = d_1 - c^2(g_1,\cdot).
\end{equation}
In parallel,
\begin{align}
\label{eq:ex-1-dag}
& \EX_1^\dag = \bR \times S^1 \times [1,\infty), \\
\label{eq:ex-1-dag-boundary}
& \partial \EX_1^\dag = \frakD_1^\dag,
\end{align}
whose compactification comes with additional codimension $1$ boundary strata
\begin{align}
\label{eq:c1mc1}
& \bR \times S^1 \times \{\infty\} = \frakC_1^\dag \times \MC_1, \\
\label{eq:d0ex1}
& \EX_0^\dag \times_{[1,\infty)} \EX_1 = \frakD_0^\dag \times \EX_1, \\ 
\label{eq:ex1d0}
& \EX_1 \times_{[1,\infty)} \EX_0^\dag = \EX_1 \times \frakD_0^{\dag}.
\end{align}
In the last two cases, our way of writing the spaces took into account the fact that the data chosen on $\EX_0^\dag$ are scale-independent. The outcome is that
\begin{equation} \label{eq:ex-dag-equation}
d_0 \mathit{ex}_1^\dag + \mathit{ex}_1^\dag d_0 = c^{2,\dag}(g_1,\cdot) - d_1^\dag +
d_0^\dag\, \mathit{ex}_1 - \mathit{ex}_1\, d_0^\dag.
\end{equation}
This example is so simple that one can readily explain the origin of the signs (something that we have otherwise avoided here, because of the lengthy bookkeeping involved; see \cite[Section 5.2g]{pomerleano-seidel23} for the $L_\infty$-operations). If one writes the coordinates on \eqref{eq:ex-1-dag} as $(s_1,t_1,a)$, then for the obvious orientations, the identification \eqref{eq:ex-1-dag-boundary} is orientation-reversing, as the outwards pointing normal vector is $(0,0,-1)$; whereas \eqref{eq:c1mc1} which arises as $a \rightarrow +\infty$, is orientation-preserving. Similarly, \eqref{eq:d0ex1} arises as $s_1 \rightarrow +\infty$ and is therefore orientation-preserving, while \eqref{eq:ex1d0} is $s_1 \rightarrow -\infty$ and orientation-reversing; which matches the right hand side of \eqref{eq:ex-dag-equation}. Resuming the main discussion, define 
\begin{equation} \label{eq:dead-end}
\begin{aligned}
& \mathit{ex}_{C_q}: C_q \longrightarrow C_g, \\
& \mathit{ex}_{C_q}(x) = x + q\,\mathit{ex}_1(x), \\
& \mathit{ex}_{C_q}(\eta x) = \eta (x + q\,\mathit{ex}_1(x)) + q\,\mathit{ex}_1^\dag(x).
\end{aligned}
\end{equation}
The equations \eqref{eq:ex-equation} imply that $\mathit{ex}_{C_q}$ is a chain map up to an error of order $O(q^2)$.
\end{example}

Unfortunately, our luck with \eqref{eq:dead-end} runs out at $q^2$: because of the appearance of the boundary components of type \ref{item:equal-scale-splitting} which are fibre products, the maps \eqref{eq:ex-operations} and \eqref{eq:ex-dag-operations} for $m>1$ do not satisfy meaningful relations. Instead, we will adopt a ``scale-ordered'' approach involving partial products of those spaces (for a model see e.g.\ \cite[Section 10e]{seidel04}). Let
\begin{equation} \label{eq:o-space}
\frakO_m \subset \coprod_{\substack{r \geq 1 \\ m_1 + \cdots m_r = m}} \EX_{m_1} \times \cdots \times \EX_{m_r}
\end{equation}
be the subset where the scales $(a_1,\dots,a_r)$ satisfy $a_1 \geq \cdots \geq a_r$ (any $r$ gives the same dimension, so we're really talking about a disjoint union of topological spaces, not a decomposition of a space into strata). We define
\begin{equation} \label{eq:o-x}
\frakO_m(x_-,x_+) \subset \coprod_{\substack{r \geq 1 \\ m_1 + \cdots m_r = m \\
x_1,\dots, x_{r-1}}} \EX_{m_1}(x_-,x_1) \times \EX_{m_2}(x_1,x_2) \times \cdots \times
\EX_{m_r}(x_{r-1},x_+)
\end{equation}
by the same scale-ordering condition. In words, elements are composable $r$-tuples of elements in our previous moduli spaces (and $\EX_m(x_-,x_+)$ itself appears in \eqref{eq:o-x} as the connected component where $r = 1$). In this situation, we are considering subsets of already defined spaces, and inherit the data already chosen for them. We only need one extra transversality condition, which is that the subset of \eqref{eq:o-x} where there are $k$ coincidences between scales should be of codimension $k$. This is unproblematic, because (due to the increase in slopes) an element of $\frakO_m(x_-,x_+)$ always consists of maps $(u_1,\dots,u_r)$ that belong to different moduli spaces. Similarly, there is no need for a separate discussion of the compactification of \eqref{eq:o-space} or \eqref{eq:o-x}, since we can define that as the closure inside the relevant product of $\overline{\EX}$ spaces. Counting isolated points in \eqref{eq:o-x} yields operations
\begin{equation}
\tilde{o}_m: \mathit{CF}^*(w_+) \longrightarrow \mathit{CF}^{*-2m}(w_-)
\end{equation}
which coincide with \eqref{eq:ex-operations} for $m = 1$. Because of the symmetric group action, they are divisible by $(m!)$. We set
\begin{equation}
o_m = \tilde{o}_m/(m!),
\end{equation}
as well as adding $o_0 = \mathit{id}$. The boundary strata in one-dimensional moduli spaces are as follows:
\begin{enumerate}[label=(O\arabic*)] \itemsep.5em
\item \label{item:o-one}
The last scale reaches $a_r = 1$ (if $r>1$, the remaning scales $a_1,\dots,a_{r-1}$ can be arbitrary). 
\item \label{item:o-split}
One of the cylinders splits into two, carrying equal scales.
\item \label{item:o-infinity}
The first scale $a_1$ goes to $+\infty$, yielding a degeneration of that cylindrical component as in \ref{item:mc-bubble} (if $r>1$, the remaining scales $a_2,\dots,a_r$ can be arbitrary values).
\item \label{item:o-equal}
Two successive scales become equal, $a_k = a_{k+1}$.
\item \label{item:o-floer}
Bubbling off of a Floer trajectory on the left or right of any component $u_k$. Note however that these cancel out in pairs, with two exceptions, namely bubbling off on the left of $u_1$ or on the right of $u_r$.
\end{enumerate}
The point of this construction is that \ref{item:o-split} and \ref{item:o-equal} yield the same limiting configuration, and the contribution of those two boundary faces will cancel each other.
The outcome of the remaining boundary points is the following equation:
\begin{equation} \label{eq:o-equation}
\sum_k o_k (d_{m-k}(x)) =
\sum_{\substack{k \\ \!\!\! m_1 + \cdots + m_k \leq m \!\!\!\!\!}} c^{k+1}(g_{m_1},...,g_{m_k},o_{m-m_1-\cdots-m_k}(x)).
\end{equation}
The left hand side expresses \ref{item:o-one} including $d_m(x)$, through the convention $o_0 = \mathit{id}$; it also includes a term $o_m(d_0(x))$ which is part of \ref{item:o-floer}.
Similarly, the left hand side corresponds to \ref{item:o-infinity} including terms $c^{k+1}(g_{m_1},\dots,g_{m_k},x)$; as well as $d_0(o_m(x))$, the other part of \ref{item:o-floer}.

We similarly define 
\begin{equation}
\frakO_m^{\dag} \subset \coprod_{\substack{r,j \\ \!\!\! m_1 + \cdots + m_r = m \!\!\!\!\!}}
\EX_{m_1} \times \cdots \EX^\dag_{m_j} \cdots \times \cdots \EX_{m_r}. 
\end{equation}
Note that while in the previous situations all the $m$'s need to be positive, in this case we allow $m_j = 0$ (even though these spaces will ultimately not contribute, due to the exceptional situation with $\EX_0^\dag$). Algebraically, the same process as before yields
\begin{equation}
o_m^\dag = \tilde{o}_m^\dag/(m!): \mathit{CF}^*(w_+) \longrightarrow \mathit{CF}^{*-2m-1}(w_-^\dag),
\end{equation}
where this time $o_0^\dag = 0$ by construction. The counterpart of \eqref{eq:o-equation} is:
\begin{equation} \label{eq:o-dag-equation}
\begin{aligned}
& \sum_k o_k(d_{m-k}^\dag(x)) + o_k^\dag(d_{m-k}(x)) \\ 
& =
\sum_{\substack{k \\ \!\!\! m_1 + \cdots + m_k \leq m \!\!\!\!\!}} -c^{k+1}(g_{m_1},...,g_{m_k},o_{m-m_1-\cdots-m_k}^\dag(x))
\\[-1em] & \qquad \qquad \qquad \qquad + c^{k+1,\dag}(g_{m_1},...,g_{m_k},o_{m-m_1-\cdots-m_k}(x)).
\end{aligned}
\end{equation}

\begin{remark}
For $m = 1$, \eqref{eq:o-equation} and \eqref{eq:o-dag-equation} reduce to \eqref{eq:ex-equation} and \eqref{eq:ex-dag-equation}, respectively. Because of the amount of notation, it is worth while spelling that out for \eqref{eq:o-dag-equation}:
\begin{equation}
\text{on the left }
\left\{
\begin{aligned}
& o_0 d_1^\dag = d_1^\dag \\
& o_1 d_0^\dag = \mathit{ex}_1 \, d_0^\dag \\
& o_1^\dag d_0 = \mathit{ex}_1^\dag d_0 \\
\end{aligned} 
\right.
\quad
\text{on the right }
\left\{
\begin{aligned}
& -c^1 o_1^\dag = -d_0 \mathit{ex}_1^\dag \\
& c^{1,\dag} o_1 = d_0^\dag \mathit{ex}_1 \\  
& c^{2,\dag}(g_1, o_0(\cdot)) = c^{2,\dag}(g_1,\cdot) 
\end{aligned}
\right.
\end{equation}
\end{remark}

Having that, we can correct the original idea \eqref{eq:dead-end}. Define a map 
\begin{equation}
\begin{aligned}
& o_{C_q}: C_q \longrightarrow C_g, \\
& o_{C_q}(x) = \sum_{m \geq 0} q^m\, o_m(x), \\
& o_{C_q}(\eta x) = \eta\, o_{C_q}(x) + \sum_{m \geq 0} q^m o_m^\dag(x).
\end{aligned}
\end{equation}
As a consequence of \eqref{eq:o-equation} and \eqref{eq:o-dag-equation}, this is a chain map. Moreover, it is the identity modulo $q$, and therefore an isomorphism. This completes the argument for Theorem \ref{th:pullout}.

\begin{remark} \label{th:other-mc}
The preprint \cite{el-alami-sheridan24b} gives another construction of a Maurer-Cartan element (not in our $L_\infty$-algebra $L$, but in a larger version whose cohomology is $\mathit{SH}^*(M)$; compare Remark \ref{th:linfty-remark}). There are numerous differences between the two approaches, making a comparison difficult to formulate succinctly. Roughly speaking, the strategy in \cite{el-alami-sheridan24b} is to enlarge the $L_\infty$-algebra by extra generators, so that the larger $L_\infty$-structure encodes curves with various orders of tangency to $D$. That larger structure admits a tautological Maurer-Cartan element; one then uses an $L_\infty$-automorphism to modify that element so that it comes to lie in the original $L_\infty$-algebra (this description suppresses the crucial role that cleverly constructed filtrations play in the construction). The last-mentioned, purely algebraic, step roughly takes the place which Theorem \ref{th:pullout} occupies in ours; in particular, neither Maurer-Cartan spaces nor extraction spaces are used in \cite{el-alami-sheridan24b}. In spite of that, it is clear that the resulting Maurer-Cartan element is equivalent to ours, because of the simplification permitted by grading considerations (see Remark \ref{th:simplify-mc}).

To complete this rough overview, \cite{el-alami-sheridan24b} then appeals to \cite{borman-sheridan-varolgunes21} to determine the cohomology of the deformed differential (for the tautological Maurer-Cartan element, and hence automatically also for the modified one). This replaces the entirety of our use of thimble maps, Sections \ref{sec:thimbles}--\ref{sec:proof}. As a final point, we should mention that \cite{el-alami-sheridan24b} uses aligned asymptotic markers throughout. Hence, their deformed differential is not a priori the same as ours, but instead is one of the versions mentioned in Section \ref{subsec:rotate} below; however, that is a minor point, as one could adapt their construction to use a different choice of markers instead.
\end{remark}

\subsection{Taking the slope to infinity\label{subsection:toinfinity}}
The pair $(M,D)$ gives rise to a Liouville domain $(\LP, \theta_{\LP})$ as follows. Choose a Hermitian metric $||\cdot||$ on the normal bundle $\pi_{\nu D}: \nu D \to D$. Also choose a Hermitian connection $\nabla$ on $\nu D$, with connection one-form $\upsilon \in \Omega^1(\nu D\setminus D)$ (here $D \subset \nu D$ is the zero-section). This should satisfy
\begin{equation} 
d\upsilon = -\pi_{\nu D}^*(\omega_D),
\end{equation}  
where $\omega_D = \omega_M|D$ (this can be achieved since $c_1(\nu D) = [\omega_D]$). Set $\mu = \half ||\xi||^2$, and consider the closed two-form
\begin{equation} 
\omega_{\nu D} = d(\mu \cdot \upsilon) + \pi_{\nu D}^*(\omega_D).
\end{equation} 
It is straightforward to see that this extends smoothly over the zero section, and is symplectic on $\lbrace \mu < 1 \rbrace \subset \nu D$. Rotation in the fibers of the normal bundle defines a Hamiltonian $S^1$-action with moment map $\mu$. The symplectic tubular neighborhood theorem shows that for some sufficiently small $\epsilon>0$, there is a symplectic embedding 
\begin{equation} \label{eq:symplectictubular}
\psi: \{\|\xi\| \leq 2\epsilon\} \hookrightarrow M,  \quad  
\psi^*(\omega_M) = \omega_{\nu D}
\end{equation}
such that  $\psi|D = \mathit{id}$. Because $D \subset M$ is Poincar\'e dual to the symplectic class, $\omega_M|(M \setminus D)$ is exact. It is not difficult to see (see e.g.\ \cite[Lemma 7.2.1]{pomerleano-seidel23}) that for a suitable choice of connection and symplectic tubular neighborhood, there exists a primitive $\theta_{M \setminus D} \in \Omega^1(M \setminus D)$ of the symplectic form, such that
\begin{equation} \label{eq:theta-alpha}
\psi^*(\theta_{M\setminus D}) = (\mu-1)\upsilon.
\end{equation} 
We fix such a connection and tubular neighborhood once and for all throughout the discussion, and denote its image by $UD$. We set 
\begin{align}
\LP = M \setminus \psi(\{\|\xi\| < \epsilon\}),
\end{align}
and let $\theta_{\LP}=(\theta_{M \setminus D})|_{\LP}$. The Liouville vector field $Z$ defined by $i_Z\omega = \theta_{M\setminus D}$ satisfies \begin{align} 
\label{eq:Liouvillefield} \psi^*(Z) = (\mu-1)\partial_\mu,
\end{align}
where $\partial_\mu$ denotes the multiple of the radial vector field  such that $d\mu(\partial_\mu)=1$ (this is only defined away from the zero-section). It follows from this that the pair $(\LP,\theta_{\LP})$ defines a Liouville domain.  

\begin{lem} 
The associated Liouville coordinate $R:UD \setminus D \to \mathbb{R}$ can be extended smoothly to all of $M$. Any such extension defines a function of slope $\sigma=\frac{1}{1-\epsilon^2/2}$ in the sense of Section \ref{sec:sh}. 
\end{lem} 

\begin{proof} It follows from \eqref{eq:Liouvillefield} that
\begin{align} 
\psi^*(R) = \frac{1-\mu}{1-\epsilon^2/2}. 
\end{align} 
As a consequence, $R$ extends smoothly over $D$, hence can be extended to all of $M$. Next, note that the Hamiltonian $S^1$-action on the normal bundle induces an $S^1$-action on $UD$, whose moment map (normalized to be zero along $D$) is identified with $\mu$ under pullback.
\end{proof}

We now interpret the constructions from Sections \ref{subsec:sh} and \ref{subsec:punctures} in terms of the Liouville domain $\LP$. The key to doing this is to impose stronger conditions on the inhomogeneous data, so that the integrated maximum principle applies. To set the stage for this, choose a sufficiently small constant $\delta>0$, and let 
\begin{align} 
V = \lbrace 1 \leq R \leq 1+\delta \rbrace \subset UD \setminus D 
\end{align} 
be a contact shell around $\partial \LP$, which is disjoint from $D$.  
\begin{definition}
A compatible almost-complex structure $J$ is said to be of contact type along the shell, if over $V$ it satisfies
\begin{align} 
\theta_{M\setminus D} \circ J = dR.  \end{align}
\end{definition} 

Let us start by describing the constraints that we will place on our Floer data $(\bar{H}_w,\bar{J}_w)$: 
\begin{enumerate} [label=(FD\arabic*)] \itemsep .5em 
\item \label{it:FDacs}  
We assume that all of our almost complex structures $\bar{J}_w$ are of contact type along $V$, in addition to the conditions from Section \ref{subsec:sh}. Note that because the shell $V$ is disjoint from the divisor, there is no issue with imposing these conditions simultaneously.  

\item \label{it:FDham}  Let $\mathring{\LP}$ denote the interior of $\LP$; over $M \setminus \mathring{\LP}$, our Hamiltonians $\bar{H}_w$ should be given by 
\begin{align} \label{eq:FloerHamiltonian} 
(\bar{H}_w)_{|M \setminus \mathring{\LP}} = w\,h_1(R),\text{ where } h_1(R)= \sigma_1(1-\half \epsilon^2)R. \end{align} 
\end{enumerate}
An application of the integrated maximum principle \cite[Lemma 7.2]{abouzaid-seidel07} now implies that: 

\begin{lem} 
Suppose that $(\bar{H}_w,\bar{J}_w)$ satisfies \ref{it:FDacs}--\ref{it:FDham}. Then any Floer cylinder whose limits are in $M \setminus D$, and which avoids $D$, must lie entirely in $\LP$. 
\end{lem}

For continuation data $(H_{s,t},J_{s,t})$ between $(\bar{H}_w,\bar{J}_w)$ and $(\bar{H}_{w+1},\bar{J}_{w+1})$, we continue to assume that $J_{s,t}$ is of contact type along the shell $V$. The analogue of \eqref{eq:FloerHamiltonian} is:\begin{align} \label{eq:Floercontinuation} 
(H_{s,t})_{|M \setminus \mathring{\LP}} = h_1(R)f(s), \quad f'(s)\leq 0 \quad \forall s. \end{align}
The integrated maximum principle applies to continuation solutions as well, and the outcome is: 

\begin{cor} There is an isomorphism of cohomologies $H^*(C,d_C) \cong SH^*(\LP)$, where $SH^*(\LP)$ is the symplectic cohomology of the Liouville domain $\LP$. \end{cor}

The same approach applies to $L_\infty$-operations and module structures. Namely, concerning inhomogeneous data $((J_z)_{z \in S}, K)$ we assume that: 
\begin{enumerate}[label=(PD\arabic*)]  \itemsep .5 em 
\item \label{it:pdacs} the complex structure $(J_z)_{z \in S}$ are all of contact type along the shell $V$.  

\item \label{it:pdham} the one-form $K \in \Omega^1(S,\smooth(M,\bR))$, satisfies
\begin{equation} \label{eq:k-data2}
K_{|S \times (M \setminus \mathring{\LP})}= h_1(R)\beta_S, 
\end{equation}
where $\beta_S \in \Omega^1(S,\bR)$ is a subclosed one-form.
\end{enumerate}
As before, the integrated maximum principle implies that solutions $u:S \to M \setminus D$ to the Floer equation associated to such ($(J_z)_{z \in S}, K)$ actually lie in $\LP$. 

With this in place, we are finally in a position to relate our constructions to the framework of quadratic Hamiltonians employed in \cite{pomerleano-seidel23}. To do this, we pass to the Liouville completion of $\LP$; this is the pair  $(\hat{\LP}, \omega_{\hat{\LP}} = d\theta_{\hat{\LP}})$ where:
\begin{equation} \label{eq:liouville-completion}
\begin{aligned}
& \hat{\LP} = \LP\cup_{\partial \LP} ([1,\infty) \times \partial \LP), \\
& \theta_{\hat{\LP}}\, |\, ([1,\infty) \times \partial \LP) = R\ (\theta_{\LP})_{|{\partial \LP}}.
\end{aligned}
\end{equation}
The constructions above can equally well be viewed as taking place inside the Liouville completion. Namely, given Floer data $(\bar{H}_w,\bar{J}_w)$ satisfying \ref{it:FDacs}--\ref{it:FDham}, we can consider Floer data $(\hat{H}_w,\hat{J}_w)$ over $\hat{\LP}$ where: 
\begin{itemize} \itemsep.5em
\item  $\hat{J}_w$ is any almost complex structure agreeing with $\bar{J}_w$ over $\LP \cup V \subset \hat{\LP}$.

\item $\hat{H}_w$ agree with $\bar{H}_w$ over $\LP$, and satisfy \eqref{eq:FloerHamiltonian} over the entire cone. 
\end{itemize}
The integrated maximum principle shows that $(\hat{H}_w,\hat{J}_w)$ has the same Floer trajectories as $(\bar{H}_w,\bar{J}_w).$ Given perturbation data $((J_z)_{z \in S}, K)$ satisfying \ref{it:pdacs}--\ref{it:pdham}, we can similarly construct perturbation data $((\hat{J}_z)_{z \in S}, \hat{K})$ over $\hat{\LP}$, giving rise to the same $L_\infty$ algebras and module structures. 

We let $(H_\infty,J_\infty)$ be a pair consisting of a quadratic time-dependent Hamiltonian $H_\infty$ on $\hat{\LP}$, and compatible almost complex structures $J_\infty$, fitting into the analytical framework of \cite[Section 4.1d]{pomerleano-seidel23}. We assume that this pair is chosen generically so that the Floer complex $\mathit{CF}(\infty) = \mathit{CF}(H_\infty)$ is defined. The constructions of \cite[Section 5.3b]{pomerleano-seidel23} make $\mathit{CF}(\infty)$ into an $L_\infty$ module over $L[1]$. For later use, we denote these $L_\infty$-module operations by 
\begin{align} 
c_{\infty}^{m+1}: CF(w_1)\otimes CF(w_2)\otimes \cdots \otimes CF(w_m) \otimes \mathit{CF}(\infty) \longrightarrow \mathit{CF}(\infty). 
\end{align}

\begin{remark} \label{th:ah-forget-it}
In \cite[Section 5.3b]{pomerleano-seidel23}, the $L_\infty$-algebra acting on $\mathit{CF}(\infty)$ is constructed using quadratic Hamiltonians, meaning that when considering operations parameterized by $\frakC_m$, the Hamiltonians at the interior punctures were also taken to be quadratic. However, the analytical framework from \cite[Section 4.1g]{pomerleano-seidel23}, used to obtain $C^0$-estimates for Floer solutions, also works when the Hamiltonians at interior marked points on the Riemann surface are taken to be linear at infinity (and is in fact slightly easier). 
\end{remark}  

Start with the space which parameterizes distinct points $z_1,z_2,\cdots,z_m$ on the cylinder together with an additional parameter $s^{\circ}$ considered up to translation in the $s$-direction: 
\begin{equation}
\begin{aligned}
& \frakH_m = (\mathit{Conf}_m^{\mathrm{ord}}(\bR \times S^1)\times \bR)/\bR =\mathit{Conf}_m^{\mathrm{ord}}(\bR \times S^1).
\end{aligned}
\end{equation}
We think of the $s^{\circ}$ parameter as picking out a distinguished circle on the cylinder. The spaces $\frakH_m$ are of course the same as $\frakC_m^{\dag}$, but it will be convenient to give them a different name (and the above slightly different description of it), as the $\frakC_m^{\dag}$ will also appear in our argument but play a different role. Let $\overline{\frakH}_m$ be the Fulton-MacPherson style compactification of $\frakH_m$. Let's repeat the standard list of codimension one strata:
\begin{enumerate}[label=(H\arabic*)]  \itemsep .5em
\item \label{item:H1}  Points can collide on the cylinder giving rise to the bubbling of a Fulton-Macpherson screen. 
\item \label{item:H2}  There are strata where a cylinder breaks off. These are isomorphic to  \begin{align} \frakC_{m_{1}} \times \frakH_{m_{2}}  \text{ or }  \frakH_{m_{1}} \times \frakC_{m_{2}}, \quad m_1+m_2=m. \end{align} 
\end{enumerate}

As before, the asymptotic markers over will be chosen to be $S^1$-invariant. Suppose we are given slopes $w_{+}, w_1,\cdots,w_m$. We choose perturbation data over $\frakH_m$ satisfying the analytic requirements of \cite[Section 4.1g]{pomerleano-seidel23} which along the ends reduce to: $(\hat{H}_{w_k} \mathit{dt}, \hat{J}_{w_k})$ near the $z_k$; $(\hat{H}_{w_+} \mathit{dt}, \hat{J}_{w_+})$ near $+\infty$ and $(H_\infty \mathit{dt}, J_{\infty})$ near $-\infty$. These should be $S_m$-equivariant, and satisfy the appropriate forms of consistency near boundary strata; near strata of \ref{item:H1} they should be consistent with our previous choices for Fulton-Macpherson bubbles at the interior marked points, and near strata of type \ref{item:H2} they should be consistent with the data previously chosen over the $\frakC_{m}$ spaces (when the breaking occurs on the left, consistency should be understood in the sense of \cite[Section 4.1h]{pomerleano-seidel23}, involving a rescaling). Given orbits $x_1,\dots,x_k, \dots,x_m$ for $\hat{H}_{w_k}$, $x_{+}$ for $\hat{H}_{w_+}$ and $x_{-}$ for $H_\infty$, we have moduli spaces $\frakH_m(x_{-},x_1,\cdots,x_m,x_{+})$. Counting rigid points in the zero-dimensional moduli spaces gives rise to operations 
\begin{align} 
h^m : \mathit{CF}(w_1)\otimes \mathit{CF}(w_2) \otimes \cdots \otimes \mathit{CF}(w_m)\otimes \mathit{CF}(w_{+}) \longrightarrow \mathit{CF}(\infty) 
\end{align}
of degree $-2m$. The ``$\dag$-variant" of the parameter space records an additional $s$-value, $s^{\dag}$ on the cylinder:
\begin{equation}
\begin{aligned}
& \frakH_m^{\dag} = (\mathit{Conf}_m^{\mathrm{ord}}(\bR \times S^1) \times \mathbb{R} \times \mathbb{R})/\bR= \mathit{Conf}_m^{\mathrm{ord}}(\bR \times S^1) \times \mathbb{R},
\end{aligned}
\end{equation}
where the $\mathbb{R}$-coordinate in the second description is given by $s^\dag - s^\circ.$ Geometrically, it is again convenient to imagine the $s^\dag$ parameter as picking out a distinguished circle on the cylinder.  There is again a compactification of this space into a manifold with corners $\overline{\frakH}_m^{\dag}$, with these codimension one faces: 
\begin{enumerate}[label=(HD\arabic*)]  \itemsep .5em 
\item \label{item:onelastfulton} As usual, points can collide on the cylinder giving rise to the bubbling of a Fulton-MacPherson screen. 

\item \label{item:cylinderbreakdag} There are strata where a cylinder breaks off and both circles corresponding to $s^\dag$, $s^\circ$ remain on the same component. These are isomorphic to  
\begin{align} \frakC_{m_{1}} \times \frakH_{m_{2}}^{\dag}  \text{ or }  \frakH_{m_{1}}^{\dag} \times \frakC_{m_{2}}, \quad m_1+m_2=m. 
\end{align} 

\item  \label{item:dagright} There are strata where $s^\dag - s^\circ \to +\infty$. Here we have cylindrical breaking into two components where the circle determined by  $s^\dag$ is on the rightmost component, and the circle determined by $s^\circ$ ends up on the leftmost component. These strata are of the form:
\begin{align} \label{eq:dagright}  \frakH_{m_{1}} \times \frakC_{m_{2}}^{\dag}, \quad m_1+m_2=m.  \end{align}

\item \label{item:dagleft} Correspondingly, there are the strata where $s^\dag - s^\circ \to -\infty$. Again we have cylindrical breaking into two components but now  the circle determined by $s^\dag$ is on the left-most component, and the circle determined by $s^\circ$ ends up on the right-most component. These strata look like:  
\begin{align}  \label{eq:dagleft}  
\frakC_{m_{1}}^{\dag} \times \frakH_{m_{2}}, \quad m_1+m_2=m.  
\end{align}
\end{enumerate}
We again choose perturbation data over these spaces, subject to same conditions on the end as in the previous case. The consistency requirements near boundary strata of type \ref{item:onelastfulton}, \ref{item:cylinderbreakdag} are also the direct analogues of those used in the previous case. Near strata of type \ref{item:dagright}, we use the data over $\frakC_{m}^{\dag}$ used to define the $L_\infty$ module structure on the telescope complex. Finally, for strata \ref{item:dagleft}, we take the data over $\frakC_{m}$ used to define the $L_\infty$ module structure on $\mathit{CF}(\infty)$, and pull it back to $\frakC_{m}^{\dag}$ along the forgetful map
\begin{align} \label{eq:pullbackalong} 
\frakC_{m}^{\dag} \longrightarrow \frakC_{m}. 
\end{align} 
The resulting moduli spaces $\frakH_m^{\dag}(x_{-},x_1,\cdots,x_m,x_{+})$ give rise to operations 
\begin{align} h^{m,\dag}: \mathit{CF}(w_1)\otimes \mathit{CF}(w_2)\otimes \cdots \otimes \mathit{CF}(w_m)\otimes \mathit{CF}(w_{+}) \longrightarrow \mathit{CF}(\infty), \end{align}
which now have degree $-2m-1$. 

\begin{prop} \label{thm:takingtheslopetoinfty} 
Let $g_q$ be any Maurer-Cartan element in $L$. The deformation of the telescope $C_g$ given by $g_q$ is quasi-isomorphic to  $\mathit{CF}(\infty)_g$, which is the deformation of $\mathit{CF}(\infty)$ by the same Maurer-Cartan element. 
\end{prop} 

\begin{proof} We define a map 
\begin{equation} \label{eq:hgdef}
\begin{aligned}
& h_g: C_g \longrightarrow \mathit{CF}(\infty)_g, \\
& h_g(x) = \sum_{m \geq 0} (1/m!)\, h^{m}(g_q^{\otimes m},x), \quad x\in \mathit{CF}(w) \subset C \\[-.5em]
& h_g(\eta x) = \sum_{m \geq 0} (1/m!)\, h^{m,\dag}(g_q^{\otimes m},x), \quad \eta x \in  \eta \mathit{CF}(w) \subset C.
\end{aligned}
\end{equation}
The fact that $h_g$ is a cochain map follows as usual from analysis of the boundaries of dimension one moduli spaces. If $x_{+} \in \mathit{CF}(w) \subset C$, then we consider dimension one moduli spaces of the form $\overline{\frakH}_m(x_{-},x_1,\cdots,x_m,x_{+}).$ \begin{itemize} \itemsep .5 em \item As usual, because we are inserting a Maurer-Cartan element, (then after summing over all $m$) the boundary strata coming from \ref{item:H1} together with Floer differentials at $z_1,\cdots,z_m$ contribute zero. \item Meanwhile the boundary strata of type \ref{item:H2} together with Floer cylinder breaking at $\pm \infty$ give rise to an equation: 
\begin{equation}\label{eq:chrelation}
\sum_{j_1+j_2=m} \frac{1}{j_1!j_2!}  h^{j_1}(g_q^{\otimes j_{1}},c^{j_2+1}(g_q^{\otimes j_2}, x)) = \sum_{j_1+j_2=m}  \frac{1}{j_1!j_2!} c_{\infty}^{j_1+1}(g_q^{\otimes j_{1}},h^{j_2}(g_q^{\otimes j_2},x)). 
\end{equation}  
After summing over all $m$, \eqref{eq:chrelation} implies that $h_g \circ d_g (x) = d_{g,\infty} \circ h_g(x)$, where $d_{g,\infty}$ denotes the deformed differential on $CF(\infty)_g$.
\end{itemize}
If $\eta x \in \eta \mathit{CF}(w) \subset C$, then we consider dimension one moduli spaces $\overline{\frakH}_m^{\dag}(x_{-},x_1,\cdots,x_m,x_{+}).$ \begin{itemize} \itemsep .5 em \item As before,  boundaries of type \ref{item:onelastfulton} together with Floer breaking at interior punctures contribute zero because of the Maurer-Cartan equation. \item Meanwhile, this time the contributions of the strata \ref{item:cylinderbreakdag}-\ref{item:dagleft} (along with Floer breaking at $\pm \infty$) give rise to the equation
\begin{equation} \label{eq:chdagrelation}
\begin{aligned} 
&  \sum_{j_1+j_2=m} \frac{1}{j_1!j_2!}  h^{j_1}(g_q^{\otimes j_{1}},c^{j_2+1,\dag}(g_q^{\otimes j_2}, x))+ \sum_{j_1+j_2=m} \frac{1}{j_1!j_2!}  h^{j_1,\dag}(g_q^{\otimes j_{1}},c^{j_2+1}(g_q^{\otimes j_2}, x)) \\
& = \sum_{j_1+j_2=m} \frac{1}{j_1!j_2!}  c_\infty^{j_1+1}(g_q^{\otimes j_{1}},h^{j_2,\dag}(g_q^{\otimes j_2}, x))- \frac{1}{m!}  h^{m}(g_q^{\otimes m}, x).
\end{aligned}
\end{equation}
The fourth term in \eqref{eq:chdagrelation} comes from the contributions of \ref{item:dagleft}. Here we have used the fact that along these strata, the perturbation data on the cylinders carrying $s^{\dag}$ is pulled back along \eqref{eq:pullbackalong}. As a consequence, all of these strata contribute zero except for the case $m_1=0,m_2=m$ of \eqref{eq:dagleft}. After summing over all $m$ and taking into account the algebraically inserted  $-id: \eta CF (w) \to CF (w)$ component of the differential on the telescope complex, \eqref{eq:chdagrelation} implies that $h_g \circ d_g (\eta x) = d_{g,\infty} \circ h_g(\eta x)$. 
\end{itemize}

Finally, to show that $h_g$ is a quasi-isomorphism, note that it is a filtered deformation of a standard continuation map
\begin{equation} 
\begin{aligned}
& h_{q=0}: C \to \mathit{CF}(\infty), \\
& h_{q=0}(x) =  h^{0}(x), \quad x \in \mathit{CF}(w) \subset C \\[-.5em]
& h_{q=0}(\eta x) = h^{0,\dag}(x), \quad \eta x \in  \eta \mathit{CF}(w) \subset C.
\end{aligned}
\end{equation} from the telescope complex to the Floer complex of the quadratic Hamiltonian $H_\infty$, which is well-known to be a quasi-isomorphism.
\end{proof}

\begin{thm} \label{th:its-the-same}
The complex $\mathit{CF}(\infty)_g$ defined by the Maurer-Cartan element \eqref{eq:geometric-mc}, and the complex $C_{q}$ from \S \ref{sec:equivariantdiff}, are quasi-isomorphic.  
\end{thm}

\begin{proof} 
Combine Theorem \ref{th:pullout} and Proposition \ref{thm:takingtheslopetoinfty}, the latter being applied to the special case where $g_q$ is \eqref{eq:geometric-mc}. 
\end{proof}

\subsection{Rotating asymptotic markers\label{subsec:rotate}}
The construction from Sections \ref{sec:extractionspaces}--\ref{subsection:toinfinity} has a generalization, which differs in the choices of asymptotic markers (and therefore in the parametrization of the ends, which affects the Floer-theoretic data one puts on the Riemann surfaces). This does not affect the abstract topology of the compactified parameter spaces, but it changes how boundary strata are identified with products of lower-dimensional spaces.

Namely, fix $U \in \bZ$. In the situation of the spaces $\frakC_m$, we change the asymptotic marker at any marked point $z_k = (s_k,t_k) \in \bR \times S^1$ to point in a $t_k$-dependent direction, namely $-\!\exp(2\pi i U t_k)$; in other words, it rotates $U$ times as $t_k$ moves around the circle (the previously used $S^1$-invariant markers are the case $U = 0$). The tubular end around the corresponding puncture is of the form
\begin{equation} \label{eq:rotated-end}
\begin{aligned}
& [0,\infty) \times S^1 \longrightarrow (\bR \times S^1) \setminus \{(s_1,t_1),\dots,(s_m,t_m)\}, 
\\
& (s,t) \longmapsto (s_k+i t_k) - \rho_k \exp(-2\pi (s+it + iUt_k)).
\end{aligned}
\end{equation}
Here, we have written the target space as $\bR \times S^1 = \bC/i\bZ$, and $\rho_k > 0$ is a constant that can (subject to being sufficiently small) be chosen freely. The other asymptotic markers and ends, near $\pm\infty$, remain unchanged. For any fixed $U$, these choices are consistent with gluing together several cylinders. The convention also affects how one thinks of limits in $\overline{\frakC}_r$. Namely, when several marked points coalesce at some $(s_*,t_*) \in \bR \times S^1$, bubbling off into a limiting component which is punctured plane, that plane should be rotated by $-2\pi U t_*$, so that it is compatible with the gluing process using \eqref{eq:rotated-end}. The same can be applied to $\frakC_m^\dag$. The outcome is that there are infinitely many different structures $c_C^U$ of an $L_\infty$-module on $C$ (all using the same $L_\infty$-algebra structure on $L[1]$), specializing to our previous $c_C$ for $U = 0$. Given a general Maurer-Cartan element $g$ in $L$, one can therefore use it to deform $C$ in different ways, depending on the choice of $U$. Let's temporarily denote these deformations by $C^U_g$.

The same change of asymptotic markers can be applied to the extraction spaces. Hence, Theorem \ref{th:pullout} applies to any value of $U$. The consequence is that for the particular choice of Maurer-Cartan element from \eqref{eq:geometric-mc}, all the complexes $C^U_g$ are isomorphic to each other. Similarly, the change to infinite slope in Section \ref{subsection:toinfinity} goes through for all $U$. This in particular covers the case $U = 1$, where the deformed Floer complex was called $\mathit{CF}_q^{\mathrm{diag}}$ in \cite[Definition 5.3.7(i)]{pomerleano-seidel23}.

\subsection{The equivariant version}
Previously (Section \ref{subsec:angle-decorated}) we constructed the $S^1$-equivariant deformed complex $C_{q,u}$ using the spaces \eqref{eq:d-space}. One can equivalently use ordered point collections, parametrized by
\begin{equation} \label{eq:tilde-frakd}
\tilde{\frakD}_m^l = \big(\Theta^l \times (\bR \times S^1)^m\big)/\bR,
\end{equation}
as well as a version $\tilde{\frakD}_m^{l,\dag}$; and then divide the resulting point counts by $(m!)$. In a similar vein, one can use the spaces
\begin{equation} \label{eq:equi-c}
\frakC_m^l = \big(\Theta^l \times \mathit{Conf}_m^{\mathrm{ord}}(\bR \times S^1) \big)/\bR
\end{equation}
and the analogously defined $\frakC_m^{l,\dag}$ to make $C[[u]]$, with its equivariant differential, into an $L_\infty$-module over $L[1]$; this is $u$-deformed version of the structure from Section \ref{subsec:punctures}. In particular, given any Maurer-Cartan element $g$, one gets a differential on $C[[u,q]]$. Let's denote the resulting complex by $C_{u,g}$. Finally, we have the construction from \cite[Section 5.3d]{pomerleano-seidel23}, which uses the same spaces \eqref{eq:equi-c} and Maurer-Cartan elements, but applies them to infinite slope Hamiltonians at $\pm\infty$ (again, the convention in \cite{pomerleano-seidel23} is that the Hamiltonians at the punctures at $z_1,\dots,z_m \in \bR \times S^1$ also have infinite slope; but as pointed out in Remark \ref{th:ah-forget-it}, it is unproblematic to adopt finite slopes instead, since those punctures are only used to insert the Maurer-Cartan element). Let's denote the resulting complex by $\mathit{CF}(\infty)_{u,g}$. The equivariant analogue of Theorem \ref{th:its-the-same} is:

\begin{thm} \label{th:its-equivariantly-the-same}
The complex $\mathit{CF}(\infty)_{u,g}$, where $g_q$ is the Maurer-Cartan element \eqref{eq:geometric-mc}, is quasi-isomorphic to $C_{u,q}$.
\end{thm}

We will not give the entire proof, since it overall follows the previous pattern; but we do want to explain one component, namely how to set up the equivariant version of the extraction spaces (since that is a prototypical example of modifying that construction to include additional structure). We consider cylinders equipped with points $z_1,\dots,z_m \in \bR \times S^1$, a one-form $\alpha = a\,\mathit{dz}$ with $a \geq 1$, and angle-decorated circles \eqref{eq:delta-l}, up to translation in $\bR$-direction. The parameter space is therefore
\begin{equation} \label{eq:ax}
\frakAX_{m,l} = \big( \Theta^l \times (\bR \times S^1)^m \times [1,\infty) \big)/\bR.
\end{equation}
For $l = 0$ this reduces to $\EX_m$. The space \eqref{eq:ax} again admits a compactification to a manifold with generalized corners. When the scale variable $a$ stays bounded, the cylinders can split into pieces, all carrying the same scale, and each of which contains some number of the original $l$ angle-decorated circles. In other words, the generalization of \eqref{eq:p-split} is that we have boundary strata in $\overline{\frakAX}_{m,l}$ of the form
\begin{equation} \label{eq:ax-split}
\frakAX_{m_1,l_1} \times_{[1,\infty)} \frakAX_{m_2,l_2} \times_{[1,\infty)} \cdots 
\times_{[1,\infty)} \frakAX_{m_p,l_p}
\end{equation}
for $m_1 + \cdots m_p = m$ (and an additional choice of partition), $l_1 + \cdots + l_p = l$. As usual with $S^1$-equivariant constructions, the limiting point in \eqref{eq:ax-split} is obtained by taking the naive limit of a sequence of cylinders, and then rotating each component of that limit according to the total angle of all its components to the right (Figure \ref{fig:twisted-gluing}). The analogue of \eqref{eq:p-bubble}, when $a \rightarrow \infty$, is
\begin{equation} \label{eq:p2-bubble}
\frakAX_{T_1,\dots,T_r} = \prod_j \MC_{T_j} \times 
(\Theta^l \times \mathit{Conf}_{r}^{\mathrm{ord}}(\bR \times S^1))/\bR,
\end{equation}
meaning that the angle-decorated circles stay on the ``main'' cylindrical component. With that in mind, let's focus on the codimension one boundary strata of $\overline{\frakAX}_{m,l}$, which are of the following kinds:
\begin{enumerate}[label=(AX\arabic*)] \itemsep.5em
\item \label{item:ax-1}
The boundary stratum $\partial_{a=1}\frakAX_{m,l}$ is identified with \eqref{eq:tilde-frakd}, which is how the construction relates to $C_{u,q}$.

\item \label{item:ac-split2}
Next we have the $p = 2$ case of \eqref{eq:ax-split}, which is 
\begin{equation} \label{eq:ax-scale}
\frakAX_{m_1,l_1} \times_{[1,\infty)} \frakAX_{m_2,l_2}. 
\end{equation}
There is one such stratum for each pair consisting of: a decomposition of $\{1,\dots,m\}$ into two subsets of size $m_1,m_2$; and $l_1+l_2=l$ ($m_k=0$ is allowed, provided that the corresponding $l_k>0$).

\item \label{item:ax-bubble}
The case of \eqref{eq:p2-bubble} in which each tree has a single vertex, which yields 
\begin{equation} \label{eq:leaf-or-root2}
\prod_{j=1}^r \MC_{m_j} \times (\Theta^l \times \mathit{Conf}_{r}^{\mathrm{ord}}(\bR \times S^1) \big)/\bR.
\end{equation} 
The last factor in the product is the space used to construct the equivariant $L_\infty$-module structure, so this is how the relation with $C_{u,g}$ arises.

\item \label{item:ax-equal}
Strata $\partial_{\sigma_i=\sigma_{i+1}}\frakAX_{m,l}$ when two successive angle-decorated circles coincide. As usual in this kind of equivariant construction, one arranges that the Floer-theoretic data only depend on $\theta_i + \theta_{i+1}$.
\end{enumerate}
Parts \ref{item:ax-1}--\ref{item:ax-bubble} are as in the non-equivariant setup, while the only new ingredient \ref{item:ax-equal} will ultimately not contribute, due to the specific choices made there. The same applies to the scale-ordered versions of these spaces, defined exactly as in \eqref{eq:o-space}. The outcome is a version of \eqref{eq:o-equation}, where one simply replaces each ingredient (the differentials $d$, the $L_\infty$-module structure $c$, and the $o$ operations) with their equivariant counterparts. The same discussion applies to 
\begin{equation}
\frakAX_{m,l}^{\dag} = \Theta^l \times (\bR \times S^1)^m \times [1,\infty);
\end{equation}
the only notable change being that in the counterpart of the splitting \eqref{eq:ax-scale}, exactly one of the two factors is an $\frakAX^\dag$ space. 
%
%

Having given an overview of the construction underlying Theorem \ref{th:its-equivariantly-the-same}, we want to state one more property.

\begin{proposition}
The cohomology level isomorphism from Theorem \ref{th:its-equivariantly-the-same} identifies the connection from Section \ref{subsec:define-connection} with that from \cite[Section 5.3f]{pomerleano-seidel23}.
\end{proposition}

We will not give any details here, but the strategy is straightforward: both definitions use closely related parameter spaces, one having marked points on the cylinders, and the other punctures (where either the Maurer-Cartan element, or its $q$-derivative, is inserted). To relate them, one takes the construction of extraction spaces, and adds the same geometric data as in the definition of connection.


\end{document}